\documentclass[a4paper,12pt]{article}
\usepackage{currfile,datetime}
\usepackage{cite}
\usepackage{amsthm,amsmath,amssymb,bbm,geometry,epsfig,listings,
paralist,
enumerate,comment,nicefrac,verbatim,multirow,xcolor}
\usepackage{mathrsfs}

\usepackage[colorlinks,linkcolor=red,citecolor=green]{hyperref}

\usepackage[capitalise,compress]{cleveref} 

\crefname{equation}{}{}

\newtheorem{lemma}{Lemma}[section]

\newtheorem{proposition}[lemma]{Proposition}
\newtheorem{theorem}[lemma]{Theorem}

\newtheorem{prop}[lemma]{Proposition}
\newtheorem{corollary}[lemma]{Corollary}

\newtheorem{setting}[lemma]{Setting}

\crefname{subsection}{Subsection}{Subsections}
\crefname{enumi}{item}{items}
\creflabelformat{enumi}{(#2\textup{#1}#3)}

\DeclareFontEncoding{LS1}{}{}
\DeclareFontSubstitution{LS1}{stix}{m}{n}
\DeclareMathAlphabet{\mathscr}{LS1}{stixscr}{m}{n}

\newcommand{\smallsum}{{\textstyle\sum}}
\newcommand{\1}{\ensuremath{\mathbbm{1}}}
\providecommand{\N}{{\ensuremath{\mathbbm{N}}}}
\providecommand{\Z}{{\ensuremath{\mathbbm{Z}}}}
\providecommand{\R}{{\ensuremath{\mathbbm{R}}}}

\providecommand{\E}{{\ensuremath{\mathbbm{E}}}}
\providecommand{\var}{{\ensuremath{\operatorname{\mathbb{V}ar}}}}
\renewcommand{\P}{{\ensuremath{\mathbbm{P}}}}
\newcommand{\F}{{\ensuremath{\mathbbm{F}}}}
\newcommand{\trace}{\mathrm{tr}}
\newcommand{\Hess}{\operatorname{Hess}}
\newcommand{\uniform}{\ensuremath{\mathcal{R}}}
\newcommand{\unif}{\ensuremath{\mathfrak{r}}}
\newcommand{\Exp}[1]{ \E \! \left[ #1 \right]}

\newcommand{\vastl}[2]{\left#2 \rule{0pt}{#1}\kern-1ex\right.}
\newcommand{\vastr}[2]{\left. \rule{0pt}{#1}\kern-1ex\right#2}
\newcommand{\funcF}{F}
\newcommand{\LipConstF}{L}
\newcommand{\fwpr}{\ensuremath{{Y}}} 
\newcommand{\sppr}{\ensuremath{Y}}
\newcommand{\exactpr}{\ensuremath{X}}
\newcommand{\bigV}{U}
\newcommand{\funcG}{g}
\newcommand{\smallV}{u}
\newcommand{\lyaV}{\varphi}

\newcommand{\smallU}{u}

\newcommand{\defeq}{\curvearrowleft}
\newcommand{\sigmaAlgebra}{{\mathfrak{S}}}
\newcommand{\induct}{\dashrightarrow}

\newcommand{\vertiii}[1]{{\left\vert\kern-0.25ex\left\vert\kern-0.25ex\left\vert #1 
		\right\vert\kern-0.25ex\right\vert\kern-0.25ex\right\vert}}
\newcommand{\tripleNorm}[1]{\vertiii#1} 

\newcommand{\lemdreisechst}{\alpha}
\newcommand{\lemdreisechsT}{\beta}
\newcommand{\lemdreisechsc}{c}

\newcommand{\size}[1]{\delta}
\newcommand{\rdownni}[2]{{\left \llcorner#1\right \lrcorner}}

\newcommand{\stTau}[3]{\boldsymbol{\tau}}

\newcommand{\smallF}{f}
\newcommand{\cF}{\mathcal{F}}

\renewcommand{\epsilon}{\varepsilon}
\newcommand{\growrate}{\rho}
\newcommand{\cO}{\mathcal{O}}

\newcommand{\norm}[1]{ \left\| #1 \right\| }
\newcommand{\cV}{\mathcal{V}}
\newcommand{\lyaPsi}{\psi}
\newcommand{\FE}{\mathfrak{C}}

\title{
Multilevel Picard approximations for\\
high-dimensional semilinear second-order\\
PDEs with Lipschitz nonlinearities}

\author{Martin Hutzenthaler$^{1}$, Arnulf Jentzen$^{2}$, \\
Thomas Kruse$^{3}$, and Tuan Anh Nguyen$^{4}$\bigskip\\
\small{$^1$ Faculty of Mathematics, University of Duisburg-Essen,}\\
\small{Essen, Germany; e-mail: \texttt{martin.hutzenthaler}\textcircled{\texttt{a}}\texttt{uni-due.de}}\\
\small{$^2$ Faculty of Mathematics and Computer Science, University of M\"unster,}\\
\small{M\"unster, Germany; e-mail: \texttt{ajentzen}\textcircled{\texttt{a}}\texttt{uni-muenster.de}}\\
\small{$^3$ Institute of Mathematics, University of Gie{\ss}en,}\\
\small{Gie{\ss}en, Germany; e-mail: \texttt{thomas.kruse}\textcircled{\texttt{a}}\texttt{math.uni-giessen.de}}\\
\small{$^4$ Faculty of Mathematics, University of Duisburg-Essen,}\\
\small{Essen, Germany; e-mail: \texttt{tuan.nguyen}\textcircled{\texttt{a}}\texttt{uni-due.de}}
}

\allowdisplaybreaks
\begin{document}

\maketitle

\begin{abstract}
The recently introduced full-history recursive multilevel Picard (MLP) approximation methods have turned out to be quite successful in the numerical approximation of solutions of high-dimensional nonlinear PDEs. In particular, there are mathematical convergence results in the literature which prove that MLP approximation methods do overcome the curse of dimensionality in the numerical approximation of nonlinear second-order PDEs in the sense that the number of computational operations of the proposed MLP approximation method grows at most polynomially in both the reciprocal $1/\epsilon$ of the prescribed approximation accuracy $\epsilon>0$ and the PDE dimension $d\in \N=\{1,2,3, \ldots\}$. However, 
in each of the convergence results for MLP approximation methods in the literature it is assumed that the coefficient functions in front of the second-order differential operator are affine linear. In particular, 
until today there is no result in the scientific literature which proves that any semilinear second-order PDE with a general time horizon and a non affine linear coefficient function in front of the second-order differential operator can be approximated without the curse of dimensionality. It is the key contribution of this article to overcome this obstacle and to propose and analyze a new type of MLP approximation method for semilinear second-order PDEs with possibly nonlinear coefficient functions in front of the second-order differential operators. In particular, the main result of this article proves that this new MLP approximation method does indeed overcome the curse of dimensionality in the numerical approximation of semilinear second-order PDEs.
\end{abstract}

\pagebreak

\tableofcontents
\section{Introduction}
It is a very challenging task in applied mathematics to design and analyze approximation algorithms for high-dimensional nonlinear partial differential equations (PDEs) and this topic of research has been very intensively studied in the scientific literature in the last two decades. Especially, there are two types of approximation methods which have turned out to be quite successful in the numerical approximation of solutions of high-dimensional nonlinear second-order PDEs, namely, (I) deep learning based approximation methods for PDEs (cf., e.g., \cite{uchiyama1993solving,MeadeFernandez1994,Lagaris1998ArtificialNN,LiLuo2003,
EHanJentzen2017CMStat,BeckEJentzen2017,EYu2018,
FujiiTakahashiTakahashi2017,HanJentzenE2017,HenryLabordere2017,
Farahmand2017DeepRL,Raissi2018,
SirignanoSpiliopoulos2017,BeckerCheriditoJentzen2018arXiv,Becketal2018,
Magill2018NeuralNT,LongLuMaDong2018,HanLong2018,Berg2018AUD,
PhamWarin2019,LyeMishraRay2019, GoudenegeMolent2019,BeckBeckerCheridito2019,
JacquierOumgari2019,HurePhamWarin2019,ChanMikaelWarin2019,
BeckerCheriditoJentzen2019,Dockhorn2019,chen2019deep,grohs2019deep}) and (II) full-history recursive multilevel Picard approximation methods for PDEs (cf.\ \cite{EHutzenthalerJentzenKruse2016,HJKNW2018,EHutzenthalerJentzenKruse2017,HutzenthalerKruse2017,hutzenthaler2019overcoming,beck2019overcoming, giles20019generalised,hjk2019overcoming,becker2020numerical}; in the following we abbreviate \textit{full-history recursive multilevel Picard} as \textit{MLP}). Deep learning based approximation methods for PDEs are, roughly speaking, based on the idea to (Ia) approximate the PDE problem under consideration through a stochastic optimization problem involving deep neural networks as approximations for the solution or the derivatives of the solution of the PDE under consideration and to (Ib) apply stochastic gradient descent methods to approximately solve the resulting stochastic optimization problem. Even though there are a number of encouraging simulation results for deep learning based approximation methods for PDEs in the scientific literature, there are only partial mathematical error analyses in the scientific literature which only partly explain why deep learning based approximation methods for PDEs can approximately solve high-dimensional PDEs (cf., e.g., \cite{hutzenthaler2019proof, HanLong2018,SirignanoSpiliopoulos2017,BernerGrohsJentzen2018, ElbraechterSchwab2018,GrohsWurstemberger2018,
JentzenSalimovaWelti2018,KutyniokPeterseb2019,ReisingerZhang2019,
GrohsHornungJentzen2019}). In particular, there are no results in the scientific literature which prove that deep learning based approximation methods for PDEs overcome the curse of dimensionality in the sense that the number of computational operations of any deep learning based approximation method grows at most polynomially in both the reciprocal of the prescribed approximation accuracy and the PDE dimension. MLP approximation methods are, roughly speaking, based on the idea to (IIa) reformulate the PDE under consideration as a stochastic fixed point problem with the PDE solution being the fixed point of the stochastic fixed point equation, to (IIb) approximate the fixed point through Banach fixed point iterates (which are also referred to as Picard iterates in the context of integral fixed point equations), and to (IIc) approximate the resulting Banach fixed point iterates through suitable full-history recursive multilevel Monte Carlo approximations. In the case of MLP approximation methods there are both encouraging numerical simulation results (see \cite{becker2020numerical, EHutzenthalerJentzenKruse2017}) and rigorous mathematical results which prove that MLP approximation methods do indeed overcome the curse of dimensionality in the numerical approximation of nonlinear second-order PDEs (see \cite{EHutzenthalerJentzenKruse2016,HJKNW2018,HutzenthalerKruse2017,hutzenthaler2019overcoming,beck2019overcoming, giles20019generalised,hjk2019overcoming,beck2020overcomingelliptic}). However, in each of the convergence results for MLP approximation methods in the scientific literature it is assumed that the coefficient functions in front of the second-order differential operator are affine linear. In particular, until today there is no result in the scientific literature which proves that any semilinear second-order PDE with a general time horizon and a non affine linear coefficient function in front of the second-order differential operator can be approximated without the curse of dimensionality.

It is precisely the subject of this article to overcome this obstacle and to propose and analyze a new type of MLP approximation method for semilinear second-order PDEs with possibly nonlinear coefficient functions in front of the second-order differential operators. In particular, the main result of this article, \cref{m01_x} in \cref{sec:main_result} below, proves that this new MLP approximation method overcomes the curse of dimensionality in the numerical approximation of semilinear second-order PDEs in the sense that the number of computational operations of the proposed MLP approximation method grows at most polynomially in both the reciprocal $1/\epsilon$ of the prescribed approximation accuracy $\epsilon\in (0,\infty)$ and the PDE dimension $d\in \N=\{1,2,3,\ldots\}$. To briefly outline the contribution of this work within this introductory section, we now present in the following result, \cref{m01b} below, a special case of \cref{m01}.

\begin{theorem}
\label{m01b}
Let
$c,T\in  [0,\infty)$,
$f\in C( \R,\R)$,
for every $d\in\N$ let $u_d\in C^{1,2}([0,T]\times\R^d,\R)$,
$\mu_d=
(\mu_{d,i})_{i\in\{1,2,\ldots,d\}}
\in C( \R^{d} ,\R^{d})$, $\sigma_d=(\sigma_{d,i,j})_{i,j\in \{1,2,\ldots,d\}}\in C( \R^{d},\R^{d\times d})$
satisfy for all $t\in[0,T]$,
 $x=(x_1,x_2,\ldots,x_d)$, $y=(y_1,y_2,\ldots,y_d)\in\R^d$
  that
\begin{equation}\label{eq:Lipschitz_f_intro}
|f(x_1)-f(y_1)|\leq c|x_1-y_1|,\quad |u_d(t,x)|^2+ 
\max_{ i, j \in \{ 1, 2, \ldots, d \} } ( | \mu_{ d, i }( 0 ) | + | \sigma_{ d, i, j }( 0 ) |)
\leq c \Bigl[d^c+\textstyle\sum\limits_{i=1}^{d}\displaystyle |x_i|^2\Bigr],
\end{equation}
\begin{equation}\label{eq:Lipschitz_coeffs_intro}
|u_d(T,x)-u_d(T,y)|^2+
\textstyle \sum\limits_{i=1}^{d} \displaystyle
 |\mu_{d,i}(x)-\mu_{d,i}(y)|^2+
\textstyle \sum\limits_{i,j=1}^{d}\displaystyle
|\sigma_{d,i,j}(x)-\sigma_{d,i,j}(y)|^2
\leq c\Big[
\textstyle \sum\limits_{i=1}^{d} \displaystyle |x_i-y_i|^2
\Big],
\end{equation}
\begin{equation}\label{eq:PDE_intro}
\text{and}\quad 
(\tfrac{\partial }{\partial t}u_d)(t,x)+
(\tfrac{ \partial }{ \partial x } u_d )( t, x ) \, \mu_d(x) 
+ \tfrac{1}{2}
\trace\big(\sigma_d(x)[\sigma_d(x)]^* (\Hess_xu)(t,x)\big)=
-f(u_d(t,x)),
\end{equation}
let $(\Omega,\mathcal{F},\P)$ be a probability space,
let 
$  \Theta = \bigcup_{ n \in \N }\! \Z^n$,
let $\unif^\theta\colon \Omega\to[0,1]$, $\theta\in \Theta$, be i.i.d.\ random variables\footnote{Note that the expression \emph{i.i.d.} is an abbreviation for the expression \emph{independent and identically distributed}.},
 let $W^{d,\theta}\colon [0,T]\times\Omega \to \R^{d}$,
$d\in\N$, $\theta\in\Theta$, be i.i.d.\ standard Brownian motions,
assume for all $t\in (0,1)$ that $\P(\unif^{0}\le t)=t$,
assume that $( \unif^{ \theta } )_{ \theta \in\Theta}$ and $( W^{ d, \theta } )_{ (d,\theta) \in \N\times\Theta }$ are independent,
for every $d,N\in\N$,
$\theta\in\Theta$, $x\in\R^d$,
$t\in[0,T)$
let 
$
Y^{d,N,\theta,x}_{t}=
(Y^{d,N,\theta,x}_{t,s})_{s\in[t,T]}\colon[t,T]\times\Omega\to\R^d$ 
 satisfy for all $n\in\{0,1,\ldots,N\}$,
 $s\in[\frac{nT}{N},\frac{(n+1)T}{N}]\cap[t,T]$ that $Y_{t,t}^{d,N,\theta,x}=x$ and
\begin{equation}\label{eq:euler_intro}
\begin{split}
&Y_{t,s}^{d,N,\theta,x} -
Y_{t,\max\left\{t,nT/N\right\}}^{d,N,\theta,x}\\
&=
\mu_d\big(Y_{t,\max\left\{t,nT/N\right\}}^{d,N,\theta,x}\big)\big(s-\max\!\big\{t,\tfrac{nT}{N}\big\}\big)+
\sigma_d\big(Y_{t,\max\left\{t,nT/N\right\}}^{d,N,\theta,x}\big)\big(W^{d,\theta}_{s} -W^{d,\theta}_{\max\{t,nT/N\}} \big),
\end{split}
\end{equation}
let
$ 
  {\bigV}_{ n,M}^{d,\theta} \colon [0, T] \times \R^d \times \Omega \to \R
$, 
$d,n,M\in\Z$, $\theta\in\Theta$, satisfy
for all $d,M \in \N$, $n\in \N_0$, $\theta\in\Theta $,
$ t \in [0,T]$, $x\in\R^d $
that 
\begin{align}\label{eq:method}
  &{\bigV}_{n,M}^{d,\theta}(t,x)
=
  \tfrac{ \1_{ \N }( n )}{M^n}
 \textstyle\sum\limits_{i=1}^{M^n} \displaystyle
      u_d\big(T,\sppr^{d,M^M,(\theta,0,-i),x}_{t,T}\big)
 \\
 \nonumber
&  +
  \textstyle\sum\limits_{\ell=0}^{n-1} \displaystyle \left[ \tfrac{(T-t)}{M^{n-\ell}}
   \textstyle\sum\limits_{i=1}^{M^{n-\ell}}\displaystyle
      \big(\smallF\circ {\bigV}_{\ell,M}^{d,(\theta,\ell,i)}-\1_{\N}(\ell)\,\smallF\circ{\bigV}_{\ell-1,M}^{d,(\theta,-\ell,i)}\big)
\big(t+(T-t)\unif^{(\theta,\ell,i)},\sppr_{t,t+(T-t)\unif^{(\theta,\ell,i)}}^{d,M^M,(\theta,\ell,i),x}\big)
    \right]\!,
\end{align}
and for every $d,n,M\in \N$ let $\FE_{d,n,M}\in\N$ 
be the number of function evaluations of $f$, $u_d(T,\cdot)$, $\mu_d$, and $\sigma_d$ and
the number of realizations of scalar random variables which
are used to compute one realization of
${\bigV}_{n,M}^{d,0}(0,0)\colon\Omega\to\R^d$ (cf.\ \eqref{eq:comp_effort_cor} for a precise definition).
Then there exist $\mathfrak{c}\in \R$ and $\mathsf{n}\colon \N\times (0,1] \to \N$ such that
for all $d\in\N$, $\epsilon\in(0,1]$
it holds that $\big(\E\bigl[|\smallU_d(0,0)-{\bigV}_{\mathsf{n}(d,\epsilon),\mathsf{n}(d,\epsilon)}^{d,0}(0,0)|^2\big]\big)^{ 1/2 }\leq \epsilon$ and $\FE_{d,\mathsf{n}(d,\epsilon),\mathsf{n}(d,\epsilon)}
\leq \mathfrak{c}d^\mathfrak{c} \epsilon^{-5}$. 
\end{theorem}
\cref{m01b} follows from \cref{m01c}. \cref{m01c}, in turn, follows from \cref{m01_x} (see \cref{sec:main_result} for details).
In the following we add a few comments concerning the mathematical objects appearing in \cref{m01b} above. The real number $T\in (0,\infty)$ in \cref{m01b} above specifies the time horizon for the PDEs (see \cref{eq:PDE_intro}) whose solutions we intend to approximate in \cref{m01b} above. The real number $c\in (0,\infty)$ in \cref{m01b} above is a constant which we employ to formulate several regularity hypotheses in \cref{m01b} above. The function $f\colon \R \to \R$ in \cref{m01b} above describes the nonlinearity for the PDEs (see \cref{eq:PDE_intro}) whose solutions we intend to approximate in \cref{m01b} above. The functions $u_d\colon [0,T]\times \R^d\to \R$, $d\in \N$, in \cref{m01b} above describe the PDE solutions which we intend to approximate in \cref{m01b} above. The functions $\mu_d\colon \R^d\to \R^d$, $d\in \N$, in \cref{m01b} above describe the coefficient functions in front of the first-order derivative terms in the PDEs (see \cref{eq:PDE_intro}) whose solutions we intend to approximate in \cref{m01b} above. The functions $\sigma_d\colon \R^d\to \R^{d\times d}$, $d\in \N$, in \cref{m01b} above describe the coefficient functions in front of the second-order derivative terms in the PDEs (see \cref{eq:PDE_intro}) whose solutions we intend to approximate in \cref{m01b} above. In \cref{eq:Lipschitz_f_intro} and \cref{eq:Lipschitz_coeffs_intro} we formulate the Lipschitz hypotheses which we employ in \cref{m01b} above. In \cref{eq:PDE_intro} we specify the PDEs whose solutions we intend to approximate in \cref{m01b} above. The probability space $(\Omega,\mathcal{F},\P)$ in \cref{m01b} above is the probability space on which we introduce the stochastic MLP approximations which we employ to approximate the solutions $u_d\colon [0,T]\times \R^d\to \R$, $d\in \N$, of the PDEs in \cref{eq:PDE_intro}. 
The set $  \Theta = \bigcup_{ n \in \N }\!\Z^n$ in \cref{m01b} above is used as an index set to introduce sufficiently many independent random variables on this index set.
The functions $\unif^\theta\colon \Omega\to[0,1]$, $\theta\in \Theta$, describe on $[0,1]$ continuously uniformly distributed independent random variables which we use as random input sources for the MLP approximations which we employ in \cref{m01b} above to approximately compute the solutions $u_d\colon [0,T]\times \R^d\to \R$, $d\in \N$, of the PDEs in \cref{eq:PDE_intro}. 
The functions $W^{d,\theta}\colon [0,T]\times\Omega \to \R^{d}$,
$d\in\N$, $\theta\in\Theta$, describe independent standard Brownian motions which we use as random input sources for the MLP approximations which we employ in \cref{m01b} above to approximately compute the solutions $u_d\colon [0,T]\times \R^d\to \R$, $d\in \N$, of the PDEs in \cref{eq:PDE_intro}. The functions 
$
Y^{d,N,\theta,x}_{t}\colon[t,T]\times\Omega\to\R^d$,
$d,N\in\N$,
$\theta\in\Theta$, $x\in\R^d$,
$t\in[0,T)$, in \cref{eq:euler_intro} above
describe Euler-Mayurama approximations which we use in the MLP approximations in \cref{eq:method} in \cref{m01b} above as discretizations of the underlying It\^o processes associated to the linear parts of the PDEs in \cref{eq:PDE_intro}. The functions $ 
  {\bigV}_{ n,M}^{d,\theta} \colon [0, T] \times \R^d \times \Omega \to \R
$, 
$d,n,M\in\N_0$, $\theta\in\Theta$, in \cref{eq:method} describe the MLP approximations which we employ in \cref{m01b} above to approximately compute the solutions $u_d\colon [0,T]\times \R^d\to \R$, $d\in \N$, of the PDEs in \cref{eq:PDE_intro}. The natural numbers $\FE_{d,n,M}\in\N$, $d,n,M\in \N$,
describe the sum of the number of function evaluations of $f$, 
of the number of function evaluations of $u_d(T,\cdot)$, 
of the number of function evaluations of $\mu_d$,
of the number of function evaluations of  $\sigma_d$, and of
the number of realizations of scalar random variables which
are used to compute one realization 
of the MLP approximations 
which we employ in \cref{m01b} above to approximately compute the solutions $u_d\colon [0,T]\times \R^d\to \R$, $d\in \N$, of the PDEs in \cref{eq:PDE_intro} (cf.\ also \eqref{eq:comp_effort_cor} in \cref{m01c} in \cref{sec:main_result} for a precise definition of $(\FE_{ d, n, M } )_{ (d, n, M) \in \N^3 } \subseteq \N)$.
 \cref{m01b} establishes that the solutions $u_d\colon [0,T]\times \R^d\to \R$, $d\in \N$, of the PDEs in \cref{eq:PDE_intro} can be approximated by the MLP approximations 
$ 
  {\bigV}_{ n,M}^{d,\theta} \colon [0, T] \times \R^d \times \Omega \to \R
$, 
$d,n,M\in\N_0$, $\theta\in\Theta$, in \cref{eq:method} 
with the number of involved function evaluations of $f$, $u_d(T,\cdot)$, $\mu_d$, and $\sigma_d$ and the number of involved scalar random variables growing at most quintically in the reciprocal $1/\epsilon$ of the prescribed approximation accuracy $\epsilon\in (0,\infty)$ and at most polynomially in the PDE dimension $d\in \N$. Our proofs of \cref{m01b} above and \cref{m01_x} below, respectively, are partially based on previous analyses for MLP approximations in the scientific literature (cf., e.g., \cite{EHutzenthalerJentzenKruse2016,HJKNW2018,HutzenthalerKruse2017,hutzenthaler2019overcoming,beck2019overcoming, giles20019generalised,hjk2019overcoming,beck2020overcomingelliptic}) and on analyses for numerical approximations for SDEs with non-globally Lipschitz continuous coefficient functions (cf., e.g., \cite{HutzenthalerJentzen2014Memoires}). 

The remainder of this article is organized as follows. As mentioned above, MLP approximation methods are, roughly speaking, based on the idea to reformulate the PDE under consideration (see \eqref{eq:PDE_intro}) as a stochastic fixed point equation (see (IIa) above) and then to approximate the fixed point of the stochastic fixed point equation through suitable full-history recursive multilevel Monte Carlo approximations (see (IIb) and (IIc) above). In \cref{s02} below we establish existence, uniqueness, and regularity properties for solutions of such stochastic fixed point equations. In \cref{sec:rate} below we introduce MLP approximations for solutions of such stochastic fixed point equations (see \eqref{t27} in \cref{t26} in \cref{t26_0} below), we study measurability, integrability, and independence properties for the introduced MLP approximations (see \cref{ppt0}, \cref{h08}, and \cref{k08} in \cref{ppt00} below), and we establish in \cref{x01} in \cref{x010} below upper bounds for the $L^2$-distances between the exact solutions of the considered stochastic fixed point equations and the proposed MLP approximations. In our proof of \cref{x01} we employ certain function space-valued Gronwall-type inequalities, which we establish in \cref{g09-1}, \cref{g090a}, and \cref{g09} in \cref{subsec:gronwall} below. In \cref{sec:main_result} we combine the existence, uniqueness, and regularity properties for solutions of stochastic fixed point equations, which we have established in \cref{s02}, with the error analysis for MLP approximations for stochastic fixed point equations, which we have established in \cref{sec:rate} (see \cref{x01} in \cref{x010}), to obtain a computational complexity analysis for MLP approximations for semilinear second-order PDEs with possibly nonlinear coefficient functions in front of the second-order differential operators.

\section{Stochastic fixed-point equations}\label{s02}
In this section we establish in the elementary results in \cref{b01} and \cref{d08} below existence, uniqueness, and regularity properties for solutions of stochastic fixed point equations. Similar existence, uniqueness, and regularity results for solutions of stochastic fixed point equations can, e.g., be found in Beck et al.~\cite{beck2019existence} and Beck et al.~\cite[Theorem 3.7]{beck2020nonlinear}. In our proof of \cref{b01} we use the well-known auxiliary measurability result in \cref{lem:fubini} below. For completeness we also include in this section a detailed proof for \cref{lem:fubini}.

\sloppy
\subsection{Existence of solutions of stochastic fixed-point equations}
\label{sec:ex_sfp}

\begin{lemma}\label{lem:fubini}
Let $(X,\mathcal X)$ be a measurable space, let $(Y,\mathcal Y, \mu)$ be a sigma-finite measure space, let $f\colon X\times Y \to \R$ be measurable, and assume for all $x\in X$ that
$
\int_Y |f(x,y)|\,\mu(dy)<\infty
$.
Then it holds that $X\ni x\mapsto \int_Y f(x,y)\,\mu(dy)\in \R$ is measurable.
\end{lemma}
\begin{proof}[Proof of \cref{lem:fubini}]
Throughout this proof let $\mathfrak{f}_k \colon X\times Y \to [0,\infty)$, $k \in \{ 0, 1 \}$, satisfy for all $k \in \{ 0, 1 \}$, $x\in X$, $y\in Y$ that
\begin{equation}
\mathfrak{f}_k(x,y)=\max\{(-1)^kf(x,y),0\}.
\end{equation}
Note that Fubini's theorem (cf., e.g., Klenke \cite[Theorem 14.16]{k08b}) ensures that for every measurable $g \colon X\times Y \to [0,\infty)$ it holds that 
$X \ni x \mapsto \int_Y g(x,y)\, \mu(dy) \in [0,\infty]$ is measurable. 
This proves that for all $k \in \{ 0, 1 \}$ it holds that $X \ni x \mapsto \int_Y \mathfrak{f}_k(x,y) \, \mu(dy) \in [0,\infty)$ is measurable. Combining this with the fact that for all $x\in X$ it holds that
$\int_Y f(x,y)\,\mu(dy)=\int_Y \mathfrak{f}_0(x,y)\,\mu(dy)-\int_Y \mathfrak{f}_1(x,y)\,\mu(dy)$ implies that $X\ni x\mapsto \int_Y f(x,y)\,\mu(dy)\in \R$ is measurable. The proof of \cref{lem:fubini} is thus complete.
\end{proof}

\renewcommand{\lyaV}{V}
\begin{prop}
	\label{b01}
Let 
	$d\in\N$, 
	$L,T,c\in [0,\infty)$, 	$\cO\in \mathcal{B}( \R^d)$ satisfy $\cO\neq \emptyset$,
let 
	$\lVert\cdot\rVert\colon\R^d \to [0,\infty)$ be a norm on $\R^d$, 
let 
	$(\Omega,\mathcal{F},\P)$ 
	be a probability space, 
let $X^x_{ t, s } 
	\colon \Omega \to \cO$, $s\in [t,T]$, $t \in [0,T]$, $x \in \cO$, be random variables,
assume for every
	measurable $\psi\colon[0,T]\times\cO\to[0,\infty)$ that 
  $\{(t,s)\in[0,T]^2\colon t\leq s\}\times\cO\ni(t,s,x)\mapsto \E[\psi(s,X^{x}_{t,s})]\in[0,\infty]$
  is
	measurable,
let 
	$f\colon [0,T]\times\cO\times\R\to\R$, $g\colon \cO\to\R$, 
and
	$V\colon [0,T]\times\cO\to (0,\infty) $ 
be measurable, 
assume for all 
	$t\in [0,T]$, 
	$s\in [t,T]$, 
	$x\in\cO$, 
	$v,w\in\R$
that 
$| f(t,x,0) | \leq c V(t,x)$, $|g(x)| \leq c V(T,x)$, $\E[V(s,X^{x}_{t,s})] \leq V(t,x)$, and
$	
	|f(t,x,v)-f(t,x,w)|
	\leq 
	L |v-w|$.
Then 
\begin{enumerate}[(i)]
\item \label{v17}
there exists a unique 
measurable $u\colon[0,T]\times\cO\to\R$
which satisfies for all $t\in [0,T]$, 
	$x\in\cO$  that
$\Exp{ |
		g(X^{x}_{t,T}) |
    }
		+ 
		\int_t^{T} 
    \Exp{|
		f(s,X^{x}_{t,s},{u}(s,X^{x}_{t,s}))|}ds+
\sup_{y\in\cO}\sup_{s\in [0,T]} \frac{|u(s,y)|}{V(s,y)}		
		<\infty$
		 and
	\begin{equation}\label{eq:sfpe}
	u(t,x) 
	= 
	\Exp{ 
		g(X^{x}_{t,T}) 
    }
		+ 
		\int_t^T 
	\Exp{ 
		f\bigl( s, X^{x}_{t,s}, u(s,X^{x}_{t,s})\bigr)} ds 
	\end{equation} 
and
\item\label{v18} it holds that for all $t\in[0,T]$ that
	\begin{equation}\label{t99}
	\sup_{x\in\cO}\left(\frac{|u(t,x)|}{V(t,x)}\right)
  \leq \left[\sup_{x\in\cO}\left(\frac{|g(x)|}{V(T,x)}\right)
  +\sup_{x\in\cO}\sup_{s\in[t,T]}\left(\frac{|Tf(s,x,0)|}{V(s,x)}\right)\right]e^{L(T-t)}.
	\end{equation}
\end{enumerate}
\end{prop}
\begin{proof}[Proof of \cref{b01}]
 Throughout this proof let $\cV$ satisfy
 \begin{equation}
	\cV = \left\{
u\colon [0,T]\times\cO\to\R \colon 
\left[
u\text{ is 
measurable and}
\left[
\sup_{t\in [0,T]}\sup_{x\in\cO} \left(\frac{|u(t,x)|}{V(t,x)}\right)<\infty \right]
\right]
	\right\} \label{v10}
 \end{equation}
 and let 
 	$\lVert\cdot\rVert_{\lambda}\colon \cV \to [0,\infty)$, $\lambda\in \R$, 
satisfy for every 	$\lambda\in \R$, $v\in\cV$ that 
	\begin{equation}\label{v10b}
	\|v\|_{\lambda} 
	= 
	\sup_{t\in [0,T]}
	\sup_{x\in \cO} 
	\left( 
	\frac{e^{\lambda t}|v(t,x)|}{V(t,x)}
	\right).
	\end{equation}
Note that \cref{v10} and \cref{v10b} ensure that for all $\lambda\in \R$ it holds that $( \cV, \lVert \cdot \rVert_{ \lambda } )$ is a normed $\R$-vector space. Next we show that $( \cV, \lVert \cdot \rVert_{ 0 } )$ is an $\R$-Banach space. For this let $v = ( v_n )_{ n \in \N } \colon \N \to \cV$ satisfy 
\begin{equation}\label{v10c}
\limsup\nolimits_{ N \to \infty } \big( \sup\nolimits_{ n, m \in \N \cap [N,\infty) } \norm{v_n-v_m}_{0}\big) = 0.
\end{equation}
Observe that \cref{v10c} demonstrates that for all $t\in[0,T]$, $x\in\cO$ it holds that $v_n(t,x)\in \R$, $n \in \N$, is a Cauchy sequence. The fact that $( \R , \lvert \cdot \rvert)$ is an $\R$-Banach space hence assures that there exists $\phi\colon[0,T]\times\cO\to\R$ which satisfies for all $t\in[0,T]$, $x\in \cO$ that $\limsup_{ n \to \infty } | \phi(t,x) - v_n(t,x) | = 0$. Combining this with the fact that for all $n\in\N$ it holds that $v_n$ is measurable proves that $\phi$ is measurable. Next observe that the fact that for all $t\in[0,T]$, $x\in \cO$ it holds that $\limsup_{ n \to \infty } | \phi(t,x) - v_n(t,x) | = 0$ yields that for all $N\in \N$ it holds that
\begin{equation}
\begin{split}
&
\sup_{ t \in [0,T] } \sup_{ x\in \cO } \left( \frac{ |\phi(t,x)| }{ V(t,x) } \right) 
= \sup_{ t \in [0,T] } \sup_{ x\in \cO }\left(\frac{ |\lim_{n\to\infty}v_n(t,x)|}{V(t,x)} \right)
\\
& 
\leq \sup_{ t \in [0,T] } \sup_{ x\in \cO } \left( \frac{ \left[ \sup_{ n \in \N } | v_n(t,x) | \right] }{ V(t,x) } \right)
 =  
\sup_{ n \in \N } \sup_{ t \in [0,T] } \sup_{ x\in \cO }  \left( \frac{ | v_n(t,x) | }{ V(t,x) } \right) = \sup_{ n\in \N} \| v_n \|_0
\\ &
\leq \left[ \sup_{ n \in \N \cap [N,\infty) } \| v_n - v_N \|_0 \right] + \left[ \max_{n \in \{ 1, 2, \ldots, N \}} \| v_n \|_0 \right]\\
&\leq \left[ \sup_{ n, m \in \N \cap [N,\infty) } \| v_n - v_m \|_0 \right] + \left[ \max_{ n \in \{ 1, 2, \ldots, N \}} \| v_n \|_0 \right].
\end{split}
\end{equation}
This and \cref{v10c} imply that
\begin{equation}
\sup_{ t \in [0,T] } \sup_{ x\in \cO } \left( \frac{ |\phi(t,x)| }{ V(t,x) } \right) \leq  \sup_{ n\in \N} \| v_n \|_0<\infty.
\end{equation}
Combining this with the fact that $\phi$ is measurable proves that $\phi \in \cV$. In addition, observe that \cref{v10c} assures that
\begin{equation}  \begin{split}
  \limsup_{n\to\infty}\|\phi-v_n\|_0
  &=
  \limsup_{n\to\infty}
  \left[
	\sup_{t\in [0,T]}
	\sup_{x\in \cO} 
	\left( 
	\frac{\lim_{m\to\infty}|v_m(t,x)-v_n(t,x)|}{V(t,x)}
	\right)
	\right]
  \\&
  \le \limsup_{ n \to \infty } \left[ \sup_{ t \in  [0,T] } \sup_{ x \in \cO } \sup_{ m \in \N \cap [n,\infty) }  \left( \frac{|v_m(t,x)-v_n(t,x)|}{V(t,x)} \right) \right]\\
  & = \limsup_{ n \to \infty } \left[ \sup_{ m \in \N \cap [n,\infty) } \|v_m-v_n\|_0 \right] = 0 .
\end{split}     \end{equation}
This demonstrates that 
	$(\cV,\lVert\cdot\rVert_{0})$ is an $\R$-Banach space. 	
Combining this with the fact that for all 
	$\Lambda\in\R$, 
	$\lambda\in [\Lambda,\infty)$, 
	$v\in \cV$ 
it holds that 
	$\norm{v}_{\Lambda} \leq \norm{v}_{\lambda} \leq e^{(\lambda-\Lambda) T}\norm{v}_{\Lambda} $ 
shows that for all 
	$\lambda\in\R$ 
it holds that 
	$(\cV,\lVert\cdot\rVert_{\lambda})$ is an $\R$-Banach space.
Moreover, observe that the fact that for all $x\in \cO$ it holds that $|g(x)| \leq c V(T,x)$ ensures that for all $t\in[0,T]$, $x\in\cO$ it holds that
\begin{equation}\label{eq:estimate_g}
\Exp{ |
		g(X^{x}_{t,T}) |
    }
    \le  \left[ \sup_{y\in\cO}\frac{|g(y)|}{V(T,y)}\right]\E[V(T,X^x_{t,T})]
    \leq
  \left[
  \sup_{y\in\cO}\left(\frac{|g(y)|}{V(T,y)}\right)\right]V(t,x)<\infty.
\end{equation}	
This implies that 
\begin{equation}\label{eq:exp_g_meas}
[0,T]\times\cO\ni(t,x)\mapsto
	\Exp{ 
		g(X^{x}_{t,T}) 
  }\in \R
\end{equation}
is measurable.
Furthermore, observe that the triangle inequality
and the fact that for all
$t\in [0,T]$, $x\in\cO$, $v,w\in\R$ it holds
that 
$
	|f(t,x,v)-f(t,x,w)|
	\leq 
	L |v-w|$
  yield that for all $v\in \cV$, $t\in[0,T]$, $s\in [t,T]$, $x\in\cO$ it holds that
  \begin{equation}  \begin{split}\label{eq:estimate.Phiv}
	&
    \Exp{|
		f(s,X^{x}_{t,s},v(s,X^{x}_{t,s}))|}
   	\leq
    \Exp{
		|f(s,X^{x}_{t,s},0)|
		+L|v(s,X^{x}_{t,s})|}
   \\&
	\leq
\left[\sup_{r\in[s,T]}\sup_{y\in\cO}\left(
  \frac{|f(r,y,0)|+L|v(r,y)|}{V(r,y)}\right)\right]\Exp{V(s,X^x_{t,s})}
   \\&
	\leq
  \left[
  \sup_{r\in[s,T]}\sup_{y\in\cO}
  \left(
  \frac{|f(r,y,0)|+L|v(r,y)|}{V(r,y)}\right)
  \right]V(t,x)<\infty.
  \end{split}     \end{equation}
 This implies that for all $v\in\cV$ it holds that
 \begin{multline}
   [0,T]\times [0,T]\times\cO\ni(s,t,x)
   \mapsto
    \1_{[t,T]}(s)\,
   \Exp{
		f\big(\max\{s, t\},X^{x}_{t,\max\{s, t\}},v(,X^{x}_{t,\max\{s,t\}})\big)}ds\in\R
 \end{multline}
 is measurable. Moreover, observe that \eqref{eq:estimate.Phiv} implies
that for all $v\in \cV$, $t\in[0,T]$, $x\in\cO$ it holds that
 \begin{equation}
 \int_0^{T} 
    \1_{[t,T]}(s)\,
   \big|\E\big[
		f\big(\max\{s, t\},X^{x}_{t,\max\{s, t\}},v(\max\{s, t\},X^{x}_{t,\max\{s,t\}})\big)\big]\big| \, ds<\infty.
 \end{equation}
 \cref{lem:fubini} and 
\cref{eq:exp_g_meas} hence prove that for all $v\in\cV$ it holds
that
  \begin{multline} 
  [0,T]\times\cO\ni(t,x)
  \mapsto
	\Exp{ 
		g(X^{x}_{t,T}) 
  }
  \\
		+ 
		\int_0^{T} 
    \1_{[t,T]}(s)\,
   \Exp{
		f\big(\max\{s, t\},X^{x}_{t,\max\{s, t\}},v(\max\{s, t\},X^{x}_{t,\max\{s,t\}})\big)}ds\in\R
   \label{v12} \end{multline}
  is measurable.
This and \eqref{eq:estimate.Phiv} imply that there exists
  $\Phi\colon\cV\to\cV$ which satisfies for all 
	$t\in [0,T]$, 
	$x\in\cO$, 
	$v\in\cV$ 
that 
\begin{equation}\label{v13}
	\left(\Phi(v)\right)\!(t,x) 
	= 
	\Exp{ 
		g(X_{t,T}^x) 
    }
		+ 
		\int_t^{T} 
    \Exp{
		f\bigl(s,X^{x}_{t,s},v(s,X^{x}_{t,s})\bigr)}ds .
\end{equation}
In addition, note that the fact that for all $t\in [0,T]$, $x\in \cO$, $v,w\in \R$ it holds that $|f(t,x,v)-f(t,x,w)|\le L|v-w|$ ensures that for all $\lambda\in (0,\infty)$, 
	$v,w\in\cV$  it holds that $
	\| \Phi(v) - \Phi(w) \|_{\lambda} 
	\leq 
	\frac{L}{\lambda} 
	\| v - w \|_{\lambda} 
$ (cf., e.g., Beck et al.~\cite[Lemma 2.8]{beck2019existence}).
Hence, we obtain for all 
	$\lambda \in [2L,\infty)$, 
	$v,w\in \cV$ 
that 
	\begin{equation}
	\| \Phi(v) - \Phi(w) \|_{\lambda} 
	\leq 
	\frac12 
	\| v - w \|_{\lambda}. 
	\end{equation}
Banach's fixed point theorem therefore demonstrates that there exists a unique $u\in\cV$ which satisfies $\Phi(u)=u$. 
Combining this, \eqref{v10}, and
\eqref{eq:estimate.Phiv} with \eqref{v13} establishes \cref{v17}. Next observe that \eqref{eq:estimate_g} and \eqref{eq:estimate.Phiv} imply that for all $t\in[0,T]$ it holds that
\begin{equation}  \begin{split}
  &\sup_{r\in[t,T]}\sup_{x\in\cO}\left(\frac{|u(r,x)|}{V(r,x)}\right)
  =
  \sup_{r\in[t,T]}\sup_{x\in\cO}\left(\frac{|(\Phi(u))(r,x)|}{V(r,x)}\right)
  \\&
	\leq
  \sup_{y\in\cO}\left(\frac{|g(y)|}{V(T,y)}\right)
  +\sup_{r\in[0,T]}\sup_{y\in\cO}\left(\frac{|Tf(r,y,0)|}{V(r,y)}\right)
  +L\int_t^T\bigg[\sup_{r\in[s,T]}\sup_{y\in\cO}\left(\frac{|u(r,y)|}{V(r,y)}\right)\bigg]\,ds.
\end{split}     \end{equation}
This, the fact that $u\in\cV$, and Gronwall's lemma
yield that for all $t\in[0,T]$ it holds that
\begin{equation}  \begin{split}
  \sup_{r\in[t,T]}\sup_{x\in\cO}\left(\frac{|u(r,x)|}{V(r,x)}\right)
	\leq
  \left[\sup_{y\in\cO}\left(\frac{|g(y)|}{V(T,y)}\right)
  +\sup_{r\in[0,T]}\sup_{y\in\cO}\left(\frac{|Tf(r,y,0)|}{V(r,y)}\right)
  \right]e^{L(T-t)}.
\end{split}     \end{equation}
This establishes \cref{v18}.
The proof of \cref{b01} is thus complete.
\end{proof}

\subsection{Perturbation analysis for stochastic fixed-point equations}
\label{sec:pert_sfp}
\renewcommand{\growrate}{\rho}

\begin{lemma}
\label{d08}
Let $d\in \N$, $\funcG\in C( \R^d,\R)$,
$c,L,\growrate,\eta\in [0,\infty)$, $T,\delta, p, q \in (0,\infty)$ satisfy $p^{-1}+q^{-1}\leq 1$,
let
$\lVert\cdot\rVert\colon\R^d \to [0,\infty)$ be a seminorm on $\R^d$,
let
$(\Omega, \mathcal{F}, \P)$
be a probability space,
let
$\exactpr_{t,(\cdot)}^{x,k}=(\exactpr_{t,s}^{x,k}(\omega))_{(s,\omega)\in[t,T]\times\Omega}
\colon[t,T]\times\Omega\to\R^d 
$,  $t\in[0,T] $, $x\in\R^d$, $k\in \{1,2\}$, be 
measurable,
let 
$\smallF\colon [0,T]\times \R^{d}\times\R\to\R$,
 $\lyaV \colon\R^d\to (0,\infty)$,
$\lyaPsi\colon [0,T]\times\R^d\to(0,\infty)$,
and
$\smallU_k\colon[0,T]\times \R^d\to\R$, $k\in \{1,2\}$,
be 
measurable, 
assume 
for all $s\in[0,T]$, $r\in[s,T]$
and all measurable $h \colon \R^d \times \R^d \to [0, \infty)$
  that 
  $
\R^d \times \R^d \ni (y_1,y_2) \mapsto \E\bigl[h\bigl(X^{y_1,1}_{s,r},X^{y_2,1}_{s,r}\bigr)\bigr] \in [0,\infty]
$
is measurable,
 and assume for all $x,y\in\R^d$, $t\in[0,T]$, $s\in[t,T]$, $r\in[s,T]$,
$v,w\in\R$, $k\in \{1,2\}$ and all measurable $h \colon \R^d \times \R^d \to [0, \infty)$ that
  \begin{equation}
  \label{b01g}
\exactpr_{t,t}^{x,k}
=x,
\qquad 
\E \big[\lyaPsi(s,\exactpr_{t,s}^{x,k})\big]\leq \eta \lyaPsi(t,x),
  \end{equation}
  \begin{equation}\label{b01k}
\E\! \left[ \E\!\left[ h\bigl(X^{\mathfrak{x},1}_{s,r},X^{\mathfrak{y},1}_{s,r}\bigr) \right] \bigr|_{(\mathfrak{x},\mathfrak{y})=(X_{t,s}^{x,1},X_{t,s}^{y,1})} \right]
   =\E\!\left[h\bigl(X^{x,1}_{t,r},X^{y,1}_{t,r}\bigr)\right] ,
  \end{equation}
  \begin{equation}
  \label{r20}
\max\{T|\smallF(t,x,v)-\smallF(t,y,w)|,|\funcG(x)-\funcG(y)|\}
\leq
{\LipConstF T}
 |v-w|+T^{-\nicefrac{1}{2}}
[
\lyaV(x)+\lyaV(y)]^{\nicefrac{1}{p}}\|x-y\|,
  \end{equation}
\begin{equation}\label{b10}\E\!\left[
\E\bigl[\|\exactpr_{s,r}^{\mathfrak{x},1}-\exactpr_{s,r}^{\mathfrak{y},1}\|^{q}\bigr]\bigr|_{\substack{(\mathfrak{x},\mathfrak{y})=(\exactpr_{t,s}^{x,1},\exactpr_{t,s}^{x,2})}}\right]\leq \delta^q\lyaPsi(t,x),
\end{equation}
\begin{equation}
\E\big[|\funcG(\exactpr_{t,T}^{x,k})|\big]+\int_t^T\E\big [| \smallF(s,\exactpr_{t,s}^{x,k},\smallU_k(s,\exactpr_{t,s}^{x,k}))|\big]\,ds
<\infty,
\end{equation}
\begin{equation}
\smallU_k(t,x)=\E\!\left[\funcG(\exactpr_{t,T}^{x,k})+\int_t^T \smallF(s,\exactpr_{t,s}^{x,k},\smallU_k(s,\exactpr_{t,s}^{x,k}))\,ds\right],
\label{b20}
\end{equation} 
and
$\max\{|\smallU_k(t,x)|^p,ce^{\growrate(t-s)}\E[\lyaV(\exactpr_{t,s}^{x,k})]\}\le c \lyaV(x)$.
Then  it holds for all  $t\in[0,T]$, $x\in\R^d$ that
\begin{align}
|\smallU_1(t,x)-  \smallU_2(t,x)|\leq
4(1+LT)T^{-\nicefrac{1}{2}} \exp\!\big((L+\tfrac{\rho}{p}+\eta^{\nicefrac{1}{q}}L)(T-t)\big) |\lyaV(x)|^{\nicefrac{1}{p}}  |\lyaPsi(t,x)|^{\nicefrac{1}{q}}  \delta.
\end{align}
\end{lemma}
\begin{proof}[Proof of \cref{d08}]
\newcommand{\calX}{\mathcal{X}}
\newcommand{\calY}{\mathcal{Y}}
Throughout this proof assume w.l.o.g.\ that $c>0$. Note that
H\"older's inequality 
implies that 
for all $s\in[0,T]$, $r\in[s,T]$, $x_1,x_2\in\R^d$,
 $k_1,k_2\in\{1,2\}$ it holds that
\begin{align}\begin{split}
&\E \!\left[\left(\lyaV(X_{s,r}^{x_1,k_1})+\lyaV(X_{s,r}^{x_2,k_2})\right)^{\nicefrac{1}{p}}\left\|X_{s,r}^{x_1,k_1}-X_{s,r}^{x_2,k_2}\right\|\right]\\
&
\leq 
\left(\E \!\left[\lyaV(X_{s,r}^{x_1,k_1})+\lyaV(X_{s,r}^{x_2,k_2})\right]\right)^{\nicefrac{1}{p}}
\left(
\E\!\left[\left\|X_{s,r}^{x_1,k_1}-X_{s,r}^{x_2,k_2}\right\|^{q}\right]\right)^{\!\nicefrac{1}{q}}\\&\leq e^{\growrate(r-s)/p}(\lyaV(x_1)+\lyaV(x_2))^{\nicefrac{1}{p}}
\left(
\E\!\left[\left\|X_{s,r}^{x_1,k_1}-X_{s,r}^{x_2,k_2}\right\|^{q}\right]\right)^{\!\nicefrac{1}{q}}.\label{d13}
\end{split}\end{align}
This, \eqref{b20}, the triangle inequality,
Tonelli's theorem, 
and
\eqref{b01g}--\eqref{r20} show that for all $t\in[0,T]$, $s\in[t,T]$, $x_1,x_2\in\R^d$ it holds that 
\begin{align}\begin{split}
&\E\!\left[|\smallU_1(s,\exactpr_{t,s}^{x_1,1})-\smallU_1(s,\exactpr_{t,s}^{x_2,1})|
 \right]\\
&= \E\biggl[\Bigl|\E\Bigl[\funcG(\exactpr_{s,T}^{y_1,1})-\funcG(\exactpr_{s,T}^{y_2,1})\\
&\quad+
\smallint_{s}^{T}
\smallF(r,\exactpr_{s,r}^{y_1,1},\smallU_1(r,\exactpr_{s,r}^{y_1,1}))
-
\smallF(r,\exactpr_{s,r}^{y_2,1},\smallU_1(r,\exactpr_{s,r}^{y_2,1}))\,dr\Bigr]\Bigr|\Bigr|
_{\substack{(y_1,y_2)=(\exactpr_{t,s}^{x_1,1},\exactpr_{t,s}^{x_2,1})}}\biggr]
\\
&\leq  \E \!\left[|\funcG(\exactpr_{t,T}^{x_1,1})-\funcG(\exactpr_{t,T}^{x_2,1})|\right]\\
&\quad 
+\int_{s}^{T}\E\!\left[\left|
\smallF(r,\exactpr_{t,r}^{x_1,1},\smallU_1(r,\exactpr_{t,r}^{x_1,1}))
-
\smallF(r,\exactpr_{t,r}^{x_2,1},\smallU_1(r,\exactpr_{t,r}^{x_2,1}))\right|\right]dr\\
&\leq \E \!\left[T^{-\nicefrac{1}{2}}\left(\lyaV(\exactpr_{t,T}^{x_1,1})+\lyaV(\exactpr_{t,T}^{x_2,1})\right)^{\nicefrac{1}{p}}\left\|\exactpr_{t,T}^{x_1,1}-\exactpr_{t,T}^{x_2,1}\right\|\right]\\
&
\quad 
+\LipConstF\int_{s}^{T}\E\!\left[\left| \smallU_1(r,\exactpr_{t,r}^{x_1,1})-\smallU_1(r,\exactpr_{t,r}^{x_2,1})\right|\right]dr\\
&\quad+\int_{s}^{T}\E \!\left[T^{-\nicefrac{3}{2}}\left(\lyaV(\exactpr_{t,r}^{x_1,1})+\lyaV(\exactpr_{t,r}^{x_2,1})\right)^{\nicefrac{1}{p}}\left\|\exactpr_{t,r}^{x_1,1}-\exactpr_{t,r}^{x_2,1}\right\|\right]dr
\\
&\leq \LipConstF\int_{s}^{T}\E\!\left[\left| \smallU_1(r,\exactpr_{t,r}^{x_1,1})-\smallU_1(r,\exactpr_{t,r}^{x_2,1})\right|\right]dr+
 e^{\growrate(T-t)/p}(\lyaV(x_1)+\lyaV(x_2))^{\nicefrac{1}{p}}\\&\quad{\cdot}
\left[T^{-\nicefrac{1}{2}}
\left(
\E\!\left[\left\|\exactpr_{t,T}^{x_1,1}-\exactpr_{t,T}^{x_2,1}\right\|^{q}\right]\right)^{\!\nicefrac{1}{q}}+
\int_{t}^{T}T^{-\nicefrac{3}{2}}
\left(
\E\!\left[\left\|\exactpr_{t,r}^{x_1,1}-\exactpr_{t,r}^{x_2,1}\right\|^{q}\right]\right)^{\!\nicefrac{1}{q}}dr\right].\end{split}\end{align}
This, Gronwall's lemma, and the fact that for all $t\in [0,T]$, $s\in [t,T]$, $x\in \R^d$ it holds that $\max\{e^{-\growrate(s-t)}\E\bigl[\lyaV(\exactpr_{t,s}^{x,k})\bigr],|\smallU_1(t,x)|^p\}\le c \lyaV(x)$ imply that for all  $t\in[0,T]$, $s\in[t,T]$, $x_1,x_2\in\R^d$ it holds that
\begin{align}\begin{split}
&\E\!\left[|\smallU_1(s,\exactpr_{t,s}^{x_1,1})-\smallU_1(s,\exactpr_{t,s}^{x_2,1})|
 \right]\leq  e^{\growrate(T-t)/p}(\lyaV(x_1)+\lyaV(x_2))^{\nicefrac{1}{p}}\\
& {\cdot} \left[T^{-\nicefrac{1}{2}}
\left(
\E\!\left[\left\|\exactpr_{t,T}^{x_1,1}-\exactpr_{t,T}^{x_2,1}\right\|^{q}\right]\right)^{\!\nicefrac{1}{q}}+
\int_{t}^{T}T^{-\nicefrac{3}{2}}
\left(
\E\!\left[\left\|\exactpr_{t,r}^{x_1,1}-\exactpr_{t,r}^{x_2,1}\right\|^{q}\right]\right)^{\!\nicefrac{1}{q}}dr\right]e^{\LipConstF (T-s)}.
\end{split}\end{align}
Combining this with \eqref{b01g} assures that for all $s\in[0,T]$, $x_1,x_2\in\R^d$ it holds that 
\begin{align}\begin{split}
&|\smallU_1(s,{x_1})-\smallU_1(s,{x_2})|
\leq  e^{(\LipConstF+\growrate/p)(T-s)}(\lyaV(x_1)+\lyaV(x_2))^{\nicefrac{1}{p}}\\
& {\cdot} \left[T^{-\nicefrac{1}{2}}
\left(
\E\!\left[\left\|\exactpr_{s,T}^{x_1,1}-\exactpr_{s,T}^{x_2,1}\right\|^{q}\right]\right)^{\!\nicefrac{1}{q}}+
\int_{s}^{T}T^{-\nicefrac{3}{2}}
\left(
\E\!\left[\left\|\exactpr_{s,r}^{x_1,1}-\exactpr_{s,r}^{x_2,1}\right\|^{q}\right]\right)^{\!\nicefrac{1}{q}}dr\right].
\end{split}\end{align}
This, 
H\"older's inequality, 
 \eqref{b01k},
and \eqref{b10}  assure that for all
$t\in [0,T]$,
 $s\in[t,T]$, $x\in\R^d$ it holds that 
\begin{align}\begin{split}
&\E\bigl[|\smallU_1(s,\exactpr_{t,s}^{x,1})- \smallU_1(s,\exactpr_{t,s}^{x,2})|\bigr]\\
&\leq  e^{(\LipConstF+\growrate/p)(T-s)}
\E\!\left[
\left.
\left[T^{-\nicefrac{1}{2}}(\lyaV(z_1)+\lyaV(z_2))^{\nicefrac{1}{p}}
\left(
\E\!\left[\left\|\exactpr_{s,T}^{z_1,1}-\exactpr_{s,T}^{z_2,1}\right\|^{q}\right]\right)^{\!\nicefrac{1}{q}}\right]\right|_{(z_1,z_2)=(\exactpr_{t,s}^{x,1},\exactpr_{t,s}^{x,2})}\right]
\\
&+e^{(\LipConstF+\growrate/p)(T-s)}
\int_{s}^{T}\E\!\left[\left.\left[T^{-\nicefrac{3}{2}}
(\lyaV(z_1)+\lyaV(z_2))^{\nicefrac{1}{p}}
\left(
\E\!\left[\left\|\exactpr_{s,r}^{z_1,1}-\exactpr_{s,r}^{z_2,1}\right\|^{q}\right]\right)^{\!\nicefrac{1}{q}}\right]\right|_{\substack{(z_1,z_2)=(\exactpr_{t,s}^{x,1},\exactpr_{t,s}^{x,2}})}\right]dr\\
&\leq  2 e^{(\LipConstF+\growrate/p)(T-s)}T^{-\nicefrac{1}{2}}
\left(\E[\lyaV (\exactpr_{t,s}^{x,1})+\lyaV(\exactpr_{t,s}^{x,2}) ]\right)^{\nicefrac{1}{p}}
\\
&\cdot
\sup_{r\in[s,T]}\left(
\E\!\left[
\E\Bigl[\left\|\exactpr_{s,r}^{z_1,1}-\exactpr_{s,r}^{z_2,1}\right\|^{q}\Bigr]\Bigr|_{\substack{(z_1,z_2)=(\exactpr_{t,s}^{x,1},\exactpr_{t,s}^{x,2})}}\right]\right)^{\!\nicefrac{1}{q}}\\
&\leq
2T^{-\nicefrac{1}{2}} e^{(\LipConstF+\growrate/p)(T-t)}(2\lyaV(x))^{\nicefrac{1}{p}}\delta (\lyaPsi(t,x))^{\nicefrac{1}{q}}.
\end{split}\end{align}
This, \eqref{b20}, the triangle inequality, \eqref{r20}, \eqref{d13},
and
\eqref{b10} show that for all $t\in[0,T]$, $x\in\R^d$, $\epsilon\in (0,1)$ it holds that 
\begin{align}\begin{split}
&|\smallU_1(t,x)-\smallU_2(t,x)|\\
&
=\left|
\E\!\left[\funcG(\exactpr_{t,T}^{x,1})-\funcG(\exactpr_{t,T}^{x,2})+\int_t^T\!
 f(s,\exactpr_{t,s}^{x,1}, \smallU_1(s,\exactpr_{t,s}^{x,1}))- f(s,\exactpr_{t,s}^{x,2}, \smallU_2( s,\exactpr_{t,s}^{x,2} ))\,ds\right]\right|
\\
&\leq \E\!\left[|\funcG(\exactpr_{t,T}^{x,1})-\funcG(\exactpr_{t,T}^{x,2})|\right]+\int_t^T\! \E\!\left[|\smallF(s,\exactpr_{t,s}^{x,1},\smallU_1(s,\exactpr_{t,s}^{x,1})- (\smallF(s,\exactpr_{t,s}^{x,2},\smallU_1(s,\exactpr_{t,s}^{x,2}))|\right]ds\\&\quad+
\int_t^T\!\E\!\left[|\smallF(s,\exactpr_{t,s}^{x,2},\smallU_1(s,\exactpr_{t,s}^{x,2}))- \smallF (s,\exactpr_{t,s}^{x,2},\smallU_2(s,\exactpr_{t,s}^{x,2}))|\right]ds\\
&\leq  \E\!\left[T^{-\nicefrac{1}{2}}\left(\lyaV(\exactpr_{t,T}^{x,1})+\lyaV(\exactpr_{t,T}^{x,2})\right)^{\nicefrac{1}{p}}
\left\|\exactpr_{t,T}^{x,1}-\exactpr_{t,T}^{x,2}\right\|
\right]
\\&\quad
+\int_{t}^{T}\E\!\left[T^{-\nicefrac{3}{2}}\left(\lyaV(\exactpr_{t,s}^{x,1})+\lyaV(\exactpr_{t,s}^{x,2})\right)^{\nicefrac{1}{p}}
\left\|\exactpr_{t,s}^{x,1}-\exactpr_{t,s}^{x,2}\right\|\right]ds
\\
&\quad 
+\LipConstF\int_t^T\! \E\!\left[|\smallU_1(s,\exactpr_{t,s}^{x,1})- \smallU_1(s,\exactpr_{t,s}^{x,2})|\right]ds
+
\LipConstF
\int_t^T\!\E\!\left[|\smallU_1(s,\exactpr_{t,s}^{x,2})-  \smallU_2(s,\exactpr_{t,s}^{x,2})|\right]
ds\\
&\leq T^{-\nicefrac{1}{2}}e^{\growrate (T-t)/p}(2\lyaV(x))^{\nicefrac{1}{p}}\delta (\epsilon+\lyaPsi(t,x))^{\nicefrac{1}{q}}+
T^{-\nicefrac{1}{2}}e^{\growrate (T-t)/p}(2\lyaV(x))^{\nicefrac{1}{p}}\delta (\epsilon+\lyaPsi(t,x))^{\nicefrac{1}{q}}
\\
&\quad 
+\LipConstF T\left[2T^{-\nicefrac{1}{2}} e^{(\LipConstF+\growrate/p)(T-t)}(2\lyaV(x))^{\nicefrac{1}{p}}\delta (\epsilon+\lyaPsi(t,x))^{\nicefrac{1}{q}}\right]\\
&\quad
+\LipConstF
\int_t^T\!\left[
\sup_{r\in[s,T]}\sup_{z\in\R^d}\tfrac{e^{\growrate r/p}|\smallU_1(r,z)-  \smallU_2(r,z)|}{(\lyaV(z))^{\nicefrac{1}{p}}  (\epsilon+\lyaPsi(r,z))^{\nicefrac{1}{q}}  }\right]e^{-\growrate s/p}
\E\!\left[(\lyaV(\exactpr_{t,s}^{x,2}))^{\nicefrac{1}{p}}
(\epsilon+\lyaPsi(s,\exactpr_{t,s}^{x,2}))^{\nicefrac{1}{q}}
\right]ds.  \label{d14}
\end{split}\end{align}
Next note that \eqref{b01g} implies
for all $t\in[0,T]$, $x\in\R^d$ 
that $\lyaPsi(t,x)\leq \eta\lyaPsi(t,x)$. The fact that $\psi>0$ hence demonstrates that $\eta\geq1$.
H\"older's inequality,  the assumption that $1/p+1/q\leq 1$, and \eqref{b01g} therefore assure that for all $t\in[0,T]$, $s\in[t,T]$, $x\in\R^d$, $\epsilon\in(0,1)$ it holds that 
\begin{align}\begin{split}
&\E\bigl[(\lyaV(\exactpr_{t,s}^{x,2})^{\nicefrac{1}{p}}
(\epsilon+\lyaPsi(\exactpr_{t,s}^{x,2}))^{\nicefrac{1}{q}}
\bigr]\leq 
\bigl(\E[\lyaV(\exactpr_{t,s}^{x,2}]\bigr)^{\nicefrac{1}{p}}
\bigl(\epsilon+\E[\lyaPsi(s,\exactpr_{t,s}^{x,2})]\bigr)^{\nicefrac{1}{q}}\\
&
\leq e^{\growrate (s-t)/p}
(\lyaV(x))^{\nicefrac{1}{p}}
(\epsilon+\eta\lyaPsi(t,x))^{\nicefrac{1}{q}}\leq 
e^{\growrate (s-t)/p}
(\lyaV(x))^{\nicefrac{1}{p}}
\bigl[\eta(\epsilon+\lyaPsi(t,x))\bigr]^{\nicefrac{1}{q}}
. 
\end{split}\end{align}
This and \eqref{d14} demonstrate that 
for all $t\in[0,T]$, $x\in\R^d$, $\epsilon\in(0,1)$ it holds that
\begin{align}\begin{split}
&|\smallU_1(t,x)-\smallU_2(t,x)|\leq (2+2LT)\left[T^{-\nicefrac{1}{2}} e^{(\LipConstF+\growrate/p)(T-t)}(2\lyaV(x))^{\nicefrac{1}{p}}\delta(\epsilon+\lyaPsi(t,x))^{\nicefrac{1}{q}}\right]\\& +\LipConstF
\int_t^T\!\left[
\sup_{r\in[s,T],z\in\R^d}\frac{e^{\growrate r/p}|\smallU_1(r,z)-  \smallU_2(r,z)|}{(\lyaV(z))^{\nicefrac{1}{p}}  (\epsilon+\lyaPsi(r,z))^{\nicefrac{1}{q}}  }\right]e^{-\growrate s/p}
e^{\growrate (s-t)/p}
(\lyaV(x))^{\nicefrac{1}{p}}\eta^{\nicefrac{1}{q}}
(\epsilon+\lyaPsi(t,x))^{\nicefrac{1}{q}}\,ds.
\end{split}\end{align}
This and the fact $2^{1/p}\leq 2$ imply that for all $t\in[0,T]$, $\epsilon\in (0,1)$ it holds that 
\begin{align}\begin{split}
&\left[
\sup_{r\in[t,T],x\in\R^d}\frac{e^{(L+\growrate /p)r}|\smallU_1(r,x)-  \smallU_2(r,x)|}{(\lyaV(x))^{\nicefrac{1}{p}}  (\epsilon+\lyaPsi(r,x))^{\nicefrac{1}{q}}  }\right]\\&\leq 
4(1+LT)T^{-\nicefrac{1}{2}} e^{(\LipConstF+\growrate/p)T}\delta+\LipConstF\eta^{\nicefrac{1}{q}}
\int_t^T\!\left[
\sup_{r\in[s,T],z\in\R^d}\frac{e^{(L+\growrate /p)r}|\smallU_1(r,z)-  \smallU_2(r,z)|}{(\lyaV(z))^{\nicefrac{1}{p}}  (\epsilon+\lyaPsi(r,z))^{\nicefrac{1}{q}}  }\right]ds.
\end{split}\end{align}
Gronwall's lemma  therefore ensures that for all $t\in[0,T]$, $\epsilon\in(0,1)$ it holds that 
\begin{align}
\left[
\sup_{r\in[t,T],x\in\R^d}\frac{e^{(L+\growrate /p)r}|\smallU_1(r,x)-  \smallU_2(r,x)|}{(\lyaV(x))^{\nicefrac{1}{p}}  (\epsilon+\lyaPsi(r,x))^{\nicefrac{1}{q}}  }\right]\leq
4(1+LT)T^{-\nicefrac{1}{2}}e^{(\LipConstF+\growrate/p)T}\delta
e^{ \LipConstF\eta^{\nicefrac{1}{q}}(T-t)}.
\end{align}
Hence, we obtain that for all $t\in[0,T]$, $x\in\R^d$, $\epsilon\in(0,1)$ it holds that
\begin{align}
|\smallU_1(t,x)-  \smallU_2(t,x)|\leq
4(1+LT)T^{-\nicefrac{1}{2}} e^{(L+\rho/p+L\eta^{\nicefrac{1}{q}})(T-t)} (\lyaV(x))^{\nicefrac{1}{p}}  (\epsilon+\lyaPsi(t,x))^{\nicefrac{1}{q}}  \delta.
\end{align}
This completes the proof of \cref{d08}.
\end{proof}

\section{Full-history recursive multilevel Picard (MLP) approximations}\label{sec:rate}

In this section we introduce MLP approximations for solutions of stochastic fixed point equations (see \eqref{t27} in \cref{t26} in \cref{t26_0} below), we study measurability, integrability, and independence properties for the introduced MLP approximations (see \cref{ppt0}, \cref{h08}, and \cref{k08} in \cref{ppt00} below), and we establish in \cref{x01} in \cref{x010} below upper bounds for the $L^2$-distances between the exact solutions of the considered stochastic fixed point equations and the proposed MLP approximations. In our proof of \cref{x01} we employ certain function space-valued Gronwall-type inequalities, which we establish in \cref{g09-1}, \cref{g090a}, and \cref{g09} in \cref{subsec:gronwall} below. Our proof of \cref{g09-1} employs the well-known and elementary auxiliary results in \cref{lem:gronwall1}, \cref{lem:gronwall2}, and \cref{g09-2}. For completeness we include in this section also detailed proofs for \cref{lem:gronwall1}, \cref{lem:gronwall2}, and \cref{g09-2}.

\subsection{Mathematical description of MLP approximations}
\label{t26_0}

\fussy
\renewcommand{\lyaV}{\varphi}
\renewcommand{\fwpr}{Y^0}
\renewcommand{\growrate}{\rho}
\begin{setting}\label{t26}
Let $d\in \N$, 
$T, c, L,\growrate\in [0,\infty)$,  
 $\Delta=\{(t,s)\in [0,T]^2\colon  t\leq s\}$,
let
$\smallF \colon [0,T]\times \R^{d}\times\R\to\R$, $\funcG\colon\R^d\to\R $,
$\smallV \colon[0,T]\times\R^d\to\R$, and $\lyaV\colon\R^d\to(0,\infty)$
 be 
measurable, 
let
$\funcF\colon\R^{[0,T]\times\R^d}\to\R^{[0,T]\times\R^d}$
satisfy for all $t\in [0,T]$,
$x\in \R^d$,
$y,\mathfrak{y}\in \R$, $w\in \R^{[0,T]\times\R^d}$ that
\begin{align}
\left|\smallF(t,x,y)-\smallF(t,x,\mathfrak{y})\right|\leq \LipConstF|y-\mathfrak{y}|\qquad\text{and}\qquad
(\funcF(w))(t,x)= f(t,x,w(t,x)),\label{t01}
\end{align}
let
$(\Omega, \mathcal{F}, \P)$
be a probability space,
let 
$  \Theta = \bigcup_{ n \in \N }\! \Z^n$,
let $\unif^\theta\colon \Omega\to[0,1]$, $\theta\in \Theta$, be i.i.d.\ random variables, 
let $
\sppr^\theta = 
(\sppr^\theta_{t,s}(x,\omega))_{(t,s,x,\omega)\in\Delta\times\R^d\times\Omega}
\colon \Delta\times \R^d \times \Omega \to \R^d$, $\theta\in\Theta$, be 
measurable,
assume for all $t\in (0,1)$ that $\P(\unif^{0}\le t)=t$,
let $\uniform^\theta\colon [0,T]\times \Omega\to [0, T]$, $\theta\in\Theta$, satisfy 
for all $t\in [0,T]$, $\theta\in \Theta$ that 
$\uniform^\theta _t = t+ (T-t)\unif^\theta$,
assume for all $t\in[0,T]$, $s\in[t,T]$, $x\in\R^d$ that 
$\sppr^\theta_{t,s}(x)$, $\theta\in\Theta$, are i.i.d.,
assume that $( \unif^{ \theta } )_{ \theta \in\Theta}$ and $(\sppr^\theta_{t,s}(x))_{(\theta,t,s,x)\in\Theta \times \Delta\times \R^d}$ are independent,
let
$ 
  {\bigV}_{ n,M}^{\theta } \colon [0, T] \times \R^d \times \Omega \to \R
$, $n,M\in\Z$, $\theta\in\Theta$, 
satisfy 
for all $n\in \N_0$, $M\in \N$, $\theta\in\Theta $, 
$ t \in [0,T]$, $x\in\R^d $
that 
\begin{equation}
\begin{split}\label{t27}
  &{\bigV}_{n,M}^{\theta}(t,x)
=
  \frac{ \1_{ \N }( n )}{M^n}
 \sum_{i=1}^{M^n} 
      \funcG \big(\sppr^{(\theta,0,-i)}_{t,T}(x)\big)
 \\
&  +
  \sum_{\ell=0}^{n-1} \frac{(T-t)}{M^{n-\ell}}
    \left[\sum_{i=1}^{M^{n-\ell}}
      \big(\funcF({\bigV}_{\ell,M}^{(\theta,\ell,i)})-\1_{\N}(\ell)\funcF( {\bigV}_{\ell-1,M}^{(\theta,-\ell,i)})\big)
      \big(\uniform_t^{(\theta,\ell,i)},\sppr_{t,\uniform_t^{(\theta,\ell,i)}}^{(\theta,\ell,i)}(x)\big)
    \right],
\end{split}
\end{equation}
assume for all
$t\in [0,T]$, $s\in [t,T]$, $x\in \R^d$ that 
$
\E\big[|\funcG(\fwpr_{t,T}(x))|\big]+\int_t^T\E\big[| (\funcF(\smallV))(r,\fwpr_{t,r}(x))| \big]\,dr <\infty
$,
$\E\big[\lyaV(\fwpr _{t,s}(x))\big]\leq e^{\growrate(s-t)}\lyaV(x)$, $
|(\funcF(0))(t,x)|^2+|\funcG(x)|^2+ |\smallV(t,x)|^2\le c \lyaV(x)
$, and
\begin{align}
\smallV(t,x)=\E\!\left[\funcG(\fwpr_{t,T}(x))+\int_t^T \left(\funcF(\smallV)\right)\!\left(r,\fwpr_{t,r}(x)\right) dr \right].\label{t27b}
\end{align}
\end{setting}

\subsection{Measurability, integrability, and independence properties for MLP approximations}
\label{ppt00}
\sloppy
\begin{lemma}[Independence and distributional properties]
\label{ppt0}
Assume \cref{t26}. 
Then
\begin{enumerate}[(i)]
\item  \label{ppt1}
it holds 
for all $n \in \N_0$, $M \in \N$, $\theta\in\Theta $ that
$
  {\bigV}_{ n,M}^{\theta }
$ and $\funcF({\bigV}_{ n,M}^{\theta })$
are 
measurable,
\item  \label{ppt2}
it holds\footnote{Let $( \Omega, \mathcal F, \mathbb P )$ be a probability space, let $n \in \N$, 
and let $( S_k, \mathcal{S}_k )$, $k \in \{ 1, 2, ..., n \}$, be measurable spaces.
Note that for all 
$ X_k \colon \Omega \to S_k $, $ k \in \{ 1, 2, ..., n \} $,
it holds that 
$\mathfrak{S}( X_1, X_2, ..., X_n )$ is the smallest sigma-algebra 
on $\Omega$ with respect to which $X_1, X_2, ..., X_n$ are measurable.}
 for all $n \in \N_0$, $M \in \N$, $\theta \in \Theta$ that
\begin{align}\begin{split}
&\sigmaAlgebra\big( ( \bigV^{ \theta }_{ n, M }( t, x ) )_{ (t,x) \in [0,T] \times \R^d}\big)
\subseteq
  \sigmaAlgebra
\big( ( \unif^{(\theta,\vartheta)} )_{ \vartheta \in \Theta }, (  \sppr_{t,s}^{(\theta,\vartheta)}(x))_{ ( \vartheta, t, s, x ) \in \Theta \times \Delta \times \R^d } \big),
\end{split}
\end{align}

\item  \label{ppt3}
it holds 
for all $n \in \N_0$, $M \in \N$, $\theta\in\Theta $ that
$(\bigV_{ n,M}^{\theta}(t,x))_{(t,x)\in [0,T]\times\R^d}$, 
$(\sppr^\theta_{t,s}(x))_{(t,s,x)\in\Delta \times \R^d}$,  and $\unif^\theta$ are independent,
\item  \label{ppt4}
it holds
for all $n, m \in \N_0$, $M \in \N$, $i,j,k,\ell,\nu \in \Z$, $\theta \in \Theta$ with $(i,j) \neq (k,l)$ 
that
$
 (\bigV^{(\theta,i,j)}_{n,M}(t,x))_{(t,x)\in[0,T]\times\R^d}
$,
$ ( 
  \bigV^{(\theta,k,\ell)}_{m,M}(t,x))_{ (t,x)\in[0,T]\times\R^d}
$,
$ \unif^{(\theta,i,j)}$, and $(\sppr^{(\theta,i,j)}_{t,s}(x))_{(t,s,x)\in\Delta \times\R^d}$
are independent, and
\item  \label{ppt5}
it holds 
for all $n \in \N_0$, $M \in \N$, $t\in[0,T]$, $x\in\R^d$ that 
$  \bigV^\theta_{n, M}(t,x)$, $\theta \in \Theta$,
are identically distributed.
\end{enumerate}
\end{lemma}
\fussy
\begin{proof}[Proof of Lemma~\ref{ppt0}]
Note that \eqref{t27}, \eqref{t01}, the assumption that for all
$M \in \N$, $\theta\in\Theta$ it holds that $\bigV^\theta_{0, M} = 0$,
the assumption that $\smallF$ is measurable,
the assumption that $\funcG$ is measurable,
the assumption that for all $\theta \in \Theta$ it holds that $\sppr^\theta$ is measurable,
the fact that for all $\theta \in \Theta$ it holds that $\uniform^\theta$ is measurable,
and induction establish \cref{ppt1}.
Next observe that \eqref{t01}, \cref{ppt1}, and the assumption that $\smallF$ is measurable assure that
for all $n \in \N_0$, $M \in \N$, $\theta\in\Theta $ it holds that 
$\funcF(\bigV^\theta_{n, M})$ is 
measurable. The assumption that for all
$M \in \N$, $\theta\in\Theta$ it holds that $\bigV^\theta_{0, M} = 0$,
\eqref{t27}, 
the assumption that 
for all $\theta \in \Theta$ it holds that 
$\sppr^\theta$ is 
measurable,
the fact that 
for all $\theta \in \Theta$ it holds that 
$\uniform^\theta$ is 
measurable, 
and induction hence prove \cref{ppt2}.
Furthermore, note that \cref{ppt2} and the fact that 
for all $\theta \in \Theta$ it holds that
$(\unif^{(\theta, \vartheta)})_{\vartheta \in \Theta}$,
$ (\sppr^{(\theta, \vartheta)}_{t,s}(x))_{(\vartheta,t,s,x) \in \Theta\Delta \times \R^d}$,
$(\sppr^\theta_{t,s}(x))_{(t,s,x)\in\Delta \times\R^d}$, and $\unif^\theta$
are independent establish \cref{ppt3}.
In addition, note that \cref{ppt2}, the fact that
for all $i,j \in \Z$, $\theta \in \Theta$ it holds that 
$\unif^{(\theta,i,j)}$ and 
$(\sppr^{(\theta,i,j)}_{t,s}(x))_{(t,s,x)\in\Delta \times\R^d}$ are independent,
and the fact that
for all $i,j,k,\ell \in \Z$, $\theta \in \Theta$ with $(i,j) \neq (k,\ell)$ it holds that 
\begin{align}
(
\unif^{(\theta,i,j, \vartheta)},\sppr^{(\theta,i,j, \vartheta)}_{t,s}(x))_{(\vartheta,t,s,x)\in\Theta\times \Delta \times \R^d}
\end{align}
and 
\begin{align}
(
\unif^{(\theta,k,\ell, \vartheta)},\sppr^{(\theta,k,\ell, \vartheta)}_{t,s}(x))_{(\vartheta,t,s,x)\in\Theta\times \Delta\times \R^d}
\end{align}
are independent prove \cref{ppt4}.
Furthermore, note that the assumption that for all
$M \in \N, \theta\in\Theta$ it holds that $\bigV^\theta_{0, M} = 0$,
the assumption that for all $t\in[0,T]$, $s\in[t,T]$, $x\in\R^d$ it holds that $\sppr^\theta_{t,s}(x)$, $\theta \in \Theta$, are i.i.d., 
the fact that for all $t\in[0,T]$ it holds that $\uniform^\theta_t$, $\theta \in \Theta$, are i.i.d., 
\cref{ppt1,ppt2,ppt3,ppt4}, induction, and,
 e.g.,  \cite[Lemma~2.4]{HJKNW2018} establish \cref{ppt5}.
The proof of Lemma~\ref{ppt0} is thus complete.
\end{proof}

\begin{lemma}[Integrability]\label{h08}
Assume \cref{t26}, let
 $M\in\N$, and let $\dim\colon \Theta\to\N$ satisfy for all $n\in\N$, $\theta\in\Z^n$ that $\dim(\theta)=n$. 
Then 
\begin{enumerate}[(i)]
\item \label{h08b}
it holds for all $t\in[0,T]$,
$\ell\in\N_0$,  $\eta,\mu,\nu\in\Theta$ with $\min\{\dim( \eta),\dim(\mu)\}\geq \dim(\nu)$
 that
\begin{align}\begin{split}
&\sup_{x\in\R^d}\left[
(\lyaV(x))^{-1}e^{\growrate t}
\E\!\left[\left|(T-t)\left( \funcF(\bigV^{\eta}_{\ell,M}) -\1_{\N}(\ell)\funcF (\bigV^{\mu}_{\ell-1,M})\right)\!(\uniform_t^\nu, \sppr^\nu_{t,\uniform^\nu_t}(x))\right|^2\right]
\right]^{\nicefrac{1}{2}}\\
&
\leq\left[ \sup_{s\in[0,T]}\sup_{x\in\R^d}
\Bigl(
\1_{\{0\}}(\ell){ (T-t)}
(\lyaV(x))^{-\nicefrac{1}{2}}e^{\growrate s/2}|(F(0))(s,x)|
\Bigr)
\right]\\
&
+
\left[
\int_{t}^{T}\sup_{r\in[s,T]}\sup_{x\in\R^d}\left[
\1_{\N}(\ell){ (T-t)^{\nicefrac{1}{2}}}\LipConstF
(\lyaV(x))^{-1}e^{\growrate r}
\E\!\left[
\left|
\bigV^{\eta}_{\ell,M}(r,x) -\bigV^{\nu}_{\ell-1,M}(r,x)\right|^2\right]\right]ds\right]^{\!\nicefrac{1}{2}},
\end{split}\end{align}
\item \label{h12d} it holds for all $\theta\in\Theta$ that
\begin{align}
\begin{split}
\sup_{t\in[0,T]}
\sup_{x\in\R^d}\left[
(\lyaV(x))^{-1}e^{\growrate t}\E\!\left[|
g(\sppr^\theta_{t,T}(x))|^2\right]\right ]\leq 
\sup_{x\in\R^d}\left[(\lyaV(x))^{-1}|g(x)|^2\right]e^{\growrate T}
<\infty,
\end{split}
\end{align}
\item\label{h08f} it holds for all  $n\in\N_0$, $\theta\in\Theta$, $\eta,\mu,\nu\in\Theta$ with  $\min\{\dim( \eta),\dim(\mu)\}\geq \dim(\nu)$
 that
\begin{align}\begin{split}
\sup_{t\in[0,T]}
\sup_{x\in\R^d}\left[
\frac{e^{\growrate t}}{\lyaV(x)}
\E\!\left[\bigl|(T-t)( \funcF(\bigV^{\eta}_{n,M}) -\1_{\N}(n)\funcF (\bigV^{\mu}_{n-1,M}))(\uniform_t^\nu, \sppr^\nu_{t,\uniform_t^\nu}(x))\bigr|^2\right]\right]^{\nicefrac{1}{2}}<\infty,
\end{split}\end{align}
\item\label{h08fb} it holds for all  $n\in\N_0$, $\theta\in\Theta$
 that
\begin{align}
\sup_{t\in[0,T]}\sup_{x\in\R^d}
\left[
(\lyaV(x))^{-1}e^{\growrate t}
\E\!\left[\left|\bigV^\theta_{n,M}(t,x)\right|^2\right]\right]^{\!\nicefrac{1}{2}}<\infty,
\end{align}
and
\item\label{h08g} it holds for all  $n\in\N_0$,  $\eta,\nu\in\Theta$ with  $\dim( \eta)\geq \dim(\nu)$
 that
\begin{align}\begin{split}
&
\sup_{t\in[0,T]}
\sup_{x\in\R^d}\left[
\frac{e^{\growrate t}}{\lyaV(x)}
\E\!\left[\bigl|(T-t)( \funcF(\bigV^{\eta}_{n,M}) )(\uniform_t^\nu, \sppr^\nu_{t,\uniform_t^\nu}(x))\bigr|^2\right]\right]^{\nicefrac{1}{2}}
\\
&=\sup_{t\in[0,T]}
\sup_{x\in\R^d}\left[
\frac{e^{\growrate t}(T-t)}{\lyaV(x)}\int_t^T
\E\!\left[\bigl|( \funcF(\bigV^{\eta}_{n,M}) )(s, \sppr^\nu_{t,s}(x))\bigr|^2\right]ds\right]^{\nicefrac{1}{2}}
<\infty.
\end{split}\end{align}

\end{enumerate}
\end{lemma}
\begin{proof}[Proof of \cref{h08}]
\sloppy
Observe that \cref{ppt2} in \cref{ppt0} and the assumption
that 
$( \unif^{ \theta } )_{ \theta \in\Theta}$ and $(\sppr^\theta_{t,s}(x))_{(\theta,t,s,x)\in\Theta \times \Delta\times \R^d}$ are independent
show that for all  $\ell\in\N_0$, $\eta,\mu,\nu\in\Theta$ with  $\min\{\dim( \eta),\allowbreak \dim(\mu)\}\geq \dim(\nu)$ it holds that
$((\funcF(\bigV_{\ell,M}^\eta)-\1_\N(\ell)\funcF(\bigV_{\ell-1,M}^\mu))(t,x))_{(t,x)\in[0,T]\times \R^d}$, $\unif^\nu$, and $(\sppr^\nu_{t,s}(x))_{(t,s,x)\in\Delta \times\R^d}$ are independent.
Combining \cref{ppt1} in \cref{ppt0}, the assumption 
that for all $\theta\in\Theta$ it holds that $\sppr^\theta$ is measurable, 
the fact that for all $\nu\in\Theta$, $r\in (0 ,1)$ it holds that $\P( \unif^\nu\le r)=r$,
and the fact that for all $t\in [0,T]$, $s\in [t,T]$, $x\in \R^d$ it holds that 
$\E\big[\lyaV(\fwpr _{t,s}(x))\big]\leq e^{\growrate(s-t)}\lyaV(x)$ with, e.g., \cite[Lemma 2.2]{HJKNW2018} therefore implies that for all $\ell\in\N_0$, $t\in[0,T)$, $x\in\R^d$,  $\eta,\mu,\nu\in\Theta$ with $\min\{\dim( \eta),\dim(\mu)\}\geq \dim(\nu)$ it holds that
\begin{align}
\begin{split}
& \E\!\left[\left|(T-t)\left( \funcF(\bigV^{\eta}_{\ell,M}) -\1_{\N}(\ell)\funcF (\bigV^{\mu}_{\ell-1,M})\right)(\uniform_t^\nu, \sppr^\nu_{t,\uniform^\nu_t}(x))\right|^2\right]\\
&=\frac{1}{T-t}\int_{t}^{T}\E\!\left[\left.\E\!\left[\left|(T-t)\left( \funcF(\bigV^{\eta}_{\ell,M}) -\1_{\N}(\ell)\funcF (\bigV^{\mu}_{\ell-1,M})\right)(s, z)\right|^2\right]\right|_{z=\sppr^\nu_{t,s}(x)}\right]ds\\
&=(T-t)\int_{t}^{T}\E\!\left[\left.\E\!\left[\left|\left( \funcF(\bigV^{\eta}_{\ell,M}) -\1_{\N}(\ell)\funcF (\bigV^{\mu}_{\ell-1,M})\right)(s, z)\right|^2\right]\right|_{z=\sppr^\nu_{t,s}(x)}\right]ds,
\end{split}
\end{align}
\begin{align}\label{h09}
\E\!\left[\bigl|(T-t)( \funcF(\bigV^{\eta}_{n,M}) )(\uniform_t^\nu, \sppr^\nu_{t,\uniform_t^\nu}(x))\bigr|^2\right]={(T-t)}\int_t^T
\E\!\left[\bigl|( \funcF(\bigV^{\eta}_{n,M}) )(s, \sppr^\nu_{t,s}(x))\bigr|^2\right]ds,
\end{align}
and
\begin{align}\begin{split}
& \E\!\left[\left|(T-t)\left( \funcF(\bigV^{\eta}_{\ell,M}) -\1_{\N}(\ell)\funcF (\bigV^{\mu}_{\ell-1,M})\right)(\uniform_{t}^\nu, \sppr^\nu_{t,\uniform^\nu_{t}}(x))\right|^2\right]\\
&\leq (T-t)\int_{t}^{T}\E\!\left[\left(
\sup_{r\in[s,T],z\in\R^d}\biggl[ \frac{e^{\growrate r}}{\lyaV(z)}\E\!\left[\left|
 \funcF(\bigV^{\eta}_{\ell,M}) -\1_{\N}(\ell)\funcF (\bigV^{\mu}_{\ell-1,M})(r,z)\right|^2\right]
\biggr]\right)
\frac{\lyaV(\sppr^\nu_{t,s}(x))}{e^{\growrate s}} \right]ds\\
&\leq (T-t)\int_{t}^{T}\left(
\sup_{r\in[s,T],z\in\R^d}\biggl[ \frac{e^{\growrate r}}{\lyaV(z)}\E\!\left[\left|
 \funcF(\bigV^{\eta}_{\ell,M}) -\1_{\N}(\ell)\funcF (\bigV^{\mu}_{\ell-1,M})(r,z)\right|^2\right]
\biggr]\right)
\frac{e^{\growrate (s-t)}
\lyaV(x)}{
e^{\growrate s}}\,ds
\\&= e^{-\growrate t}
\lyaV(x)
(T-t)^2
\left[\frac{1}{T-t}\int_{t}^{T}
\sup_{r\in[s,T],z\in\R^d}\biggl[\frac{e^{\growrate r}}{\lyaV(z)}\E\!\left[\left|
 \funcF(\bigV^{\eta}_{\ell,M}) -\1_{\N}(\ell)\funcF (\bigV^{\mu}_{\ell-1,M})(r,z)\right|^2\right]
\biggr]
\,ds\right].
\end{split}\end{align}
This, the fact that $\forall\,\eta\in\Theta,t\in[0,T],x\in\R^d\colon\bigV^{\eta}_{0,M}(t,x)=0 $, and
 \eqref{t01}  imply that
for all $\ell\in\N_0$, $t\in[0,T)$,  $\eta,\mu,\nu\in\Theta$ with $\min\{\dim( \eta),\dim(\mu)\}\geq \dim(\nu)$ it holds that
\begin{align}\begin{split}
&\sup_{x\in\R^d}\biggl[
(\lyaV(x))^{-1}e^{\growrate t}
\E\!\left[\left|(T-t)\left( \funcF(\bigV^{\eta}_{\ell,M}) -\1_{\N}(\ell)\funcF (\bigV^{\mu}_{\ell-1,M})\right)\!(\uniform_t^\nu, \sppr^\nu_{t,\uniform^\nu_t}(x))\right|^2\right]
\biggr]^{\nicefrac{1}{2}}
\\
&
\leq{ (T-t)}\left[ \sup_{s\in[0,T]}\sup_{x\in\R^d}
\Bigl(
\1_{\{0\}}(\ell)
(\lyaV(x))^{-\nicefrac{1}{2}}e^{\growrate s/2}|(F(0))(s,x)|
\Bigr)
\right]\\
&
+
{ (T-t)}
\left[
\frac{1}{T-t}
\int_{t}^{T}\sup_{r\in[s,T]}\sup_{x\in\R^d}\left[
\1_{\N}(\ell)\LipConstF
(\lyaV(x))^{-1}e^{\growrate r}
\E\!\left[
\left|
(\bigV^{\eta}_{\ell,M} -\bigV^{\nu}_{\ell-1,M})(r,x)\right|^2\right]\right]ds\right]^{\!\nicefrac{1}{2}}
\\
&
={ (T-t)}\left[ \sup_{s\in[0,T]}\sup_{x\in\R^d}
\Bigl(
\1_{\{0\}}(\ell)
(\lyaV(x))^{-\nicefrac{1}{2}}e^{\growrate s/2}|(F(0))(s,x)|
\Bigr)
\right]\\
&
+
{ (T-t)^{\nicefrac{1}{2}}}
\left[
\int_{t}^{T}\sup_{r\in[s,T]}\sup_{x\in\R^d}\left[
\1_{\N}(\ell)\LipConstF
(\lyaV(x))^{-1}e^{\growrate r}
\E\!\left[
\left|
(\bigV^{\eta}_{\ell,M} -\bigV^{\nu}_{\ell-1,M})(r,x)\right|^2\right]\right]ds\right]^{\!\nicefrac{1}{2}}
.
\end{split}\end{align}
This establishes \cref{h08b}.
Next observe that
the fact that for all $t\in[0,T] $, $s\in[t,T]$, $x\in\R^d$ it holds that
$\sppr^\theta_{t,s}(x)$, $\theta\in\Theta$, are identically distributed and the fact that for all $t\in [0,T]$, $s\in [t,T]$, $x\in \R^d$ it holds that
$\E\big[\lyaV(\fwpr _{t,s}(x))\big]\leq e^{\growrate(s-t)}\lyaV(x)$ imply that for all
$\theta\in\Theta$,
 $t\in[0,T]$, $x\in\R^d$ it holds that
\begin{align}\begin{split}
\E\!\left[\left|
g(\sppr^\theta_{t,T}(x))\right|^2\right]&=
\E\!\left[\left|
g(\sppr^0_{t,T}(x))\right|^2\right]\leq
\E\!\left[ \sup_{z\in\R^d}\left[{|g(z)|^2}/{\lyaV(z)}\right]
\left|
\lyaV(\sppr^0_{t,T}(x))\right|\right]\\&\leq 
\sup_{z\in\R^d}\left[{|g(z)|^2}/{\lyaV(z)}\right]e^{\growrate(T-t)}\lyaV(x).\end{split}
\end{align}
This 
and the fact that for all $y\in \R^d$ it holds that $|g(y)|^2\le C\lyaV(y)$ imply that for all $\theta\in\Theta$ it holds that
\begin{align}\label{h12c}
\begin{split}
\sup_{t\in[0,T]}
\sup_{x\in\R^d}\left[
(\lyaV(x))^{-1}e^{\growrate t}(\E\!\left[\left|
g(\sppr^\theta_{t,T}(x))\right|^2\right]\right ]\leq 
\sup_{z\in\R^d}\left[{|g(z)|^2}/{\lyaV(z)}\right]e^{\growrate T}
<\infty. 
\end{split}\end{align}
This establishes \cref{h12d}.
In the next step we prove \cref{h08f,h08fb} by induction on $ n \in \N_0 $. 
The fact that
$\forall\, \theta\in\Theta\colon \bigV^\theta_{0,M}=\bigV^\theta_{-1,M}=0$, \cref{h08b}, and the fact that for all $t\in [0,T]$, $x\in \R^d$ it holds that $|( \funcF(0))(t,x)|^2\le C\lyaV(x)$
show that for all $\nu\in\Theta$ it holds that
\begin{align}\small\begin{split}
\sup_{t\in[0,T]}
\sup_{x\in\R^d}\Bigg[
(\lyaV(x))^{-1}e^{\growrate t}
\E\!\left[\left|(T-t)( \funcF(0))(\uniform_t^\nu, \sppr^\nu_{t,\uniform_t^\nu}(x))\right|^2\right]\Bigg]^{\!\nicefrac{1}{2}}<\infty.
\end{split}\end{align}
This establishes \cref{h08f,h08fb} in the base case $n = 0$. For the induction step $\N_0\ni (n-1)\induct n \in\N$
let $n \in \N$ satisfy for all
 $\ell\in [0,n-1]\cap\N_0$,  $\theta\in\Theta$, $\eta,\mu,\nu\in\Theta$ with  $\min\{\dim( \eta),\dim(\mu)\}\geq \dim(\nu)$
that 
\begin{align}\small\begin{split}
\sup_{t\in[0,T]}
\sup_{x\in\R^d}\Bigg[
(\lyaV(x))^{-1}e^{\growrate t}
\E\!\left[\left|(T-t)( \funcF(\bigV^{\eta}_{\ell,M}) -\1_{\N}(\ell)\funcF (\bigV^{\mu}_{\ell-1,M}))(\uniform_t^\nu, \sppr^\nu_{t,\uniform_t^\nu}(x))\right|^2\right]\Bigg]^{\!\nicefrac{1}{2}}<\infty\label{h12}
\end{split}\end{align}and
\begin{align}
\sup_{t\in[0,T]}\sup_{x\in\R^d}
\left[
(\lyaV(x))^{-1}e^{\growrate t}
\E\!\left[\left|\bigV^\theta_{\ell,M}(t,x)\right|^2\right]\right]^{\!\nicefrac{1}{2}}<\infty.\label{h12b}
\end{align} 
Observe that the triangle inequality, \eqref{t27}, \cref{h12d}, and \eqref{h12}  imply that for all
$\theta\in\Theta$ it holds that
\begin{align}\label{h13}
\sup_{t\in[0,T]}
\sup_{x\in\R^d}
\left[
(\lyaV(x))^{-1}e^{\growrate t}
\E\!\left[\left|\bigV^\theta_{n,M}(t,x)\right|^2\right]\right]^{\!\nicefrac{1}{2}}<\infty.\end{align}
Combining \cref{h08b} and
 \eqref{h12b} with the fact that $n\in\N$
hence shows that
for all 
$\eta,\mu,\nu\in\Theta$ with  $\min\{\dim( \eta),\dim(\mu)\}\geq \dim(\nu)$ it holds that
\begin{align}\small\begin{split}
&\sup_{t\in[0,T]}\sup_{x\in\R^d}\biggl[
e^{\growrate t}
(\lyaV(x))^{-1}
\E\!\left[\left|(T-t)\left( \funcF(\bigV^{\eta}_{n,M}) -\1_{\N}(n)\funcF (\bigV^{\mu}_{n-1,M})\right)\!(\uniform_t^\nu, \sppr^\nu_{t,\uniform^\nu}(x))\right|^2\right]
\biggr]^{\nicefrac{1}{2}}\\
&\leq T\LipConstF
\left[\sup_{s\in[0,T]}\sup_{x\in\R^d}\left[(\lyaV(x))^{-1}e^{\growrate s}
\E\!\left[
\left|(
\bigV^{\eta}_{n,M} -\bigV^{\nu}_{n-1,M})(s,x)\right|^2\right]\right]\right]^{\!\nicefrac{1}{2}}<\infty.
\end{split}\end{align}
Induction, \eqref{h12b}, and \eqref{h13} hence establish \cref{h08f,h08fb}.
Next observe that the triangle inequality and \eqref{h12} ensure that for all $n\in\N_0$,  $\eta,\nu\in\Theta$ with $\dim( \eta)\geq \dim(\nu)$ it holds that
\begin{align}\begin{split}
&
\sup_{t\in[0,T]}
\sup_{x\in\R^d}\left[
\frac{e^{\growrate t}}{\lyaV(x)}
\E\!\left[\bigl|(T-t)( \funcF(\bigV^{\eta}_{n,M}) )(\uniform_t^\nu, \sppr^\nu_{t,\uniform_t^\nu}(x))\bigr|^2\right]\right]^{\nicefrac{1}{2}}
\\
&\leq 
\sum_{l=0}^{n}\sup_{t\in[0,T]}
\sup_{x\in\R^d}\left[
\tfrac{e^{\growrate t}(T-t)^2}{\lyaV(x)}
\E\!\left[\bigl|( \funcF(\bigV^{\eta}_{l,M}) -
\1_{\N}(l)\funcF(\bigV^{\eta}_{l-1,M}) ))
(\uniform_t^\nu, \sppr^\nu_{t,\uniform_t^\nu}(x))\bigr|^2\right]\right]^{\nicefrac{1}{2}}<\infty.
\end{split}\end{align}
This and \eqref{h09} establish \cref{h08g}.
The proof of \cref{h08} is thus complete.
\end{proof}

\fussy
%


\begin{lemma}[Expectations of approximations]\label{k08}Assume \cref{t26} and let $\theta\in\Theta$. Then
\begin{enumerate}[(i)]
\item\label{k05} it holds for all  $\ell\in\N_0$, $t\in[0,T]$, $x\in\R^d $ that
$
\big(\funcF({\bigV}_{\ell,M}^{({\theta},\ell,i)})-\1_{\N}(\ell)\funcF( {\bigV}_{\ell-1,M}^{({\theta},-\ell,i)})\big)
      \big(\uniform_t^{({\theta},\ell,i)},\\ \sppr_{t,\uniform_t^{({\theta},\ell,i)}}^{({\theta},\ell,i)}(x)\big)$,
       $i\in\N$, are i.i.d.\ and
\item\label{k07} it holds for all $n\in\N$, $t\in[0,T]$, $x\in\R^d$ that
\begin{equation}\label{eq:expectations}
\begin{split}
  \E \!\left[{\bigV}_{n,M}^{{\theta}}(t,x)\right]
&= \E\!\left[\funcG\left( \sppr^{\theta}_{t,T}(x)\right)\right]+
(T-t) \E\!\left[ \left(\funcF( \bigV_{n-1,M}^{\theta})\right)\!\left(\uniform_t^\theta,
\sppr^{\theta}_{t,\uniform_t^\theta}(x)\right)\right]\\
&= \E\!\left[\funcG\left( \sppr^{\theta}_{t,T}(x)\right)\right]+
\int_{t}^{T} \E\!\left[ \left(\funcF( \bigV_{n-1,M}^{\theta})\right)\!\!\left(s,
\sppr^\theta_{t,s}(x)\right)\right]ds.
\end{split}
\end{equation}
\end{enumerate}
\end{lemma}
\begin{proof}[Proof of \cref{k08}]
Observe that
\cref{ppt2} in \cref{ppt0} shows that for all 
$i\in\N$, $\ell\in\N_0$ it holds that
\begin{align}\begin{split}
&\sigmaAlgebra\!\left(
({\bigV}_{\ell,M}^{({\theta},\ell,i)}(t,x))_{(t,x)\in [0,T]\times \R^d}
\right)
\subseteq  
\sigmaAlgebra\!\left(
(\unif^{(\theta, \ell,i,\vartheta)})_{\vartheta\in\Theta},  (\sppr^{(\theta, \ell,i,\vartheta)}_{s,t}(x))_{(\vartheta, s,t,x)\in\Theta\times \Delta \times \R^d}
\right)
\end{split}\label{k03}
\end{align}
and
\begin{align}\begin{split}\label{k04}
&\sigmaAlgebra\!\left(
({\bigV}_{\ell-1,M}^{({\theta},-\ell,i)}(t,x))_{(t,x)\in [0,T]\times \R^d}
\right)
\subseteq  \sigmaAlgebra\!\left(
(\unif^{(\theta,- \ell,i,\vartheta)})_{\vartheta\in\Theta},
(\sppr^{(\theta,- \ell,i,\vartheta)}_{s,t}(x))_{(\vartheta,s,t,x)\in\Theta\times \Delta\times \R^d}\right).
\end{split}
\end{align}
Combining the fact that $\unif^\nu$, $\nu\in \Theta$, are independent and the fact that
for all $t\in[0,T]$, $s\in[t,T]$, $x\in\R^d$ it holds that
$\sppr^\nu_{t,s}(x)$, $\nu\in\Theta$, are independent
hence assures
  that
 for all $\ell\in\N_0$, $t\in[0,T]$, $x\in\R^d$ it holds that
$
\bigl(\funcF\bigl({\bigV}_{\ell,M}^{({\theta},\ell,i)}\bigr)-\1_{\N}(\ell)\funcF\bigl( {\bigV}_{\ell-1,M}^{({\theta},-\ell,i)}\bigr)\bigr)
      \bigl(\uniform_t^{({\theta},\ell,i)},\sppr_{t,\uniform_t^{({\theta},\ell,i)}}^{({\theta},\ell,i)}(x)\bigr)$, $i\in\N$, are independent.
Moreover, observe that \eqref{k04}, the fact that $\unif^\nu$, $\nu\in \Theta$, are independent, the fact that
for all $t\in[0,T]$, $s\in[t,T]$, $x\in\R^d$ it holds that
$\sppr^\nu_{t,s}(x)$, $\nu\in\Theta$, are independent, 
and the fact that $ (\unif^{ \nu } )_{ \nu \in\Theta}$ and $(\sppr^\nu_{t,s}(x))_{(\nu,t,s,x)\in\Theta \times \Delta\times \R^d}$ are independent
demonstrate that
 for all $\ell\in\N_0$, $i\in\N$, $t\in[0,T]$, $x\in\R^d$ it holds that
$
{\bigV}_{\ell,M}^{({\theta},\ell,i)}$, ${\bigV}_{\ell-1,M}^{({\theta},-\ell,i)}$, $\unif^{({\theta},\ell,i)}$, and $\sppr^{({\theta},\ell,i)}$ are independent.
\sloppy
Combining this and, e.g., the disintegration-type result in \cite[Lemma~2.2]{HJKNW2018}
establishes \cref{k05}.
Next note that \cref{ppt3} in \cref{ppt0} implies that
 for all
$\ell\in\N_0$,
$ k \in \{ 0, 1 \} $
it holds
that 
$(
\bigV_{\ell-k,M}^{\theta}(t,x))_{(t,x)\in[0,T]\times\R^d}$, $\unif^{\theta}$, 
and $( \sppr^{\theta}_{t,s}(x))_{(t,s,x)\in \{(r,u)\in [0,T]^2: u\in [r,T]\}\times \R^d}$ are independent.
Combining 
the fact that for all $\ell\in\N_0$, $i\in\N$, $t\in[0,T]$, $s\in[t,T]$, $x\in\R^d$ it holds that $\bigV^{(\theta,\ell,i)}(t,x)$ and $\bigV^{\theta}(t,x)$ are identically distributed (see \cref{ppt5} in \cref{ppt0}),
the fact that for all $\ell\in\N_0$, $i\in\N$ it holds that 
 $\unif^{(\theta,\ell,i)}$ and $\unif^{\theta}$ are identically distributed, 
 the fact that for all $\ell\in\N_0$, $i\in\N$, $t\in[0,T]$, $s\in[t,T]$, $x\in\R^d$ it holds that
  $\sppr^{(\theta,\ell,i)}_{t,s}(x)$ and $\sppr^{\theta}_{t,s}(x)$ are identically distributed, \cref{h08f} in \cref{h08}, and, 
e.g.,
the disintegration-type result in \cite[Lemma~2.2]{HJKNW2018}
 hence proves that for all $i\in\N$, $\ell\in\N_0$, $t\in [0,T]$, $x\in\R^d$ it holds that
\begin{align}\small\begin{split}
&\E\!\left[
      \left(\funcF\bigl({\bigV}_{\ell,M}^{({\theta},\ell,i)}\bigr)-\1_{\N}(\ell)\funcF\bigl( {\bigV}_{\ell-1,M}^{({\theta},-\ell,i)}\bigr)\right)\!\!
      \left(\uniform_t^{({\theta},\ell,i)},\sppr_{t,\uniform_t^{({\theta},\ell,i)}}^{({\theta},\ell,i)}(x)\right)\right]\\
&
=\E\!\left[
      \left(\funcF\bigl({\bigV}_{\ell,M}^{({\theta},\ell,i)}\bigr)\right)\!\!
      \left(\uniform_t^{({\theta},\ell,i)},\sppr_{t,\uniform_t^{({\theta},\ell,i)}}^{({\theta},\ell,i)}(x)\right)\right]
-\1_{\N}(\ell)\E\!\left[
      \left(\funcF\bigl( {\bigV}_{\ell-1,M}^{({\theta},-\ell,i)}\bigr)\right)\!\!
      \left(\uniform_t^{({\theta},\ell,i)},\sppr_{t,\uniform_t^{({\theta},\ell,i)}}^{({\theta},\ell,i)}(x)\right)\right]\\
&=\E\!\left[
      \left(\funcF({\bigV}_{\ell,M}^{\theta})\right)\!\!
      \left(\uniform_t^{\theta},\sppr_{t,\uniform_t^{\theta}}^{\theta}(x)\right)\right]
-\1_{\N}(\ell)
\E\!\left[
      \left(\funcF\bigl( {\bigV}_{\ell-1,M}^{\theta}\bigr)\right)\!\!
      \left(\uniform_t^{\theta},\sppr_{t,\uniform_t^{\theta}}^{\theta}(x)\right)\right].
\end{split}
\end{align}
The assumption that for all $t\in[0,T] $, $s\in[t,T]$, $x\in\R^d$ it holds that
$\sppr^\vartheta_{t,s}(x)$, $\vartheta\in\Theta$, are identically distributed, 
\cref{h08f} in \cref{h08},
\cref{ppt3} in \cref{ppt0},
the fact that
for all $t\in[0,T)$ it holds that $\uniform_t$ is continuous uniformly distributed on $[t,T]$,
and, e.g., the disintegration-type result in \cite[Lemma~2.2]{HJKNW2018}
 therefore imply that for all $n\in\N$, $t\in[0,T]$, $x\in\R^d$ it holds that
\begin{equation}
\begin{split}
  &\E \!\left[{\bigV}_{n,M}^{{\theta}}(t,x)\right]
=
  \frac{1}{M^n}
 \sum_{i=1}^{M^n} 
     \E\!\left[ \funcG \!\left(\sppr^{({\theta},0,-i)}_{t,T}(x)\right)\right]
 \\
& \quad  +
  \sum_{\ell=0}^{n-1} \frac{(T-t)}{M^{n-\ell}}
    \left[\sum_{i=1}^{M^{n-\ell}}
\E\!\left[
      \left(\funcF\bigl({\bigV}_{\ell,M}^{({\theta},\ell,i)}\bigr)-\1_{\N}(\ell)\funcF\bigl( {\bigV}_{\ell-1,M}^{({\theta},-\ell,i)}\bigr)\right)\!\!
      \left(\uniform_t^{({\theta},\ell,i)},\sppr_{t,\uniform_t^{({\theta},\ell,i)}}^{({\theta},\ell,i)}(x)\right)\right]
    \right]\\
\end{split}
\end{equation}and
\begin{align}\begin{split}
&\E \!\left[{\bigV}_{n,M}^{{\theta}}(t,x)\right]
=
     \E\!\left[ \funcG \!\left(\sppr^{{\theta}}_{t,T}(x)\right)\right]\\
&\quad
    +(T-t)
      \sum_{\ell=0}^{n-1}\Biggl[
\E\!\left[
      \left(\funcF({\bigV}_{\ell,M}^{{\theta}})\right)\!\!
      \left(\uniform_t^{\theta},\sppr_{t,\uniform_t^{\theta}}^{\theta}(x)\right)\right]
-\1_{\N}(\ell)
\E\!\left[
      \left(\funcF\bigl( {\bigV}_{\ell-1,M}^{\theta}\bigr)\right)\!\!
      \left(\uniform_t^{\theta},\sppr_{t,\uniform_t^{\theta}}^{\theta}(x)\right)\right]\Biggr]\\
&= \E\!\left[\funcG\left( \sppr^{\theta}_{t,T}(x)\right)\right]+
(T-t) \E\!\left[ \left(\funcF( \bigV_{n-1,M}^{\theta})\right)\!\!\left(\uniform_t^\theta,
\sppr^{\theta}_{t,\uniform_t^\theta}(x)\right)\right]\\
&=\E\!\left[\funcG\left( \sppr^{\theta}_{t,T}(x)\right)\right]+
\int_{t}^{T} \E\!\left[ \left(\funcF( \bigV_{n-1,M}^{\theta})\right)\!\!\left(s,
\sppr^{\theta}_{t,s}(x)\right)\right]ds.\end{split}
\end{align}
This establishes \cref{k07}. The proof of \cref{k08} is thus complete.
\fussy
\end{proof}
\subsection{Recursive error bounds for MLP approximations}

\begin{lemma}[Error recursion]\label{k01}Assume \cref{t26} and let $n,M\in\N$,  $t\in[0,T]$. Then
\begin{align}\begin{split}\label{k01_b}
&\sup_{x\in\R^d}
\Bigl[
e^{\growrate t}\lyaV(x)^{-1}\E\!\left[\left|{\bigV}_{n,M}^{0}(t,x)-\smallV(t,x)\right|^2\right]\Bigr]^{\!\nicefrac{1}{2}}\\&\leq
\frac{2e^{\growrate T/2}}{\sqrt{M^n}}\sup_{s\in[0,T]}\sup_{z\in\R^d}\left[\frac{
\max\{|T(\funcF(0))(s,z)|,|g(z)|\}}{\sqrt{\lyaV(z)}}\right]\\
 &\quad+
\sum_{\ell=0}^{n-1}\left[
\frac{2(T-t)^{\nicefrac{1}{2}}\LipConstF}{\sqrt{M^{n-\ell-1}}}
\left(
\int_{t}^{T}
\sup_{r\in[s,T]}\sup_{x\in\R^d}\biggl[(\lyaV(x))^{-1}e^{\growrate r}\E\Bigl[
\bigl|
\bigV^{0}_{\ell,M}(r,x) -\smallV(r,x)|^2\Bigr]
\biggr]ds\right)^{\!\nicefrac{1}{2}}\right].
\end{split}\end{align}
\end{lemma}
\begin{proof}[Proof of \cref{k01}]
Observe that
the triangle inequality, \eqref{t27},
Bienaym\'e's identity,
the fact that for all
$x\in\R^d$ it holds that
$\sppr^\theta_{t,T}(x)$, $\theta\in\Theta$, are i.i.d., and
\cref{k05} in \cref{k08}
  imply that for all $x\in\R^d$ it holds\footnote{Note that for every probability space $(\Omega, \cF,\P)$ and every random variable $X \colon \Omega \to \R$ with $\E[ | X | ] < \infty$ it holds that $\var( X ) = \E[ | X - \E[X] |^2 ]$.} that
\begin{align}\begin{split}
&\Bigl(\var\!\left({\bigV}_{n,M}^{{0}}(t,x)\right)\Bigr)^{\!\nicefrac{1}{2}}\leq 
\left(\var\!
\left(
\frac{1}{M^n}
 \sum_{i=1}^{M^n} 
      \funcG \!\left(\sppr^{({0},0,-i)}_{t,T}(x)\right)\right)\right)^{\!\!\nicefrac{1}{2}}\\
&+\sum_{\ell=0}^{n-1} \left(
\var \!\left(
\frac{1}{M^{n-\ell}}
    \sum_{i=1}^{M^{n-\ell}}(T-t)
      \left(\funcF\bigl({\bigV}_{\ell,M}^{({0},\ell,i)}\bigr)-\1_{\N}(\ell)\funcF\bigl( {\bigV}_{\ell-1,M}^{({0},-\ell,i)}\bigr)\right)\!\!
      \left(\uniform_t^{({0},\ell,i)},\sppr_{t,\uniform_t^{({0},\ell,i)}}^{({0},\ell,i)}(x)\right)
    \right)\right)^{\!\!\nicefrac{1}{2}}\\
&\leq \frac{1}{\sqrt{M^n}}\Bigl(\E\!
\left[ \left|\funcG\!\left(\sppr^{0}_{t,T}(x)\right)\right|^2\right]\Bigr)^{\!\nicefrac{1}{2}}
\\&+
\sum_{\ell=0}^{n-1}
\frac{1}{\sqrt{M^{n-\ell}}}\left(
\E \!\left[\left|(T-t)
      \left(\funcF\bigl({\bigV}_{\ell,M}^{({0},\ell,1)}\bigr)-\1_{\N}(\ell)\funcF\bigl( {\bigV}_{\ell-1,M}^{({0},-\ell,1)}\bigr)\right)\!\!
      \left(\uniform_t^{({0},\ell,1)},\sppr_{t,\uniform_t^{({0},\ell,1)}}^{({0},\ell,1)}(x)\right)\right|^2
    \right]\right)^{\!\nicefrac{1}{2}}.
\end{split}
\end{align}
\Cref{h08} (applied for every $\ell\in [0,n-1]\cap\N$ with $\eta\defeq (0,\ell,1)$, $\mu\defeq (0,-\ell,1)$, 
$\nu\defeq(0,\ell,1)$  in the notation of \cref{h08}) hence shows that for all $x\in\R^d$ it holds that
\begin{align}\begin{split}
\Bigl(\var\!\left({\bigV}_{n,M}^{{0}}(t,x)\right)\Bigr)^{\!\nicefrac{1}{2}}
&\leq \frac{2e^{\growrate (T-t)/2}\sqrt{\lyaV(x)}}{\sqrt{M^n}}
\sup_{s\in[0,T]}\sup_{z\in\R^d}\left[\frac{\max\{|T(\funcF(0))(s,z)|,|g(z)|\}}{\sqrt{\lyaV(z)}}\right]
\\&+
\sum_{\ell=1}^{n-1}
\frac{e^{-\growrate t/2}\sqrt{\lyaV(x)}}{\sqrt{M^{n-\ell}}}
(T-t)^{\nicefrac{1}{2}}\LipConstF\\& 
 \cdot
\left(\int_{t}^{T}\sup_{r\in[s,T]}\sup_{z\in\R^d}\biggl[\frac{e^{\growrate r}}{\lyaV(z)}\E\!\left[
\bigl|\bigV^{({0},\ell,1)}_{\ell,M}(r,z) -\bigV^{({0},-\ell,1)}_{\ell-1,M}(r,z)\bigr|^2\right]
\biggr]ds\right)^{\!\!\nicefrac{1}{2}}.
\end{split}\label{k02}
\end{align}
Next note that \cref{ppt5} in \cref{ppt0} and the triangle inequality 
demonstrate that for all $\ell\in \N$, $\eta,\nu\in\Theta$, $s\in[0,T]$, $x\in\R^d$ it holds that
\begin{align}\begin{split}
&\left(\E\!\left [\left|\bigV_{\ell,M}^\eta(s,x)-\bigV_{\ell-1,M}^\nu(s,x)\right|^2\right]\right)^{\!\nicefrac{1}{2}}\\&\leq \left(
\E\!\left [\left|\bigV_{\ell,M}^\eta(s,x)-\smallV(s,x)\right|^2\right]
\right)^{\!\nicefrac{1}{2}}
+\left(
\E\!\left [\left|\bigV_{\ell-1,M}^\nu(s,x)-\smallV(s,x)\right|^2\right]\right)^{\!\nicefrac{1}{2}}\\
&= \left(
\E\!\left [\left|\bigV_{\ell,M}^0(s,x)-\smallV(s,x)\right|^2\right]\right)^{\!\nicefrac{1}{2}}+\left(
\E\!\left [\left|\bigV_{\ell-1,M}^0(s,x)-\smallV(s,x)\right|^2\right]\right)^{\!\nicefrac{1}{2}}.
\end{split}\end{align}
This, \eqref{k02}, and the fact that for all
$a_0,a_1,\ldots,a_n\in [0,\infty]$ it holds that  $\sum_{\ell=1}^{n-1}(a_\ell+a_{\ell-1})\leq \sum_{\ell=0}^{n-1}[(2-\1_{\{n-1\}}(\ell)) a_\ell]
$
imply that for all $x\in\R^d$ it holds that 
\begin{align}\begin{split}
&\Bigl(\var\!\left({\bigV}_{n,M}^{{0}}(t,x)\right)\Bigr)^{\!\nicefrac{1}{2}}\leq  \frac{2e^{\growrate (T-t)/2}\sqrt{\lyaV(x)}}{\sqrt{M^n}}
\sup_{s\in[0,T]}\sup_{z\in\R^d}\left[\frac{\max\{|T(\funcF(0))(s,z)|,|g(z)|\}}{\sqrt{\lyaV(z)}}\right]
\\&+
\sum_{\ell=0}^{n-1}\vastl{25pt}{[}
\frac{(2-\1_{\{n-1\}}(\ell))e^{-\growrate t/2}\sqrt{\lyaV(x)}}{\sqrt{M^{n-\ell-1}}}
(T-t)^{\nicefrac{1}{2}}\LipConstF\\
&\cdot
\left(\int_{t}^{T}\sup_{r\in[s,T]}\sup_{z\in\R^d}\left[\frac{e^{\growrate r}\E\bigl[
|\bigV^{0}_{\ell,M}(r,z) -\smallV(r,z)|^2\bigr]
}{\lyaV(z)}\right]\!ds\right)^{\!\!\nicefrac{1}{2}}\vastr{25pt}{]}.
\end{split}\label{k09b}
\end{align}
Hence, we obtain that
\begin{align}\begin{split}
&\sup_{x\in\R^d}\Bigl((\lyaV(x))^{-1}e^{\growrate t}\var\!\left({\bigV}_{n,M}^{{0}}(t,x)\right)\Bigr)^{\!\nicefrac{1}{2}}
\leq \frac{2e^{\growrate T/2}}{\sqrt{M^n}}\sup_{s\in[0,T]}\sup_{z\in\R^d}\left[\frac{\max\{|T(\funcF(0))(s,z)|,|g(z)|\}}{\sqrt{\lyaV(z)}}\right]
\\&+
\sum_{\ell=0}^{n-1}\left[
\frac{(2-\1_{\{n-1\}}(\ell))(T-t)^{\nicefrac{1}{2}}\LipConstF}{\sqrt{M^{n-\ell-1}}}
\left(\int_{t}^{T}\sup_{r\in[s,T]}\sup_{x\in\R^d}\left[\frac{e^{\growrate r}\E\bigl[
|\bigV^{0}_{\ell,M}(r,x) -\smallV(r,x)|^2\bigr]
}{\lyaV(x)}\right]ds\right)^{\!\nicefrac{1}{2}}\right].\label{k09}
\end{split}\end{align}
Next observe that
\eqref{t27b} and \cref{k07} in \cref{k08} imply that for all $x\in\R^d$ it holds that
\begin{align}
\E\!\left [{\bigV}_{n,M}^{{0}}(t,x)\right]-\smallV(t,x)=\int_t^T
\E\!\left[(\funcF(\bigV^0_{n-1,M}))(s,\sppr^0_{t,s}(x))- (\funcF(\smallV))(s,\sppr^0_{t,s}(x)) \right]ds.\end{align}
Combining this, Jensen's inequality, \cref{ppt3} in \cref{ppt0},
 the fact that for all $t\in [0,T]$, $s\in [t,T]$, $x\in \R^d$
 it holds that
$\E\big[\lyaV(\fwpr _{t,s}(x))\big]\leq e^{\growrate(s-t)}\lyaV(x)$,
\eqref{t01},
and, e.g., the disintegration-type result in \cite[Lemma~2.2]{HJKNW2018}
demonstrates that for all $x\in\R^d$ it holds that
\begin{align}\begin{split}
&\left|\E \!\left[{\bigV}_{n,M}^{{0}}(t,x)\right]-\smallV(t,x)\right|\\
&\leq (T-t)^{\nicefrac{1}{2}}
\left(
\int_t^T
\E\!\left[\left|(\funcF(\bigV^0_{n-1,M}))(s,\sppr^0_{t,s}(x))- (\funcF(\smallV))(s,\sppr^0_{t,s}(x)) \right|^2\right]ds\right)^{\!\!\nicefrac{1}{2}}\\
&= (T-t)^{\nicefrac{1}{2}}\left(
\int_t^T
\E\!\left[
\E\!\left[ \left|(\funcF(\bigV_{n-1,M}^{0}))(s,z)-(\funcF(\smallV))(s,z)\right|^2\right]\Bigr|_{z=\sppr^0_{t,s}(x)}\right]ds\right)^{\!\!\nicefrac{1}{2}}\\
&\leq  (T-t)^{\nicefrac{1}{2}}\left(\int_{t}^{T}
\left[\sup_{r\in[s,T]}\sup_{z\in\R^d}
\frac{\E\!\left[ \left|(\funcF(\bigV_{n-1,M}^{0}))(r,z)-(\funcF(\smallV))(r,
z)\right|^2\right]}{\lyaV(z)}
\right]\!\E\!\left[\lyaV(\sppr_{t,s}^0(x))\right]\!ds\right)^{\!\!\nicefrac{1}{2}}\\
&\leq \LipConstF (T-t)^{\nicefrac{1}{2}}\left(\int_{t}^{T}
\left[
\sup_{r\in[s,T]}\sup_{z\in\R^d}
\frac{\E\!\left[ \left|\bigV_{n-1,M}^{0}(r,z)-\smallV(r,
z)\right|^2\right]}{\lyaV(z)}
\right]\!
\E\!\left[\lyaV(\fwpr_{t,s}(x))\right]ds\right)^{\!\!\nicefrac{1}{2}}\\
&\leq \LipConstF (T-t)^{\nicefrac{1}{2}}\left(\int_{t}^{T}
\left[\sup_{r\in[s,T]}\sup_{z\in\R^d}
\frac{\E\!\left[ \left|\bigV_{n-1,M}^{0}(r,z)-\smallV(r,
z)\right|^2\right]
}{\lyaV(z)}\right]e^{\growrate (s-t)}\lyaV(x)\,ds\right)^{\!\!\nicefrac{1}{2}}.
\end{split}\end{align}
Therefore, we obtain that
\begin{align}\begin{split}
&\sup_{x\in\R^d}\left[\Bigl.
\frac{e^{\growrate t}\left|\E \!\left[{\bigV}_{n,M}^{{0}}(t,x)\right]-\smallV(t,x)\right|}{\lyaV(x)}\right]^{\!\nicefrac{1}{2}}\\&\leq 
\LipConstF (T-t)^{\nicefrac{1}{2}}\left(\int_{t}^{T}\sup_{r\in[s,T]}\sup_{x\in\R^d}\left[
\frac{e^{\growrate r}
\E\!\left[ |\bigV_{n-1,M}^{0}(r,x)-\smallV(r,x)|^2\right]}{\lyaV(x)}\right]\! ds\right)^{\!\!\nicefrac{1}{2}}.\end{split}
\end{align}
Combining \eqref{k09} and the triangle inequality hence shows that
\begin{align}\begin{split}
&\sup_{x\in\R^d}
\Bigl[
e^{\growrate t}\lyaV(x)^{-1}\E\!\left[\left|{\bigV}_{n,M}^{{0}}(t,x)-\smallV(t,x)\right|^2\right]\Bigr]^{\!\nicefrac{1}{2}}\\
&\leq \sup_{x\in\R^d}
\Bigl[
e^{\growrate t}\lyaV(x)^{-1}\left|\E \!\left[{\bigV}_{n,M}^{{0}}(t,x)\right]-\smallV(t,x)\right|\Bigr]^{\!\nicefrac{1}{2}}+\sup_{x\in\R^d}\Bigl[e^{\growrate t}\lyaV(x)^{-1}\var\!\left({\bigV}_{n,M}^{{0}}(t,x)\right)\Bigr]^{\!\nicefrac{1}{2}}\\
&\leq 
\LipConstF (T-t)^{\nicefrac{1}{2}}\left(\int_{t}^{T}\sup_{r\in[s,T]}\sup_{x\in\R^d}\left[(\lyaV(x))^{-1}
e^{\growrate r}
\E\!\left[ \left|\bigV_{n-1,M}^{0}(r,x)-\smallV(r,x)\right|^2\right]\right]ds \right)^{\!\nicefrac{1}{2}}\\&\quad +\frac{2e^{\growrate T/2}}{\sqrt{M^n}}
\sup_{s\in[0,T]}\sup_{z\in\R^d}\left[\frac{\max\{|T(\funcF(0))(s,z)|,|g(z)|\}}{\sqrt{\lyaV(z)}}\right]
\\&\quad +
\left[
\sum_{\ell=0}^{n-1}
\frac{(2-\1_{\{n-1\}}(\ell))(T-t)^{\nicefrac{1}{2}}\LipConstF}{\sqrt{M^{n-\ell-1}}}
\left(
\int_{t}^{T}
\sup_{r\in[s,T]}\sup_{x\in\R^d}\biggl[
\frac{e^{\growrate r}\E\!\left[|
\bigV^{0}_{\ell,M}(r,x) -\smallV(r,x)|^2\right]}{\lyaV(x)}
\biggr]ds\right)^{\!\nicefrac{1}{2}}\right]\\
&= 
\frac{2e^{\growrate T/2}}{\sqrt{M^n}}\sup_{s\in[0,T]}\sup_{z\in\R^d}\left[\frac{\max\{|T(\funcF(0))(s,z)|,|g(z)|\}}{\sqrt{\lyaV(z)}}\right]\\
&\quad
 +
\sum_{\ell=0}^{n-1}
\frac{2(T-t)^{\nicefrac{1}{2}}\LipConstF}{\sqrt{M^{n-\ell-1}}}
\left(
\int_{t}^{T}
\sup_{r\in[s,T]}\sup_{x\in\R^d}\left[\Bigl.
\frac{e^{\growrate r}\E\!\left[
|
\bigV^{0}_{\ell,M}(r,x) -\smallV(r,x)|^2\right]}{\lyaV(x)}
\right]\!ds\right)^{\!\nicefrac{1}{2}}.
\end{split}\end{align}
This establishes \eqref{k01_b}. The proof of \cref{k01} is thus complete.
\end{proof}

\subsection{Function space-valued Gronwall-type inequalities}
\label{subsec:gronwall}

\begin{lemma}\label{lem:gronwall1}
Let $K\in \N$, $\alpha,\beta\in [0,\infty)$, $\varepsilon_0, \varepsilon_1,\ldots \varepsilon_K\in [0,\infty]$ satisfy $\max\{\varepsilon_0, \varepsilon_1,\ldots \varepsilon_{K-1}\}<\infty$ and
$
\varepsilon_K\le \alpha +\beta\big[\sum_{k=0}^{K-1}\varepsilon_k\big]
$.
Then $\varepsilon_K<\infty$.
\end{lemma}
\begin{proof}[Proof of \cref{lem:gronwall1}]
Note that the hypothesis that $\max\{\varepsilon_0, \varepsilon_1,\ldots \varepsilon_{K-1}\}<\infty$ implies that
$
\alpha +\beta\big[\sum_{k=0}^{K-1}\varepsilon_k\big]<\infty
$.
This and the hypothesis that $\varepsilon_K\le \alpha +\beta\big[\sum_{k=0}^{K-1}\varepsilon_k\big]$ establish that $\varepsilon_K<\infty$. The proof of \cref{lem:gronwall1} is thus complete.
\end{proof}

\begin{lemma}\label{lem:gronwall2}
Let $N\in \N$, $\beta, \alpha_0, \alpha_1,\ldots, \alpha_N\in [0,\infty)$, $\varepsilon_0, \varepsilon_1,\ldots \varepsilon_N\in [0,\infty]$ satisfy for all $n\in \{0,1,\ldots,N\}$ that
$\varepsilon_n\le \alpha_n +\beta\big[\sum_{k=0}^{n-1}\varepsilon_k\big]$
(cf.\ \cref{lem:gronwall1}).
Then it holds for all $n\in \{0,1,\ldots,N\}$ that
\begin{equation}
\varepsilon_n\le \alpha_n+\beta \left[\sum_{k=0}^{n-1}(1+\beta)^{n-k-1}\alpha_k\right]<\infty.
\end{equation}
\end{lemma}
\begin{proof}[Proof of \cref{lem:gronwall2}]
Throughout this proof let 
$\gamma_0, \gamma_1, \ldots, \gamma_N
\in \R$ satisfy for all $n\in \{0,1,\ldots,N\}$ that
\begin{equation}\label{eq:gronwall3}
\gamma_n=\alpha_n +\beta\left[\sum_{k=0}^{n-1}\gamma_k\right].
\end{equation}
We claim that for all $n\in \{0,1,\ldots,N\}$ it holds that
\begin{equation}\label{eq:gronwall2}
\gamma_n=\alpha_n+\beta \left[\sum_{k=0}^{n-1}(1+\beta)^{n-k-1}\alpha_k\right].
\end{equation}
We prove \cref{eq:gronwall2} by induction on $n\in \{0,1,\ldots,N\}$. For the base case $n=0$ observe that \cref{eq:gronwall3} ensures that $\gamma_0=\alpha_0$. This proves \cref{eq:gronwall2} in the base case $n=0$. For the induction step $\{0,1,\ldots,N-1\}\ni n-1 \induct n \in \{1,2,\ldots,N\}$ observe that \cref{eq:gronwall3} implies that for all  $n \in \{1,2,\ldots,N\}$ with $\gamma_{n-1}= \alpha_{n-1}+\beta \sum_{k=0}^{n-2}(1+\beta)^{n-k-2}\alpha_k$ it holds that
\begin{equation}
\begin{split}
\gamma_n&=\alpha_n +\beta\left[\sum_{k=0}^{n-1}\gamma_k\right]
=\alpha_n-\alpha_{n-1}+\beta \gamma_{n-1}+\alpha_{n-1}+\beta\left[\sum_{k=0}^{n-2}\gamma_k\right]\\
&=\alpha_n-\alpha_{n-1}+\beta \gamma_{n-1}+\gamma_{n-1}=
\alpha_n-\alpha_{n-1}+(1+\beta)\gamma_{n-1}\\
&=\alpha_n-\alpha_{n-1}+(1+\beta)\left(
 \alpha_{n-1}+\beta \left[\sum_{k=0}^{n-2}(1+\beta)^{n-k-2}\alpha_k\right]\right)\\
 &=\alpha_n+\beta
 \alpha_{n-1}+\beta \left[\sum_{k=0}^{n-2}(1+\beta)^{n-k-1}\alpha_k\right]
 =\alpha_n+\beta \left[\sum_{k=0}^{n-1}(1+\beta)^{n-k-1}\alpha_k\right].
\end{split}
\end{equation}
Induction hence establishes \cref{eq:gronwall2}. Moreover, note that \cref{eq:gronwall3}, induction, and the assumption that for all $n\in \{0,1,\ldots,N\}$ it holds that
$\varepsilon_n\le \alpha_n +\beta\big[\sum_{k=0}^{n-1}\varepsilon_k\big]$ prove that for all $n\in \{0,1,\ldots,N\}$ it holds that $\varepsilon_n\le \gamma_n$. This and \cref{eq:gronwall2} establish that for all $n\in \{0,1,\ldots,N\}$ it holds that
\begin{equation}
\varepsilon_n\le \alpha_n+\beta \left[\sum_{k=0}^{n-1}(1+\beta)^{n-k-1}\alpha_k\right]<\infty.
\end{equation}
The proof of \cref{lem:gronwall2} is thus complete.
\end{proof}

\begin{lemma}\label{g09-2} Let $K\in\N$, $a,b,\lemdreisechsc \in [0,\infty)$, ${\lemdreisechst}\in \R$, $\lemdreisechsT\in [\lemdreisechst,\infty)$, $p\in (0,\infty)$,
let $f_n\colon [{\lemdreisechst},\lemdreisechsT]\to[0,\infty]$, $n\in \N_0$, be 
measurable, assume $\sup_{s\in [\lemdreisechst,\lemdreisechsT]}\max\{|f_0(s)|,|f_1(s)|,\ldots,|f_{K-1}(s)|\}<\infty$, and assume for all $t\in[\lemdreisechst,\lemdreisechsT]$ that
\begin{align}
|f_K(t)|\leq a\lemdreisechsc^K + \sum_{\ell=0}^{K-1}\left[
b\lemdreisechsc^{K-\ell-1}
\left[
\int_\lemdreisechst^t|f_\ell(s)|^p\,ds\right]^{\!\nicefrac{1}{p}}\right].
\label{g09-2a}
\end{align}
Then $\sup_{s\in [\lemdreisechst,\lemdreisechsT]}|f_K(s)|<\infty$.
\end{lemma}
\begin{proof}[Proof of \cref{g09-2}]
Note that the hypothesis that $\sup_{s\in [\lemdreisechst,\lemdreisechsT]}\max\{|f_0(s)|,|f_1(s)|,\ldots,|f_{K-1}(s)|\}<\infty$ implies that
\begin{equation}
\begin{split}
&\sup_{t \in [\lemdreisechst,\lemdreisechsT]}\left(\sum_{\ell=0}^{K-1}\left[
b\lemdreisechsc^{K-\ell-1}
\left[
\int_\lemdreisechst^t|f_\ell(s)|^p\,ds\right]^{\!\nicefrac{1}{p}}\right]\right)
\le 
\sum_{\ell=0}^{K-1}\left[
b\lemdreisechsc^{K-\ell-1}
\left[
\int_\lemdreisechst^\lemdreisechsT |f_\ell(s)|^p\,ds\right]^{\!\nicefrac{1}{p}}\right]\\
&\le 
\sum_{\ell=0}^{K-1}\left[
b\lemdreisechsc^{K-\ell-1}
\left[ \sup_{s \in [\lemdreisechst,\lemdreisechsT]} |f_\ell(s)|\right]
\left[ \lemdreisechsT -\lemdreisechst\right]^{\!\nicefrac{1}{p}}\right]<\infty.
\end{split}
\end{equation}
Combining this with \cref{g09-2a} establishes that $\sup_{s\in [\lemdreisechst,\lemdreisechsT]}|f_K(s)|<\infty$. The proof of \cref{g09-2} is thus complete.
\end{proof}

\begin{lemma}\label{g09-1} Let $N\in\N$, $a,b,\lemdreisechsc \in [0,\infty)$, ${\lemdreisechst}\in \R$, $\lemdreisechsT\in [\lemdreisechst,\infty)$, $p\in [1,\infty)$,
let $f_n\colon [{\lemdreisechst},\lemdreisechsT]\to[0,\infty]$, $n\in \N_0$, be 
measurable, assume $\sup_{s\in [\lemdreisechst,\lemdreisechsT]}|f_0(s)|<\infty$, and assume for all $n\in\{1,2,\ldots,N\}$, $t\in[\lemdreisechst,\lemdreisechsT]$ that
\begin{align}
|f_n(t)|\leq a\lemdreisechsc^n + \sum_{\ell=0}^{n-1}\left[
b\lemdreisechsc^{n-\ell-1}
\left[
\int_\lemdreisechst^t|f_\ell(s)|^p\,ds\right]^{\!\nicefrac{1}{p}}\right]
\label{g09-1a}
\end{align}
(cf.\ \cref{g09-2}).
Then
\begin{equation}\label{g09-2x}
\begin{split}
f_N ({\lemdreisechsT})&\leq
ac^N+b(\beta-\alpha)^{\nicefrac{1}{p}}[1+b(\beta-\alpha)^{\nicefrac{1}{p}}]^{N-1}\left[ \max_{ k \in \{0,1,\ldots,N\} } \tfrac{ \lemdreisechsc^{N-k}}{ (k!)^{1/p} } \right]\left[\sup_{s\in [\lemdreisechst,\lemdreisechsT]}|f_0(s)|\right]\\
&\quad +ab(\beta-\alpha)^{\nicefrac{1}{p}}\sum_{n=1}^{N-1}[1+b(\beta-\alpha)^{\nicefrac{1}{p}}]^{N-n-1}\left[ \max_{ k \in \{0,1,\ldots,N-n\}} \tfrac{ \lemdreisechsc^{N-k}}{ (k!)^{1/p} } \right]\\
&\leq 
\left[ a+b(\lemdreisechsT-\lemdreisechst)^{\nicefrac{1}{p}}\left[\sup_{s\in [\lemdreisechst,\lemdreisechsT]}|f_0(s)|\right]\right]
\left[\max_{ k \in \{0,1,\ldots,N\}} \tfrac{ \lemdreisechsc^{N-k}}{ (k!)^{1/p} }\right]
\left[1+b(\lemdreisechsT-\lemdreisechst)^{\nicefrac{1}{p}}\right]^{N-1}.
\end{split}
\end{equation}
\end{lemma}

\begin{proof}[Proof of \cref{g09-1}]
Throughout this proof assume w.l.o.g.\ that $\alpha < \beta$,
let $\gamma_k\colon \mathcal{B}([{\lemdreisechst},\lemdreisechsT])\to [0,\infty)$, $k\in\N_0$, 
satisfy
for all $k\in \N$, $A\in \mathcal{B}([{\lemdreisechst},\lemdreisechsT])$ that 
\begin{align}
\gamma_{0}(A)=\1_A({\lemdreisechsT})\qquad\text{and}\qquad\gamma_k(A)= \frac{1}{(\lemdreisechsT-{\lemdreisechst})^k}\int_A\frac{{(\lemdreisechsT-t)^{k-1}}}{(k-1)!}\,dt,\label{g10}
\end{align} 
and let $\varepsilon_n\in[0,\infty]$, $n\in \N_0$, satisfy for all $n\in \N_0$ that
\begin{align}\label{g10b}
\varepsilon_n=\sup\left\{\lemdreisechsc^j{\left[\int_{{\lemdreisechst}}^{\lemdreisechsT}|f_n(t)|^p\gamma_k(dt)\right]^{\!\nicefrac{1}{p}}}\colon j,k\in\N_0,n+j+k=N\right\}.
\end{align}
Observe that \cref{g10} ensures that for all
$k\in \N_0$ 
%
it holds that
\begin{align}\label{g11b}
\int_{{\lemdreisechst}}^{\lemdreisechsT}\gamma_k(dt)=\frac{1}{k!}
\end{align}
This
and \eqref{g10b} show that
\begin{align}\begin{split}
\varepsilon_0&= 
\sup\left\{{\lemdreisechsc^j\left[\int_{{\lemdreisechst}}^{\lemdreisechsT}|f_0(t)|^p\gamma_k(dt)\right]^{\!\nicefrac{1}{p}}}\colon j,k\in\N_0,j+k=N\right\}\\
&\leq 
\sup\left\{
{\lemdreisechsc^j|f_0(s)|\left[\int_{{\lemdreisechst}}^{\lemdreisechsT}\gamma_k(dt)\right]^{\!\nicefrac{1}{p}}}\colon j,k\in\N_0,j+k=N,{s\in[\lemdreisechst,\lemdreisechsT]}\right\}\\
&= 
\sup\left\{
{\frac{\lemdreisechsc^j|f_0(s)|}{(k!)^{1/p}}}\colon j,k\in\N_0,j+k=N,{s\in[\lemdreisechst,\lemdreisechsT]}\right\}\leq 
\left[ \max_{ k \in \{0,1,\ldots,N\} } \tfrac{ \lemdreisechsc^{N-k}}{ (k!)^{1/p} } \right]\left[\sup_{s\in [\lemdreisechst,\lemdreisechsT]}|f_0(s)|\right].\label{g11}\end{split}
\end{align}
Next note that Fubini's theorem and \eqref{g10} imply that for all $\ell\in[0,N]\cap\N$, $k\in\N_0$ it holds that
\begin{align}\begin{split}
&\int_{{\lemdreisechst}}^{\lemdreisechsT} \int_\lemdreisechst^t |{f_\ell}(s)|^p\,ds\,\gamma_k(dt)\\
&=
\1_{\{0\}}(k)\left[\int_{{\lemdreisechst}}^{\lemdreisechsT}|f_\ell(s)|^p\,ds\right]+
\1_{\N}(k)\left[
\frac{1}{(\lemdreisechsT-{\lemdreisechst})^k}
\int_{{\lemdreisechst}}^{\lemdreisechsT} \frac{{(\lemdreisechsT-t)}^{k-1}}{(k-1)!}\int_\lemdreisechst^t |f_\ell(s)|^p\,ds\,dt\right]\\
&=
\1_{\{0\}}(k)\left[\int_{{\lemdreisechst}}^{\lemdreisechsT}|f_\ell(s)|^p\,ds\right]+
\1_{\N}(k)
\left[
\frac{1}{(\lemdreisechsT-{\lemdreisechst})^k}
\int_{{\lemdreisechst}}^{\lemdreisechsT}\int_{s}^{\lemdreisechsT} \frac{{(\lemdreisechsT-t)}^{k-1}}{(k-1)!} |f_\ell(s)|^p\,dt\,ds\right]
\\
&= 
\1_{\{0\}}(k)\left[\int_{{\lemdreisechst}}^{\lemdreisechsT}|f_\ell(s)|^p\,ds\right]+\1_{\N}(k)
\left[\frac{1}{(\lemdreisechsT-{\lemdreisechst})^{k}}
\int_{{\lemdreisechst}}^{\lemdreisechsT} \frac{{(\lemdreisechsT-s)}^{k}}{k!} |f_\ell(s)|^p\,ds\right]\\
&=
\1_{\{0\}}(k)\left[(\lemdreisechsT-\lemdreisechst)\int_{{\lemdreisechst}}^{\lemdreisechsT}|f_\ell(s)|^p\,\gamma_1(ds)\right]+\1_{\N}(k)
\left[(\lemdreisechsT-{\lemdreisechst})
\int_{{\lemdreisechst}}^{\lemdreisechsT} |f_\ell(s)|^p\,\gamma_{k+1}(ds)\right]\\
&
=(\lemdreisechsT-\lemdreisechst) \int_{{\lemdreisechst}}^{\lemdreisechsT}|f_\ell(s)|^p\,\gamma_{k+1}(ds).
\end{split}\end{align}
Combining \eqref{g09-1a}
 and \eqref{g11b} hence assures that
 for all $n\in \{1,2,\ldots,N\}$, $k,j\in\N_0$ with $n+k+j=N$ it holds that
\begin{align}\small\begin{split}
&\lemdreisechsc^j\left[\int_{{\lemdreisechst}}^{\lemdreisechsT}|f_n(t)|^p\gamma_k(dt)\right]^{\!\nicefrac{1}{p}}\\
&
\leq a \lemdreisechsc^{n+j}\left[\int_{\lemdreisechst}^{\lemdreisechsT}\gamma_k(dt)\right]^{\nicefrac{1}{p}}+ \sum_{\ell=0}^{n-1}\left[b\lemdreisechsc^{n+j-\ell-1}\left[\int_{\lemdreisechst}^\lemdreisechsT
\int_\lemdreisechst^t|f_\ell(s)|^p\,ds\,\gamma_k(dt)\right]^{\!\nicefrac{1}{p}}\right]\\
&
\leq \frac{a \lemdreisechsc^{N-k}}{(k!)^{1/p}}+ \sum_{\ell=0}^{n-1}\left[
b(\lemdreisechsT-\lemdreisechst)^{\nicefrac{1}{p}}\lemdreisechsc^{n+j-\ell-1}\left[
\int_{\lemdreisechst}^\lemdreisechsT |f_\ell(s)|^p\gamma_{k+1}(ds)\right]^{\!\nicefrac{1}{p}}\right].
\end{split}\end{align}
This, \eqref{g10b}, and the fact that for all
$
n,k,j\in\N_0 $ with $n+k+j=N$ it holds that $\ell+(n+j-\ell-1)+(k+1)=N
$ imply that for all $n\in \{1,2,\ldots,N\}$ it holds that
\begin{align}
\varepsilon_n\leq 
a\left[ \max_{ k \in \{0,1,\ldots,N-n\}} \frac{ \lemdreisechsc^{N-k}}{ (k!)^{1/p} } \right]
+b(\lemdreisechsT-\lemdreisechst)^{\nicefrac{1}{p}}\left[\sum_{\ell=0}^{n-1} \varepsilon_\ell \right].
\end{align}
Combining \eqref{g10}, \eqref{g10b}, and \eqref{g11} with \cref{lem:gronwall2} (applied with $\beta\defeq b(\beta-\alpha)^{\nicefrac{1}{p}}$,
$\alpha_0\defeq [ \max_{ k \in \{0,1,\ldots,N\} } \tfrac{ \lemdreisechsc^{N-k}}{ (k!)^{1/p} } ][\sup_{s\in [\lemdreisechst,\lemdreisechsT]}|f_0(s)|]$, $(\alpha_n)_{n\in \{1,2,\ldots,N\}}\defeq (a[ \max_{ k \in \{0,1,\ldots,N-n\}} \tfrac{ \lemdreisechsc^{N-k}}{ (k!)^{1/p} } ])_{n\in \{1,2,\ldots,N\}}$ in the notation of  \cref{lem:gronwall2})
hence shows that
\begin{align}\begin{split}\label{g10x}
&f_N({\lemdreisechsT})=\left[\int_{\lemdreisechst}^\lemdreisechsT|f_N(t)|^p\,\gamma_0(dt)\right]^{\!\nicefrac{1}{p}}= \varepsilon_N\\
&\le ac^N+b(\beta-\alpha)^{\nicefrac{1}{p}}[1+b(\beta-\alpha)^{\nicefrac{1}{p}}]^{N-1}\left[ \max_{ k \in \{0,1,\ldots,N\} } \tfrac{ \lemdreisechsc^{N-k}}{ (k!)^{1/p} } \right]\left[\sup_{s\in [\lemdreisechst,\lemdreisechsT]}|f_0(s)|\right]\\
&+ab(\beta-\alpha)^{\nicefrac{1}{p}}\sum_{n=1}^{N-1}[1+b(\beta-\alpha)^{\nicefrac{1}{p}}]^{N-n-1}\left[ \max_{ k \in \{0,1,\ldots,N-n\}} \tfrac{ \lemdreisechsc^{N-k}}{ (k!)^{1/p} } \right]
.\end{split}
\end{align}
The fact that for all $l\in \{0,1,\ldots,N\}$ it holds that $\max_{ k \in \{0,1,\ldots,l\}} \tfrac{ \lemdreisechsc^{N-k}}{ (k!)^{1/p} } \le  \max_{ k \in \{0,1,\ldots,N\}} \tfrac{ \lemdreisechsc^{N-k}}{ (k!)^{1/p} }$ therefore implies that
\begin{equation}
\begin{split}
\varepsilon_N&\le a\left[\max_{ k \in \{0,1,\ldots,N\}} \tfrac{ \lemdreisechsc^{N-k}}{ (k!)^{1/p} }\right]+ab(\beta-\alpha)^{\nicefrac{1}{p}}\left[ \max_{ k \in \{0,1,\ldots,N\}} \tfrac{ \lemdreisechsc^{N-k}}{ (k!)^{1/p} } \right]\sum_{n=1}^{N-1}[1+b(\beta-\alpha)^{\nicefrac{1}{p}}]^{N-n-1}\\
&+b(\beta-\alpha)^{\nicefrac{1}{p}}[1+b(\beta-\alpha)^{\nicefrac{1}{p}}]^{N-1}\left[ \max_{ k \in \{0,1,\ldots,N\} } \tfrac{ \lemdreisechsc^{N-k}}{ (k!)^{1/p} } \right]\left[\sup_{s\in [\lemdreisechst,\lemdreisechsT]}|f_0(s)|\right]\\
&=
a\left[\max_{ k \in \{0,1,\ldots,N\}} \tfrac{ \lemdreisechsc^{N-k}}{ (k!)^{1/p} }\right] [1+b(\beta-\alpha)^{\nicefrac{1}{p}}]^{N-1}\\
&
+b(\beta-\alpha)^{\nicefrac{1}{p}}[1+b(\beta-\alpha)^{\nicefrac{1}{p}}]^{N-1}\left[ \max_{ k \in \{0,1,\ldots,N\} } \tfrac{ \lemdreisechsc^{N-k}}{ (k!)^{1/p}} \right]\left[\sup_{s\in [\lemdreisechst,\lemdreisechsT]}|f_0(s)|\right]\\
&=\left[\max_{ k \in \{0,1,\ldots,N\}} \tfrac{ \lemdreisechsc^{N-k}}{ (k!)^{1/p} }\right] [1+b(\beta-\alpha)^{\nicefrac{1}{p}}]^{N-1}\left[a+b(\beta-\alpha)^{\nicefrac{1}{p}}
\left[\sup_{s\in [\lemdreisechst,\lemdreisechsT]}|f_0(s)|\right]\right].
\end{split}
\end{equation}
Combining this with \cref{g10x} establishes \eqref{g09-2x}. The proof of \cref{g09-1} is thus complete.
\end{proof}

\begin{lemma}\label{g090a} Let $N\in\N$, $a,b,\lemdreisechsc \in [0,\infty)$, ${\lemdreisechst}\in \R$, $\lemdreisechsT\in [\lemdreisechst,\infty)$, $p\in [1,\infty)$,
let $f_n\colon [{\lemdreisechst},\lemdreisechsT]\to[0,\infty]$, $n\in \N_0$, be 
measurable, assume $\sup_{s\in [\lemdreisechst,\lemdreisechsT]}|f_0(s)|<\infty$, and assume for all $n\in\{1,2,\ldots,N\}$, $t\in[\lemdreisechst,\lemdreisechsT]$ that
\begin{align}
|f_n(t)|\leq a\lemdreisechsc^n + \sum_{\ell=0}^{n-1}\left[
b\lemdreisechsc^{n-\ell-1}
\left[
\int_t^\lemdreisechsT|f_\ell(s)|^p\,ds\right]^{\!\nicefrac{1}{p}}\right]
\label{g09a}
\end{align}
(cf.\ \cref{g09-2}).
Then
\begin{align}\label{g09a0}
f_N ({\lemdreisechst})\leq\left[ a+b(\lemdreisechsT-\lemdreisechst)^{\nicefrac{1}{p}}\left[\sup_{s\in [\lemdreisechst,\lemdreisechsT]}|f_0(s)|\right]\right]
\left[\max_{ k \in \{0,1,\ldots,N\}} \tfrac{ \lemdreisechsc^{N-k}}{ (k!)^{1/p} }\right]
\left[1+b(\lemdreisechsT-\lemdreisechst)^{\nicefrac{1}{p}}\right]^{N-1}
.\end{align}
\end{lemma}
\begin{proof}[Proof of \cref{g090a}]
Note that \cref{g09-1} (applied with 
$\lemdreisechst\defeq -\lemdreisechsT$, $\lemdreisechsT\defeq -\lemdreisechst$, $( f_n )_{ n \in \N_0 }\defeq (([-\lemdreisechsT,-\lemdreisechst] \ni t\mapsto f_n(-t)\in [0,\infty]))_{ n \in \N_0 }$
in the notation of \cref{g09-1}) 
establishes \cref{g09a0}. The proof of \cref{g090a} is thus complete.
\end{proof}

\begin{lemma}\label{g09}Let $M,N\in\N$, 
$T\in (0,\infty)$, ${\tau}\in [0,T]$,
$a,b\in [0,\infty)$, $p\in [1,\infty)$, let $f_n\colon [{\tau},T]\to[0,\infty]$, $n\in \N_0$, be 
measurable, assume $\sup_{s\in [\tau,T]}|f_0(s)|<\infty$, and assume for all $n\in\{1,2,\ldots,N\}$, $t\in[\tau,T]$ that
\begin{align}
|f_n(t)|\leq \frac{a}{\sqrt{M^n}}+ \sum_{\ell=0}^{n-1}\left[
\frac{b}{\sqrt{M^{n-\ell-1}}}
\left[
\int_t^T|f_\ell(s)|^p\,ds\right]^{\!\nicefrac{1}{p}}\right]
\label{g09z}
\end{align}
(cf.\ \cref{g09-2}).
Then
\begin{align}
f_N({\tau})\leq\left[ a+b(T-\tau)^{\nicefrac{1}{p}}\left[\sup_{s\in [\tau,T]}|f_0(s)|\right]\right]\exp\!\left(\frac{M^{\nicefrac{p}{2}}}{p}\right)M^{-\nicefrac{N}{2}}
\left[1+b(T-\tau)^{\nicefrac{1}{p}}\right]^{N-1}
.\end{align}
\end{lemma}

\begin{proof}[Proof of \cref{g09}]
Note that \cref{g090a} (applied
with
$\lemdreisechsc\defeq M^{-1/2}$,
$\lemdreisechst\defeq\tau$,
$\lemdreisechsT\defeq T$
 in the notation of
\cref{g090a}) assures
that 
\begin{align}
f_N ({\tau})\leq\left[ a+b(T-\tau)^{\nicefrac{1}{p}}\left[\sup_{s\in [\tau,T]}|f_0(s)|\right]\right]\left[ \sup_{ k \in \N_0 } \tfrac{ M^{\nicefrac{k}{2}} }{  (k!)^{1/p} } \right]M^{-\nicefrac{N}{2}}
\left[1+b(T-\tau)^{\nicefrac{1}{p}}\right]^{N-1}
.\end{align}
The fact that $\sup_{ k \in \N_0 } (\tfrac{ M^{k/2} }{  (k!)^{1/p} })\le \exp\bigl(\frac{M^{p/2}}{p}\bigr)$ hence proves that 
\begin{align}
f_N ({\tau})\leq\left[ a+b(T-\tau)^{\nicefrac{1}{p}}\left[\sup_{s\in [\tau,T]}|f_0(s)|\right]\right]\exp\!\left(\frac{M^{\nicefrac{p}{2}}}{p}\right)M^{-\nicefrac{N}{2}}
\left[1+b(T-\tau)^{\nicefrac{1}{p}}\right]^{N-1}.
\end{align}
The proof of \cref{g09} is thus complete.
\end{proof}

\subsection{Non-recursive error bounds for MLP approximations}
\label{x010}
\begin{corollary}[Error estimate]\label{x01}Assume \cref{t26} and let $N,M\in\N$,  $\tau\in [0,T]$. Then
\begin{align}\begin{split}
&\sup_{x\in\R^d}\left[\frac{
\E\!\left[|{\bigV}_{N,M}^{0}(\tau,x)-\smallV(\tau,x)|^2\right]}{\lyaV(x)}\right]^{\!\nicefrac{1}{2}}\leq 2 e^{M/2}M^{-N/2}(1+2TL)^{N-1}e^{\growrate (T-\tau)/2}
\\& \cdot 
\left[\sup_{t\in[0,T]}\sup_{x\in\R^d}\left[\frac{\max\{|T(\funcF(0))(t,x)|,|g(x)|\}}{\sqrt{\lyaV(x)}}\right]+TL\sup_{t\in[0,T]}\sup_{x\in\R^d}\left[
\frac{|\smallV(t,x)|}{\sqrt{\lyaV(x)}}\right]\right].
\end{split}\end{align}
\end{corollary}
\begin{proof}[Proof of \cref{x01}]
Throughout this proof 
let
$a_1,a_2\in \R$ satisfy
\begin{equation}
a_1= 2e^{\growrate T/2}\left[\sup_{t\in[0,T]}\sup_{x\in\R^d}\left(\frac{\max\{|T(\funcF(0))(t,x)|,|g(x)|\}}{\sqrt{\lyaV(x)}}\right)\right]
\qquad\text{and}\qquad a_2=
2\sqrt{T}\LipConstF,
\end{equation}
let
 $f_n\colon [\tau,T]\to[0,\infty]$,
$n\in\{ 0, 1, \ldots, N \}$, satisfy for all $n\in\{ 0, 1, \ldots, N \}$,
$t\in [\tau,T]$ that
\begin{align}\begin{split}
&f_n(t)= \sup_{s\in[t,T]}\sup_{x\in\R^d}
\Bigl[
e^{\growrate s}|\lyaV(x)|^{-1}\E\big[|{\bigV}_{n,M}^{0}(s,x)-\smallV(s,x)|^2\big]\Bigr]^{\!\nicefrac{1}{2}}.\label{k10}
\end{split}\end{align}
Observe that \eqref{k10} ensures that for all $n\in \{ 0, 1, \ldots, N \}$ it holds that $f_n$ is
measurable. 
Furthermore, note that \eqref{k10} and \cref{k01} imply that for all $n\in \{ 1, 2, \ldots, N \}$, $t\in[\tau,T]$ it holds that
\begin{align}\begin{split}
|f_n(t)|&\leq 
\sup_{r\in[t,T]}\left[
\frac{a_1}{\sqrt{M^n}}+ \left[\sum_{\ell=0}^{n-1}
\frac{a_2}{\sqrt{M^{n-\ell-1}}}
\left[
\int_r^T|f_\ell(s)|^2\,ds\right]^{\!\nicefrac{1}{2}}\right]\right]\\
&\leq 
\frac{a_1}{\sqrt{M^n}}+ \left[\sum_{\ell=0}^{n-1}
\frac{a_2}{\sqrt{M^{n-\ell-1}}}
\left[
\int_t^T|f_\ell(s)|^2\,ds\right]^{\!\nicefrac{1}{2}}\right].
\end{split}\end{align}
\cref{g09}, \eqref{k10}, and the fact that for all $t\in[\tau,T]$ it holds that
\begin{align}
f_0(t)=\sup_{s\in[t,T]}\sup_{x\in\R^d}\left(
\frac{e^{\growrate s/2}|\smallV(s,x)|}{\sqrt{\lyaV(x)}}\right)
\leq e^{\growrate T/2}\left[ \sup_{s\in[0,T]}\sup_{x\in\R^d}
\left(\frac{|\smallV(s,x)|}{\sqrt{\lyaV(x)}}\right]\right)<\infty
\end{align}
 therefore demonstrate that
\begin{align}\begin{split}
&\sup_{t \in[\tau,T]}\sup_{x\in\R^d}
\Bigl[
e^{\growrate t}\lyaV(x)^{-1}\E\!\left[\left|{\bigV}_{N,M}^{0}(t,x)-\smallV(t,x)\right|^2\right]\Bigr]^{\!\nicefrac{1}{2}}=
f_N({\tau})\\
&\leq \left[ a_1+a_2\sqrt{T}\textstyle\sup_{t\in [{\tau},T)}  |f_0(t)|\right]e^{M/2}M^{-N/2}(1+a_2\sqrt{T})^{N-1}
\\
&\leq e^{M/2}M^{-N/2}(1+2TL)^{N-1}
\\& \cdot
\left[2e^{\growrate T/2}\sup_{t\in[0,T]}\sup_{x\in\R^d}\left[\frac{\max\{|T(\funcF(0))(t,x)|,|g(x)|\}}{\sqrt{\lyaV(x)}}\right]+2TLe^{\growrate T/2} \sup_{t\in[0,T]}\sup_{x\in\R^d}\left[
\frac{|\smallV(t,x)|}{\sqrt{\lyaV(x)}}\right]\right].
\end{split}\end{align}
This completes the proof of \cref{x01}.
\end{proof}


\section{Computational complexity analysis for MLP approximations}\label{sec:main_result}

In this section we combine the existence, uniqueness, and regularity properties for solutions of stochastic fixed point equations, which we have established in \cref{s02}, with the error analysis for MLP approximations for stochastic fixed point equations, which we have established in \cref{sec:rate} (see \cref{x01} in \cref{x010}), to obtain in \cref{m01_x} in \cref{sec:fix_dim} a computational complexity analysis for MLP approximations 
for semilinear second-order PDEs 
in fixed space dimensions. In \cref{sec:var_dim} we combine the computational complexity analysis in \cref{m01_x} with the elementary auxiliary result in \cref{a01} to obtain in \cref{m01c}  a computational complexity analysis for MLP approximations 
for semilinear second-order PDEs 
in variable space dimensions.

\subsection{Error bounds for MLP approximations involving Euler-Maruyama approximations}

\renewcommand{\lyaV}{\varphi}
\begin{proposition}
\label{m01}
\newcommand{\bfdelta}{\boldsymbol{\delta}}
Let $d,m,M,K\in\N$,
  $\beta,b,c\in [1,\infty)$,  
$p\in[2\beta,\infty)$, 
$\lyaV\in C^2(\R^d, [1,\infty))$,
$\funcG\in C( \R^d,\R)$,
$\mu\in C( \R^{d} ,\R^{d})$, $\sigma=(\sigma_1,\sigma_2,\ldots,\sigma_m)\in C( \R^{d},\R^{d\times m})$,
$T,\tau_0, \tau_1,\ldots, \tau_K\in \R$ satisfy $0=\tau_0<\tau_1<\ldots<\tau_K=T$,
let $\lVert\cdot \rVert\colon \R^{d}\to[0,\infty)$ be the standard norm on $\R^d$, 
 let
$\smallF\colon [0,T]\times \R^{d}\times \R \to\R$ be 
measurable,
let $\funcF\colon\R^{[0,T]\times\R^d}\to\R^{[0,T]\times\R^d}$ satisfy for all
$t\in [0,T]$,
$x\in \R^d$,
 $v\in \R^{[0,T]\times\R^d}$ that
$
(\funcF(v))(t,x)= f(t,x,v(t,x)),
$
assume  for all  $x,y\in\R^d$, $z\in \R^d\backslash \{0\}$, $t\in[0,T]$,  $v,w\in\R $  that 
\begin{gather}
\label{v05b}
\max\bigl\{\tfrac{|( \lyaV'(x))(z)|}{(\lyaV(x))^{(p-1)/p}\|z\|}, \tfrac{|( \lyaV'' (x))(z,z)|}{(\lyaV(x))^{(p-2)/p}\|z\|^2},
\tfrac{c\|x\|+\|\mu(0)\|}{(\lyaV(x))^{1/p}}, 
\tfrac{c\|x\|+[ \sum_{ i=1 }^m \| \sigma_i(0) \|^2 ]^{ 1/2 }} {(\lyaV(x))^{1/p}} \bigr\}
\leq c,
\\
\label{v04b}
\max\bigl\{|Tf(t,x,0)|,|g(x)|\bigr\}\leq b(\lyaV(x))^{\beta/p},
\\
\label{v07}
\max\bigl\{|g(x)-g(y)|,T|f(t,x,v)-f(t,y,w)|\bigr\}\leq cT|v-w|+ \tfrac{(\lyaV(x)+\lyaV(y))^{\beta/p}
\|x-y\|}{T^{1/2}b^{-1}}
,
\\
\max\bigl\{\|\mu(x)-\mu(y)\|^2,
 \smallsum_{ i=1 }^m \| \sigma_i(x)-\sigma_i(y) \|^2 
\bigr\}
\leq c^2\|x-y\|^2,\label{v03-2}
\end{gather}
let $(\Omega,\mathcal{F},\P, (\F_t)_{t\in[0,T]})$ be a filtered probability space which satisfies the usual conditions\footnote{Let $T \in [0,\infty)$ and let ${\bf \Omega} = (\Omega,\mathcal{F},\P, (\F_t)_{t\in[0,T]})$ be 
a filtered probability space. 
Then we say that ${\bf \Omega}$
satisfies the usual conditions if and only if 
it holds that $\{ A\in \mathcal F: \P(A)=0 \} \subseteq \F_0$ and $\forall \, t\in [0,T) \colon \mathbb{F}_t 
= \cap_{ s \in (t,T] } \F_s$.},
let 
$  \Theta = \bigcup_{ n \in \N }\! \Z^n$,
let $\unif^\theta\colon \Omega\to[0,1]$, $\theta\in \Theta$, be i.i.d.\ random variables,
assume for all $t\in(0,1)$ that $\P(\unif^0\le t)=t$, 
let $\uniform^\theta\colon [0,T]\times \Omega\to [0, T]$, $\theta\in\Theta$, satisfy 
for all $\theta\in \Theta$, $t\in [0,T]$ that 
$\uniform^\theta _t = t+ (T-t)\unif^\theta$,
 let $W^\theta\colon [0,T]\times\Omega \to \R^{m}$, $\theta\in\Theta$, be i.i.d.\ standard $(\F_{t})_{t\in[0,T]}$-Brownian motions, 
assume that
$(\unif^\theta)_{\theta\in\Theta}$ and
$(W^\theta)_{\theta\in\Theta}$ are independent,
let
$\rdownni{\cdot}{\tau}\colon \R\to \R $ satisfy for all $t\in \R$ that
$\rdownni{t}{\tau}=\max\bigl(\{\tau_0,\tau_1, \ldots, \tau_n\} \cap ((-\infty,t)\cup \{\tau_0\}) \bigr)$,
for every 
$\theta\in\Theta$,
$t\in[0,T]$, $x\in\R^d$
let
$
Y^{\theta,x}_{t}=(Y^{\theta,x}_{t,s})_{s\in[t,T]}\colon[t,T]\times\Omega\to\R^d$ 
 satisfy for all $s\in[t,T]$ that $Y_{t,t}^{\theta,x}=x$ and
\begin{align}\begin{split}\label{v01}
&Y_{t,s}^{\theta,x} - Y_{t,\max\{t, \rdownni{s}{\tau}\}}^{\theta,x}\\
&=  \mu(Y_{t,\max\{t, \rdownni{s}{\tau}\}}^{\theta,x})\bigl(s-\max\{t, \rdownni{s}{\tau}\}\bigr)+\sigma(Y_{t,\max\{t, \rdownni{s}{\tau}\}}^{\theta,x})\bigl(W^\theta_{s} -W^\theta_{\max\{t, \rdownni{s}{\tau}\}} \bigr),
\end{split}\end{align}
and let
$ 
  {\bigV}_{ n}^{\theta} \colon [0, T] \times \R^d \times \Omega \to \R
$, $n\in\Z$, $\theta\in\Theta$,
satisfy 
for all 
$\theta\in\Theta $,
$n\in \N_0$, 
$ t \in [0,T]$, $x\in\R^d $
that 
\begin{equation}\label{v02}
\begin{split}
  &{\bigV}_{n}^{\theta}(t,x)
=
  \frac{\1_{\N}(n)}{M^n}
 \sum_{i=1}^{M^n} 
      \funcG \big(Y^{(\theta,0,-i),x}_{t,T}\big)
 \\
&  +
  \sum_{\ell=0}^{n-1} \frac{(T-t)}{M^{n-\ell}}
    \left[\sum_{i=1}^{M^{n-\ell}}
      \bigl(\funcF\bigl({\bigV}_{\ell}^{(\theta,\ell,i)}\bigr)-\1_{\N}(\ell)\funcF\bigl( {\bigV}_{\ell-1}^{(\theta,-\ell,i)}\bigr)\bigr)
      \bigl(\uniform_t^{(\theta,\ell,i)},Y_{t,\uniform_t^{(\theta,\ell,i)}}^{(\theta,\ell,i),x}\bigr)
    \right].
\end{split}
\end{equation}
Then
\begin{enumerate}[(i)]
\item \label{k20}
for every $t\in[0,T]$, $\theta\in\Theta$
there exists  
  an up to indistinguishability unique continuous random  field 
$
( X^{ \theta, x }_{ t, s } )_{ (s,x)\in[t,T]\times \R^d} \colon
[t,T]\times\R^d\times \Omega\to\R^d$ 
which satisfies that for all $x\in\R^d$ it holds
that
$( X^{ \theta, x }_{ t, s } )_{s\in[t,T] }$
is $(\F_{s})_{s\in[t,T]}$-adapted and which satisfies that for all
$s\in[t,T]$, $x\in\R^d$
it holds 
$\P$-a.s.\ that
\begin{align}
X_{t,s}^{\theta,x}= x+\int_{t}^s \mu(X_{t,r}^{\theta,x})\,dr +\int_t^s \sigma(X_{t,r}^{\theta,x})\,dW_r^\theta,
\end{align}
\item \label{k20b}
it holds
for all $\theta\in\Theta$, $t\in[0,T]$, $s\in[t,T]$, $r\in[s,T]$, $x\in\R^d$ that
$
\P\bigl(X_{s,r}^{\theta,X_{t,s}^{\theta,x}}= X_{t,r}^{\theta,x}\bigr)=1
$,
\item \label{k21}there exists a unique
measurable
$u\colon [0,T]\times\R^d\to\R$ which satisfies
for all $t\in[0,T]$, $x\in \R^d$ that
$
\bigl(\sup_{s\in[0,T],y\in\R^d} [{|\smallU(s,y)|}{(\lyaV(y))^{-\beta /p}}]\bigr)+\int_{t}^{T}\E\bigl[|f(s,X_{t,s}^{0,x},u(s,X_{t,s}^{0,x}))| \bigr]\,ds
+
\E\bigl[|g(X^{0,x}_{t,T})|\bigr] \allowbreak < \infty$
and
\begin{align}
u(t,x)=\E\!\left[g(X^{0,x}_{t,T})\right]+\int_{t}^{T}\E\!\left[f(s,X_{t,s}^{0,x},u(s,X_{t,s}^{0,x}))\right]ds ,
\end{align}%
\item \label{k22}it holds for all $t\in[0,T]$, $x\in\R^d$, 
$n\in\N_0$, $\theta \in \Theta$
 that ${\bigV}_{n}^{\theta}(t,x)$ is 
 measurable, and
\item \label{k22b} 
 it holds for all $t\in[0,T]$, $x\in\R^d$, 
$n\in\N_0$ that
\begin{align}\small\begin{split}
&\frac{\big(\E\big[|{\bigV}_{n}^{0}(t,x)-\smallU(t,x)|^2\big]\big)^{\!\nicefrac{1}{2}}}
{12bc^2 \! \exp(9c^3T)
|\lyaV(x)|^{(\beta+1)/p}}
\leq 
\left[
\frac{ \exp( 2 n c T + \frac{ M }{ 2 } ) }{ M^{ \nicefrac{n}{2} } }
+\max\limits_{i\in \{1,2,\ldots,K\}}\left(\tfrac
{|\tau_{i}-\tau_{i-1}|^{\nicefrac{1}{2}}}{T^{\nicefrac{1}{2}}}\right)\right].\end{split}
\end{align}
\end{enumerate}
\end{proposition}

\begin{proof}[Proof of \cref{m01}]\sloppy
\newcommand{\bfdelta}{\boldsymbol{\tau}}
\newcommand{\Xprop}{\mathfrak{X}}
\newcommand{\uprop}{\mathscr{u}}
\newcommand{\linbm}{\mathfrak{Y}}
Throughout this proof let $\tripleNorm{\cdot}\colon (\bigcup_{L,N\in\N}\R^{L\times N})\to[0,\infty)$ satisfy for all $L,N\in\N$,
$
A=(A_{i,j})_{ (i,j) \in \{1,2,\ldots,L\}\times \{1,2,\ldots,N\}}\in \R^{L\times N}
$  
that $\tripleNorm{A}=[\sum_{i=1}^{L}\sum_{j=1}^{N}|A_{ij}|^2]^{1/2}$, 
let
$
\size{\tau} =\max_{i\in \{1,2,\ldots,K\}} |\tau_{i}-\tau_{i-1}|
$,
let 
$\Delta\subseteq [0,T]^2$ satisfy $\Delta= \{(t,s)\in[0,T]^2\colon t\leq s\}$,
let $\Xprop^k=(\Xprop^{k,x}_{t,s})_{(t,s,x)\in \Delta \times \R^d}\colon \Delta \times \R^d \times \Omega \to \R^d$, $k\in \{0,1\}$, satisfy for all $t\in [0,T]$, $s\in [t,T]$, $x\in \R^d$ that $\Xprop^{0,x}_{t,s}=X^{0,x}_{t,s}$ and $\Xprop^{1,x}_{t,s}=Y^{0,x}_{t,s}$,
let $\linbm^x=(\linbm^x_t)_{t\in [0,T]}\colon [0,T]\times \Omega \to \R^d$, $x\in \R^d$, satisfy for all $x\in \R^d$, $t\in [0,T]$ that $\linbm^x_t=x+\mu(x)t+\sigma(x)W_{t}$, and let $\tau^x_n\colon \Omega \to [0,T]$, $x\in \R^d$, $n\in \N$, satisfy for all $n\in \N$, $x\in \R^d$ that
$\tau^x_n=\inf(\{T\}\cup \{t\in [0,T]\colon [\sup_{s\in [0,t]}\lyaV(\linbm^x_s)]+\int_0^t\sum_{i=1}^m |(\lyaV'(\linbm^x_s))(\sigma_i(x))|^2\,ds\ge n\} )$. 
Observe that
 \eqref{v03-2} establishes \cref{k20,k20b}
 (cf., e.g., Rogers \& Williams~\cite[Theorem~13.1 and Lemma~13.6]{RogersWilliams2000b}).
Next note that \eqref{v03-2} and
\eqref{v05b} imply that for all $x\in\R^d$ it holds that
\begin{equation}
\begin{split}
\max\{\|\mu(x)\|,\vertiii{\sigma(x)}\}&\leq 
\max\{\|\mu(x)-\mu(0)\|+\|\mu(0)\|,\vertiii{\sigma(x)-\sigma(0)}
+\vertiii{\sigma(0)}\}\\
&\leq 
\max\{c\|x\|+\|\mu(0)\|,c\|x\|+\vertiii{\sigma(0)}\}\leq 
c(\lyaV(x))^{\frac{1}{p}}.\label{r11}
\end{split}
\end{equation}
This, \eqref{v05b}, and the fact that
$\forall\,a,b\in [0,\infty),\lambda\in (0,1)\colon a^\lambda b^{1-\lambda}\leq \lambda a+(1-\lambda)b$ imply that for all $x,y\in\R^d$ it holds that 
\begin{align}\small\begin{split}
&\left|( \lyaV'(y))(\mu(x))\right|+\frac{1}{2}\left|\sum_{k=1}^{m}
(\lyaV''(y))(\sigma_k(x),\sigma_k(x))\right|\\
&
\leq
c(\lyaV(y))^{1-\frac{1}{p}}
\|\mu(x)\|+\tfrac{c}{2} 
(\lyaV(y))^{1-\frac{2}{p}}\sum_{k=1}^{m}
\|\sigma_k(x)\|^2=
c(\lyaV(y))^{1-\frac{1}{p}}
\|\mu(x)\|+\tfrac{c}{2} 
(\lyaV(y))^{1-\frac{2}{p}}\vertiii{\sigma(x)}^2
\\
&
\leq  c(\lyaV(y))^{1-\frac{1}{p}}c(\lyaV(x))^{\frac{1}{p}}+\tfrac{c}{2}(\lyaV(y))^{1-\frac{2}{p}}c^2(\lyaV(x))^\frac{2}{p}
\\
&
\leq
c^2\left[\left(1-\tfrac{1}{p}\right)\lyaV(y)+\tfrac{1}{p}\lyaV(x)\right]+ \tfrac{c^3}{2}\left[\left(1-\tfrac{2}{p}\right)\lyaV(y)+\tfrac{2}{p}\lyaV(x)\right]\\
&\leq \left[c^3\left(1-\tfrac{1}{p}\right)+\tfrac{c^3}{2}\left(1-\tfrac{2}{p}\right)\right]\lyaV(y)
+
\left[\tfrac{c^3}{p}+\tfrac{2c^3}{2p}
\right]\lyaV(x)
= \left(\tfrac{3c^3}{2}-\tfrac{2c^3}{p}\right)\lyaV(y)+\tfrac{2c^3}{p}\lyaV(x).
\end{split}\label{r04}\end{align}
Combining this and, e.g., Cox et al.~\cite[Lemma~2.2]{CoxHutzenthalerJentzen2014} (applied
for every $t\in[0,T)$,
$s\in[t,T]$,  $x\in\R^d$, $\theta\in\Theta$
with
$T\defeq T-t$,
$O\defeq\R^d$,
$V\defeq( [0,T-t]\times\R^d\ni(s,x)\mapsto\lyaV(x)\in[0,\infty))$, $\alpha\defeq ([0,T-t]\ni s\mapsto 2c^3 \in [0,\infty) )$, 
$\tau\defeq s-t$,
$X\defeq (X^{\theta,x}_{t,t+r})_{r\in[0,T-t]}$  
 in the notation of Cox et al.~\cite[Lemma~2.2]{CoxHutzenthalerJentzen2014}) demonstrates that for all
$\theta\in\Theta$, $x\in\R^d$, $t\in[0,T]$, $s\in [t,T]$ it holds
  that 
\begin{align}\label{r05}
\E \bigl[\lyaV(X_{t,s}^{\theta,x})\bigr]\leq e^{2c^3(s-t)}\lyaV(x).
\end{align}
It\^o's formula, \cref{r04}, and the fact that $\lyaV\ge 1$ imply that for all $x\in\R^d$, $t\in[0,T]$ it holds that
\begin{equation}
\begin{split}
&\E[\lyaV(\linbm^x_{\min\{\tau_n^x,t\}})]\\
&=
\lyaV(x)+\E\!\left[\int_0^{\min\{\tau_n^x,t\}}
( \lyaV'(\linbm^x_s))(\mu(x))+\frac{1}{2}\sum_{k=1}^{m}
(\lyaV''(\linbm^x_s))(\sigma_k(x),\sigma_k(x))\,ds\right]\\
&\le 
\lyaV(x)+\E\!\left[\int_0^{\min\{\tau_n^x,t\}}
\left(\tfrac{3c^3}{2}-\tfrac{2c^3}{p}\right)\lyaV(\linbm_s^x)+\tfrac{2c^3}{p}\lyaV(x)\,ds\right]\\
&\le 
\lyaV(x)\left(1+\tfrac{2c^3t}{p}\right)+\left(\tfrac{3c^3}{2}-\tfrac{2c^3}{p}\right)\E\!\left[\int_0^{t}
 \lyaV(\linbm_s^x)\1_{[0,\tau_n^x]}(s) \,ds\right]\\
&\le 
\lyaV(x)\left(1+\tfrac{2c^3t}{p}\right)+\left(\tfrac{3c^3}{2}-\tfrac{2c^3}{p}\right)\int_0^{t}
\E[\lyaV(\linbm_{\min\{\tau_n^x,s\}}^x)]\,ds.
\end{split}
\end{equation}
Gronwall's inequality and
the fact that for all $a\in \R$ it holds that $ 1+ a\leq e^{a}$ therefore assure that for all $x\in\R^d$, $t\in[0,T]$ it holds that 
\begin{align}
\E[\lyaV(\linbm^x_{\min\{\tau_n^x,t\}})]
\leq \exp\!\left( \left[ \tfrac{ 3c^3}{ 2 } - \tfrac{ 2c^3}{ p } \right] t \right) 
\left[ 1 + \tfrac{ 2c^3t }{p } \right] \lyaV(x)
\leq e^{2c^3t}\lyaV(x).
\end{align}
Fatou's lemma hence proves that for all $x\in\R^d$, $t\in[0,T]$ it holds that
\begin{align}
\E\bigl[
\lyaV(x+\mu(x)t+\sigma(x)W_{t})\bigr]
=\E[\lyaV(\linbm^x_{t})]
\leq e^{2c^3t}\lyaV(x)
\end{align}%
(cf., e.g., also Hudde et al.~\cite[Theorem~2.4]{HMH19}).
The tower property for conditional expectations, 
the fact that for all $t\in[0,T]$, $s\in[t,T]$, $\theta\in\Theta$ it holds that
$W^\theta_{s}-W^\theta_{t}$ and $\F_t$ are independent,
 and the fact that for all $t\in[0,T]$, $s\in[t,T]$, $\theta\in\Theta$, $B\in\mathcal{B}(\R^d)$ it holds that $\P((W^\theta_{s}-W^\theta_{t})\in B)=\P( W^\theta_{s-t}\in B)$ hence prove that for all
$\theta\in\Theta$,
$x\in\R^d$, $t\in[0,T]$, $s\in [t,T]$ 
it holds 
  that
\begin{align}\begin{split}
&\E \bigl[\lyaV(Y_{t,s}^{\theta,x})\bigr]\\
&=\E\biggl[\E \Bigl[\lyaV\bigl(Y_{t,\max\{t,\rdownni{s}{\tau}\}}^{\theta,x}+
\mu(Y_{t,\max\{t,\rdownni{s}{\tau}\}}^{\theta,x})(s- \max\{t,\rdownni{s}{\tau}\})\\
&\quad+
\sigma(Y_{t,\max\{t,\rdownni{s}{\tau}\}}^{\theta,x})(W^{ {\theta}}_{s} -W^{ {\theta}}_{ \max\{t,\rdownni{s}{\tau}\}} )\bigr)\Big|\F_{\rdownni{s}{\tau}}\Bigr] \biggr]
\\
&
=\E\!\left[\E \Bigl[\lyaV \bigl(z+
\mu(z)(s- \max\{t,\rdownni{s}{\tau}\})+
\sigma(z)(W^{ {\theta}}_{s- \max\{t,\rdownni{s}{\tau}\}} ) \bigr) \Bigr]\Bigr|_{ z=Y_{t,\max\{t,\rdownni{s}{\tau}\}}^{\theta,x}}\right]\\
&\leq e^{2c^3(s-\max\{t,\rdownni{s}{\tau}\})}
\E\!\left[
\lyaV\bigl(Y_{t,\max\{t,\rdownni{s}{\tau}\}}^{\theta,x}\bigr)\right].
\end{split}\end{align}
Induction and \eqref{v01} hence show that
for all 
$\theta\in\Theta$,
$x\in\R^d$, $t\in[0,T]$, $s\in [t,T]$ it holds
  that 
$
\E \bigl[\lyaV(Y_{t,s}^{\theta,x})\bigr]\leq e^{2c^3(s-t)}\lyaV(x).
$ 
Jensen's inequality and \eqref{r05} therefore prove that 
for all $q\in[0,p]$,
$\theta\in\Theta$,
$x\in\R^d$, $t\in[0,T]$, $s\in [t,T]$ it holds
that
\begin{align}\label{r06}\begin{split}
&\max\!\big\{\E \bigl[(\lyaV(Y_{t,s}^{\theta,x}))^{\frac{q}{p}}\bigr],
\E \bigl[(\lyaV(X_{t,s}^{\theta,x}))^{\frac{q}{p}}\bigr]\big\}\\
&
\leq 
\max\!\left\{
\left(
\E \bigl[\lyaV(Y_{t,s}^{\theta,x})\bigr]\right)^{\frac{q}{p}},
\left(\E \bigl[\lyaV(X_{t,s}^{\theta,x})\bigr]\right)^{\frac{q}{p}}\right\}
\leq e^{2qc^3(s-t)/p}(\lyaV(x))^{\frac{q}{p}}.
\end{split}\end{align}
Moreover, observe that the fact that $\mu$ is continuous,
the fact that $\sigma$ is continuous, the fact that for all 
$\theta\in\Theta$, $\omega \in \Omega$ it holds that
$[0,T]\ni t\mapsto W^\theta_t(\omega)\in \R^d$ is continuous,
and Fubini's theorem imply that for all
$\theta\in\Theta$ and all
measurable
$\eta\colon  [0,T]\times\R^d\to[0,\infty)$ it holds
that  
\begin{equation}\label{v09d}
\Delta\times\R^d \ni (t,s,x)\mapsto
\E\bigl[\eta\bigl(s,Y_{t,s}^{\theta,x}\bigr)\bigr]\in[0,\infty]
\end{equation}
is measurable.
Furthermore, note that \eqref{v05b}, \eqref{v03-2}, \eqref{r04}, and, e.g., Beck et al.~\cite[Lemma~3.7]{beck2019existence} (applied with $\mathcal{O}\defeq\R^d$,  
$V\defeq
([0,T]\times\R^d\ni (t,x)\mapsto e^{-2c^3 t/p}\lyaV(x) \in(0,\infty))
$
 in the notation of Beck et al.~\cite[Lemma~3.7]{beck2019existence})
imply that
$\Delta\times\R^d\times\R^d \ni (t,s,x,y)\mapsto
\bigl(s,X_{t,s}^{\theta,x},X_{t,s}^{\theta,y}\bigr)\in\mathcal{L}^0(\Omega;\R\times\R^d\times\R^d )$ is continuous.
This and the dominated convergence theorem prove that for all $\theta\in\Theta$ and all bounded and
continuous $\eta\colon  [0,T]\times\R^d\times\R^d\to[0,\infty)$ it holds that
$ 
\Delta\times\R^d\times\R^d \ni (t,s,x,y)\mapsto
\E\bigl[\eta\bigl(s,X_{t,s}^{\theta,x},X_{t,s}^{\theta,y}\bigr)\bigr]\in[0,\infty]$
is continuous. Hence, we obtain that 
for all $\theta\in\Theta$ and all bounded and
continuous $\eta\colon  [0,T]\times\R^d\times\R^d\to[0,\infty)$ it holds that
$ 
\Delta\times\R^d\times\R^d \ni (t,s,x,y)\mapsto
\E\bigl[\eta\bigl(s,X_{t,s}^{\theta,x},X_{t,s}^{\theta,y}\bigr)\bigr]\in[0,\infty]$
is measurable.
This implies
 that for all $\theta\in\Theta$
and all
measurable
$\eta\colon  [0,T]\times\R^d\times\R^d\to[0,\infty)$ it holds
that  
\begin{align}\label{v09e}
\Delta\times\R^d\times\R^d \ni (t,s,x,y)\mapsto
\E\bigl[\eta\bigl(s,X_{t,s}^{\theta,x},X_{t,s}^{\theta,y}\bigr)\bigr]\in[0,\infty]
\end{align}
is measurable.
Combining
 \eqref{v09d}, \eqref{r06},
\eqref{v04b}, \eqref{v07}, and
\cref{b01} (applied
for every 
$ k\in \{0,1\}$
with $L\defeq c$,
$\mathcal{O}\defeq \R^d$, 
$ (X_{t,s}^x)_{(t,s,x)\in \Delta \times \R^d}\defeq  (\mathfrak{X}^{k,x}_{t,s})_{(t,s,x)\in\Delta \times \R^d}$, 
$
V\defeq ([0,T]\times\R^d\ni(s,x)\mapsto e^{2c^3  \beta(T-s)/p}(\lyaV (x))^{\beta /p}\in(0,\infty))
$
 in the notation of \cref{b01}) hence
 establishes that 
\begin{enumerate}[a)]
\item there exist unique
measurable
$\uprop_k \colon [0,T]\times\R^d\to\R$, $k\in \{0,1\}$,
which satisfy for all $k\in \{0,1\}$, $t\in[0,T]$, $x\in\R^d$ 
that
$\sup_{s\in[0,T]}\sup_{x\in\R^d} \bigl[|\uprop_k(s,x)|{(\lyaV(x))^{-\beta /p}}\bigr]+\E\bigl[\bigl|g\bigl(\Xprop^{k,x}_{t,T}\bigr)\bigr|+\int_{t}^{T}\bigl|f\bigl(s,\Xprop_{t,s}^{k,x},\uprop_k\bigl(s,\Xprop_{t,s}^{k,x}\bigr)\bigr)\bigr|\,ds \bigr]<\infty$
and 
\begin{align}\begin{split}
&\uprop_k(t,x)=\E\!\left[g\bigl(\Xprop^{k,x}_{t,T}\bigr)+\int_{t}^{T}f\bigl(s,\Xprop_{t,s}^{k,x},\uprop_k(s,\Xprop_{t,s}^{k,x})\bigr)\,ds \right]
\end{split}\label{v15}\end{align}
and
\item 
it holds for all $k\in \{0,1\}$ that 
\begin{align}\label{v14}\small\begin{split}
\sup_{t\in[0,T]}\sup_{x\in\R^d} \left[\frac{|\uprop_k(t,x)|}{e^{2c^3 \beta  (T-t)/p}(\lyaV(x))^{\beta /p}}\right]\leq 
\sup_{t\in[0,T]}\sup_{x\in\R^d}\left[\left[\frac{|g(x)|}{(\lyaV(x))^{\beta/p}}
+\frac{|Tf(t,x,0)|}{(\lyaV(x))^{\beta /p}}\right]e^{cT}\right]
\leq
2be^{cT}.
\end{split}\end{align}
\end{enumerate}
This proves \cref{k21}. Moreover, note that \cref{ppt0} establishes \cref{k22}.
Next observe that \eqref{r11} and \eqref{r06} demonstrate that for all 
$\theta\in\Theta$, $t\in[0,T]$, $r\in[t,T]$, $x\in\R^d$ it holds that
\begin{align}\label{r12}\begin{split}
&\max\!\left\{\E\!\left[\bigl\|\mu(Y_{t,\max\{t,\rdownni{r}{\tau}\}}^{\theta,x})\bigr\|^2\right],
\E\!\left[\vertiii{\sigma(Y_{t,\max\{t, \rdownni{r}{\tau}\}}^{\theta,x})}^2\right]\right\}\\&\leq c^2
\E\!\left[\bigl(\lyaV\bigl(Y_{t,\max\{t, \rdownni{r}{\tau}\}}^{\theta,x}\bigr)\bigr)^{\nicefrac{2}{p}}\right]\leq c^2 e^{4c^3(r-t)/p}(\lyaV(x))^{\nicefrac{2}{p}}.
\end{split}\end{align}
Furthermore, note that \eqref{v01} demonstrates that 
for all $t\in[0,T]$, 
$r\in[t,T]$,
 $x\in\R^d$, $\theta\in\Theta$
 it holds that
$ \sigma (\{Y_{t,\max\{t, \rdownni{r}{\tau}\}}^{\theta,x}\})\subseteq {\F_r}$.
Combining this and \eqref{r12} with the fact that for all $t\in[0,T]$,  $x\in\R^d$ it holds that 
$\E\!\left[\big.\!\left\| \sigma(x)W_t\right\|  ^{2} \right]
=
\vertiii{\sigma(x)}^2{t}
$ shows that for all
$\theta\in\Theta$,
 $t\in[0,T]$, $r\in[t,T]$, $x\in\R^d$ it holds that
\begin{align}\begin{split}
&\E\!\left[\bigl\|\sigma(Y_{t,\max\{t, \rdownni{r}{\tau}\}}^{\theta,x})(W^\theta_{r} -W^\theta_{\max\{t, \rdownni{r}{\tau}\}} )\bigr\|^2\right]
\\
&
=
\E\!\left[\E\!\left[\bigl\|\sigma(y)(W^\theta_{r} -W^\theta_{\max\{t, \rdownni{r}{\tau}\}} )\bigr\|^2\right]\Bigr|_{y= Y_{t,\max\{t, \rdownni{r}{\tau}\}}^{\theta,x}}\right]
\\
&
=
\E\!\left[ 
\vert \kern-0.25ex\vert \kern-0.25ex\vert \sigma(Y_{t,\max\{t, \rdownni{r}{\tau}\}}^{\theta,x})\vert \kern-0.25ex \vert \kern-0.25ex \vert
^2(r- \max\{t, \rdownni{r}{\tau}\} )\right]
\\
&
\leq 
\E\!\left[ 
\vert \kern-0.25ex\vert \kern-0.25ex\vert \sigma(Y_{t,\max\{t, \rdownni{r}{\tau}\}}^{\theta,x})\vert \kern-0.25ex \vert \kern-0.25ex \vert
^2{\size{\tau}}\right]
\leq  c^2 e^{4c^3(r-t)/p}(\lyaV(x))^{\nicefrac{2}{p}}{\size{\tau}} .
\end{split}\end{align}
This, \eqref{v01}, the triangle inequality, and \eqref{r12}
 imply that for all
$\theta\in\Theta$, $t\in[0,T]$, $r\in[t,T]$, $x\in\R^d$ it holds that 
\begin{align}\begin{split}
&\left(\E\!\left[ \bigl\|Y_{t,\max\{t, \rdownni{r}{\tau}\}}^{\theta,x}-
Y_{t,r}^{\theta,x}\bigr\|^2 \right]\right)^{\nicefrac{1}{2}}\\
&\leq 
\left(\E\!\left[\bigl\|\mu(Y_{t,\max\{t, \rdownni{r}{\tau}\}}^{\theta,x})\bigr\|^2\right]\right)^{\nicefrac{1}{2}}
(r-\max\{t, \rdownni{r}{\tau}\})\\&\quad
+\left(\E\!\left[\bigl\|\sigma(Y_{t,\max\{t, \rdownni{r}{\tau}\}}^{\theta,x})(W^\theta_{r} -W^\theta_{\max\{t, \rdownni{r}{\tau}\}} )\bigr\|^2\right]\right)^{\nicefrac{1}{2}}\\
&\leq ce^{2c^3 (r-t)/p}(\lyaV(x))^{\nicefrac{1}{p}}
\size{\tau}^{\nicefrac{1}{2}}
|r-t|^{\nicefrac{1}{2}}
+
 ce^{2c^3 (r-t)/p}(\lyaV(x))^{\nicefrac{1}{p}}{\size{\tau}}^{\nicefrac{1}{2}}
\\
&=
 c\bigl[|r-t|^{\nicefrac{1}{2}}+1\bigr]
e^{2c^3 (r-t)/p}(\lyaV(x))^{\nicefrac{1}{p}}
{\size{\tau}}^{\nicefrac{1}{2}}
.\label{r35}
\end{split}\end{align}
Next note that \eqref{v03-2} and the fact that $c\geq 1$ assure that for all $ {z},y\in\R^d$ with $ {z}\neq y$ it holds that
\begin{align}\footnotesize\begin{aligned}\label{r34}
\frac{\langle  {z}-y,\mu( {z})-\mu(y)\rangle+\tfrac{1}{2}\vertiii{\sigma( {z})-\sigma(y)}^2}{\| {z}-y\|^2}+\frac{(\tfrac{2}{2}-1)\|(\sigma ( {z})-\sigma(y))^{\mathsf{T}}( {z}-y)\|^2}{\| {z}-y\|^4}\end{aligned}\leq   2c^2.
\end{align}
This, \cite[Theorem 1.2]{HJ20} (applied
for every $\theta\in\Theta$,
$t\in[0,T)$, $s\in(t,T]$,
$x\in\R^d$ 
with
$
H\defeq \R^d$, 
$U\defeq \R^m$, 
$D\defeq\R^d$,
$T\defeq (s-t)$, $({\F}_r)_{r\in [0,T]}\defeq (\F_{r+t})_{r\in[0,s-t]}$,
$
(W_r)_{r\in [0,T]}\defeq 
(W_{t+r}^{\theta}-W_t^{\theta})_{r\in [0,s-t]}
$,
$(X_r)_{r\in[0,T]}\defeq (X^{\theta,x}_{t,t+r})_{r\in [0,s-t]}$,
$(Y_r)_{r\in[0,T]}\defeq (Y^{\theta,x}_{t,t+r})_{r\in [0,s-t]}$,
$(a_r)_{r\in[0,T]}\defeq (\mu(Y^{\theta,x}_{t,\max\{t,\rdownni{t+r}{\tau}\}}))_{r\in [0,s-t]}$,
$(b_r)_{r\in[0,T]}\defeq (\sigma(Y^{\theta,x}_{t,\max\{t,\rdownni{t+r}{\tau}\}}))_{r\in [0,s-t]}$, $\epsilon\defeq 1$,
$p\defeq 2$, $\tau \defeq (\Omega\ni \omega\mapsto s-t\in[0,s-t])$,
$\alpha\defeq 1$, $\beta \defeq 1$, $r\defeq 2$, 
$q\defeq \infty$
in the notation of \cite[Theorem 1.2]{HJ20}), \eqref{v03-2}, \eqref{r35}, the fact that for all $t\in[0,\infty)$ it holds that $\sqrt{t}(\sqrt{t}+1)\leq e^t$, the fact that $1\leq c$, and the fact that $p\geq 2$ imply that for all
$\theta\in\Theta$, $t\in[0,T]$, $s\in[t,T]$, $x\in\R^d$ it holds that
\begin{equation}\small\begin{split}
&\left(\E\!
\left[\bigl\|X_{t,s}^{\theta,x}-Y_{t,s}^{\theta,x}\bigr\|^2\right]\right)^{\nicefrac{1}{2}}\\
&
\leq 
\sup_{\substack{ {z},y\in\R^d,\\ {z}\neq y}}
\exp\! \left(\int_{t}^{s}\left[\frac{\langle  {z}-y,\mu( {z})-\mu(y)\rangle+
\frac{(2-1)(1+1)}{2}
\vertiii{\sigma( {z})-\sigma(y)}^2}{\| {z}-y\|^2}+\tfrac{1-\frac{1}{2}}{1}+\tfrac{\frac{1}{2}-\frac{1}{2}}{1}\right]^+
dr
\right)\\
&\quad\cdot
\Biggl[\left(\int_{t}^{s}\E\!\left[ \bigl\|\mu \bigl(Y_{t,\max\{t,\rdownni{r}{\tau}\}}^{\theta,x}\bigr)-
\mu\bigl (Y_{t,r}^{\theta,x}\bigr)\bigr\|^2 \right]dr\right)^{\nicefrac{1}{2}}\\
&\quad
+\sqrt{\tfrac{(2-1)(1+1)}{1}}\left(\int_{t}^{s}\E\!\left[ \vertiii{\sigma \bigl(Y_{t,\max\{t,\rdownni{r}{\tau}\}}^{\theta,x}\bigr)-
\sigma \bigl(Y_{t,r}^{\theta,x}\bigr)}^2 \right]dr\right)^{\nicefrac{1}{2}}
\Biggr]
\\
&\leq e^{3c^2(s-t)}3c\left(|s-t|\sup_{r\in[t,s]}\E\!\left[ \bigl\|Y_{t,\max\{t,\rdownni{r}{\tau}\}}^{\theta,x}-
Y_{t,r}^{\theta,x}\bigr\|^2 \right]\right)^{\nicefrac{1}{2}}\\
&\leq e^{3c^2(s-t)}3c 
|s-t|^{\nicefrac{1}{2}}
 c\bigl[|s-t|^{\nicefrac{1}{2}}+1\bigr]
e^{2c^3 (s-t)/p}(\lyaV(x))^{\nicefrac{1}{p}}
{\size{\tau}}^{\nicefrac{1}{2}}
\\&\leq 3c^2 e^{4c^2T} 
e^{2c^3 ( {s}-t)/p}(\lyaV(x))^{\nicefrac{1}{p}}
{\size{\tau}}^{\nicefrac{1}{2}}.\label{r16}
\end{split}
\end{equation}
 Next observe that
 \cref{k20}, \eqref{r34}, and, e.g.,
Cox et al.~\cite[Corollary~2.26]{CoxHutzenthalerJentzen2014} (applied
for every $t\in[0,T)$, $s\in(t,T]$
with
 $T\defeq s-t$,
$O\defeq \R^d$,
 $(\mathcal{F}_{r})_{r\in[0,T]} \defeq (\F_{t,t+r})_{r\in[0,s-t]} $,
 $(W_{r})_{r\in[0,T]} \defeq (W^0_{t,t+r}-W^0_t)_{r\in[0,s-t]} $,
$\alpha_0\defeq 0$, 
$\alpha_1\defeq 0$,
$\beta_0\defeq 0$,
$\beta_1\defeq 0$,
$c\defeq 2c^2$,
$r\defeq 2$, $p\defeq 2$, $q_0\defeq \infty$, $q_1\defeq \infty$,
$U_0\defeq (\R^d\ni x\mapsto 0\in \R)$,
$U_1\defeq (\R^d\ni x\mapsto 0\in [0,\infty))$,
$\overline{U}\defeq (\R^d\ni x\mapsto 0\in \R)$,
$(X^x_{r})_{r\in[0,T],x\in\R^d}\defeq (X_{t,t+r}^{0,x})_{r\in[0,s-t],x\in\R^d} $
in the notation of Cox et al.~\cite[Corollary~2.26]{CoxHutzenthalerJentzen2014}) demonstrate that  
for all  $t\in[0,T)$, $s\in(t,T]$, $x,y\in\R^d$ it holds that
$
\left(\E\!\left[
\|X_{t,s}^{0,x}-X_{t,s}^{0,y}\|^2\right]\right)^{\nicefrac{1}{2}}\leq e^{2c^2(s-t)}\|x-y\|.
$
This
and \eqref{r16}
imply that
for all $t\in[0,T]$, $s\in[t,T]$,
$r\in[s,T]$, $x,y\in\R^d$  it holds that
\begin{align}\begin{split}
&
\left(\E\!\left[
\E\!\left[\bigl\|X_{s,r}^{0,\mathfrak{x}}-X_{s,r}^{0,\mathfrak{y}}\bigr\|^{2}\right]\Bigr|_{\substack{(\mathfrak{x},\mathfrak{y})=(X^{0,x}_{t,s},Y_{t,s}^{0,x})}}\right]\right)^{\!\!\nicefrac{1}{2}}
\leq \left(\E\!\left[
\left[e^{2c^2(r-s)}\left\|X^{0,x}_{t,s}-Y_{t,s}^{0,x}\right\|\right]^2\right]\right)^{\!\nicefrac{1}{2}}\\&\leq 
e^{2c^2(r-s)}
3c^2 e^{4c^2T} 
e^{2c^3 (s-t)/p}(\lyaV(x))^{\nicefrac{1}{p}}
{\size{\tau}}^{\nicefrac{1}{2}}
\leq 
3c^2 e^{4c^2T} 
{\size{\tau}}^{\nicefrac{1}{2}}\bigl[ e^{4c^3 (T-t)/p}(\lyaV(x))^{\nicefrac{2}{p}}\bigr]^{\nicefrac{1}{2}}
.
\end{split}\label{v09c}\end{align}
Furthermore, note that \cref{k20} and Tonelli's theorem
 ensure that for all 
$t\in[0,T]$, $s\in[t,T]$, $r\in[s,T]$, $x,y\in\R^d$ and all measurable $h\colon \R^{d}\times\R^d\to [0,\infty)$ it holds
 that
$
\R^d \times \R^d \ni (y_1,y_2) \mapsto \E\bigl[h\bigl(X^{0,y_1}_{s,r},X^{0,y_2}_{s,r}\bigr)\bigr] \in [0,\infty]
$
is measurable.
Moreover, observe that \cref{k20} assures that for all 
$t\in[0,T]$, $s\in[t,T]$, $r\in[s,T]$, $x,y\in\R^d$ it holds that
$X^{0,x}_{t,s} $ and 
$X^{0,y}_{s,r} $ are independent.
This and, e.g., the disintegration-type result in \cite[Lemma~2.2]{HJKNW2018}
show that for all
$t\in[0,T]$, $s\in[t,T]$, $r\in[s,T]$, $x,y\in\R^d$ and all measurable $h\colon \R^{d}\times\R^d\to [0,\infty)$ it holds
 that
$
\E\bigl[ \E\bigr[ h\bigl(X^{0,\tilde{x}}_{s,r},X^{0,\tilde{y}}_{s,r}\bigr) \bigr] |_{\tilde{x}=X_{t,s}^{0,x},\tilde{y}=X_{t,s}^{0,y}} \bigr]
   =\E\!\left[h\bigl(X^{0,x}_{t,r},X^{0,y}_{t,r}\bigr)\right] .
$
Combining 
\cref{k20}, 
\eqref{v01},  
\eqref{r06}, 
\eqref{v09e}, 
\eqref{v07}, 
\eqref{v09c}, 
\eqref{v15}, 
\eqref{v14}, 
\cref{d08} (applied with
$L\defeq c$, $\rho\defeq 2c^3$, 
$\eta\defeq 1$,
$\delta\defeq  3c^2 e^{4c^2T} 
{\size{\tau}}^{\nicefrac{1}{2}}$,
$p\defeq p/\beta$,
$ q\defeq 2$,
$(X_{t,s}^{x,1})_{t\in[0,T],s\in[t,T],x\in\R^d}\defeq (X_{t,s}^{0,x})_{t\in[0,T],s\in[t,T],x\in\R^d}$,
$(X_{t,s}^{x,2})_{t\in[0,T],s\in [t,T],x\in\R^d}\defeq (Y_{t,s}^{0,x})_{t\in[0,T],s\in [t,T],x\in\R^d}$,
$V\defeq b^{p/\beta}\lyaV$,
$\psi \defeq \bigl([0,T]\times\R^d\ni(t,x)\mapsto 
e^{
4c^3 (T-t)/p}(\lyaV(x))^{\nicefrac{2}{p}}\in(0,\infty)\bigr)$, 
$u_1\defeq \uprop_0$, $u_2\defeq \uprop_1$  in the notation of \cref{d08}),
the fact that
$1+cT\leq e^{cT}$, 
the fact that $c\geq1$,
the fact that $\lyaV\geq1$,
the fact that $p\geq 2$,
and
the fact that $p\geq 2\beta$
hence implies that 
for all $t\in[0,T]$, $x\in\R^d$ it holds that
\begin{align}\begin{split}
&|\uprop_0(t,x)-  \uprop_1(t,x)|\\&\leq
4(1+cT)T^{-\nicefrac{1}{2}}
e^{cT+(2c^3\beta/p+c) T}
(b^{p/\beta}\lyaV(x))^{\nicefrac{\beta }{p}}
\bigl[ e^{
4c^3 (T-t)/p}(\lyaV(x))^{\nicefrac{2}{p}}\bigr]^{\nicefrac{1}{2}}
3c^2 e^{4c^2T} 
{\size{\tau}}^{\nicefrac{1}{2}}\\
&\leq 4 e^{cT} T^{-\nicefrac{1}{2}}e^{cT+c^3T+cT}b (\lyaV(x))^{\frac{\beta}{p}}e^{c^3T}(\lyaV(x))^{\frac{1}{p}}3c^2 e^{4c^2T}{\size{\tau}}^{\nicefrac{1}{2}}
\\
&\leq 12bc^2 T^{-\nicefrac{1}{2}}e^{{9}c^3T}
(\lyaV(x))^{\frac{\beta+1}{p}}\size{\tau}^{\nicefrac{1}{2}}.
\end{split}\label{v16}\end{align}
Next observe that \cref{x01} (applied for every $t\in[0,T)$, $n\in\N$ with $L\defeq c$, $\rho\defeq 4 \beta c^3/p$, 
$Y^\theta\defeq Y^{\theta}$,
$\bigV^\theta\defeq \bigV^{\theta}$,
$\smallV\defeq \uprop_1$,
$\varphi\defeq \lyaV^{2\beta  /p}$, 
$N\defeq n$,
$t_0\defeq t$  in the notation of \cref{x01}), 
\eqref{v07}, 
\eqref{v15},
\eqref{v14},
\eqref{r06}, 
\eqref{v04b}, and the fact that $p\geq 2\beta\geq 2 $ assure that for all
$t\in[0,T]$,  $n\in\N$ it holds that
\begin{align}\begin{split}
&\sup_{x\in\R^d}\left[\frac{
\E\!\left[|{\bigV}_{n}^{0}(t,x)-\uprop_1(t,x)|^2\right]}{(\lyaV(x))^{{2\beta}/{p}}}\right]^{\!\nicefrac{1}{2}}\\
&\leq e^{M/2}M^{-n/2}(1+2Tc)^{n-1}e^{2c^3 \beta T/p}
\\& \quad\cdot
\left[2\sup_{s\in[0,T]}\sup_{x\in\R^d}\left[\frac{\max\{|T(\funcF(0))(s,x)|,|g(x)|\}}{(\lyaV(x))^{\beta/p}}\right]+2Tc\sup_{s\in[0,T]}\sup_{x\in\R^d}\left[
\frac{|\uprop_1(s,x)|}{(\lyaV(x))^{\beta/p}}\right]\right]\\
&\leq e^{M/2}M^{-n/2}(1+2Tc)^{n-1}e^{2c^3 \beta T/p}\left[2b+4Tcb e^{cT+2c^3 \beta T/p}\right]\\
&\leq  e^{M/2}M^{-n/2}(1+2Tc)^{n-1}e^{2c^3 \beta T/p}2be^{cT+2c^3 \beta T/p}(1+2Tc )
\leq 
2be^{M/2}M^{-n/2}e^{2ncT} e^{3c^3T}.
\end{split}\end{align}
The triangle inequality, \eqref{v16},
the fact that $c\geq 1$, 
the fact that $\lyaV \geq 1$,
and the fact that $p\geq 2$
hence
 show that for all $t\in[0,T]$, 
$x\in\R^d$,
$n\in\N$ it holds that
\begin{align}\begin{split}
&\left(\E\!\left[|{\bigV}_{n}^{0}(t,x)-\uprop_0(t,x)|^2\right]\right)^{\nicefrac{1}{2}}\leq 
\left(\E\!\left[|{\bigV}_{n}^{0}(t,x)-\uprop_1(t,x)|^2\right]\right)^{\nicefrac{1}{2}}+
| \uprop_1(t,x)-\uprop_0(t,x)|\\
&
\leq 2be^{M/2}M^{-n/2}e^{2ncT} e^{3c^3T}(\lyaV(x))^{\frac{\beta}{p}}+12bc^2 T^{-\nicefrac{1}{2}}e^{{9}c^3T}
(\lyaV(x))^{\frac{\beta+1}{p}}\size{\tau}^{\nicefrac{1}{2}}\\
&\leq 
\left(e^{M/2}e^{2ncT}M^{-n/2}+{\size{\tau}^{\nicefrac{1}{2}}}{T^{-\nicefrac{1}{2}}}\right)12bc^2 e^{{9}c^3T}
(\lyaV(x))^{\frac{\beta+1}{p}}.\end{split}\label{r21}
\end{align}
This and \eqref{v14} establish \cref{k22b}.
The proof of \cref{m01} is thus complete.
\end{proof}

\subsection{Complexity analysis for MLP approximations in fixed space dimensions}
\label{sec:fix_dim}
\newcommand{\cone}{\mathfrak{m}}
\newcommand{\ctwo}{\mathfrak{g}}
\newcommand{\cthree}{\mathfrak{f}}
\begin{theorem}
\label{m01_x}
Let $d,m\in\N$,
$T\in  (0,\infty)$, 
$\cthree, \ctwo, \cone \in [0,\infty)$,
  $\beta,b,c\in [1,\infty)$,  
$p\in[2\beta,\infty)$, 
$\lyaV\in C^2(\R^d, [1,\infty))$,
$\funcG\in C( \R^d,\R)$,
$\mu\in C( \R^{d} ,\R^{d})$, $\sigma=(\sigma_1,\sigma_2,\ldots,\sigma_m)\in C( \R^{d},\R^{d\times m})$,
let $\lVert\cdot \rVert\colon \R^{d}\to[0,\infty)$ be the standard norm on $\R^d$, 
 let
$\smallF\colon [0,T]\times \R^{d}\times \R \to\R$ be 
measurable,
let $\funcF\colon\R^{[0,T]\times\R^d}\to\R^{[0,T]\times\R^d}$ satisfy for all
$t\in [0,T]$,
$x\in \R^d$,
 $v\in \R^{[0,T]\times\R^d}$ that
$
(\funcF(v))(t,x)= f(t,x,v(t,x)),
$
assume  for all  $x,y\in\R^d$, $z\in \R^d\backslash \{0\}$, $t\in[0,T]$,  $v,w\in\R $  that 
\begin{gather}
\label{v05b_x}
\max\bigl\{\tfrac{|( \lyaV'(x))(z)|}{(\lyaV(x))^{(p-1)/p}\|z\|}, \tfrac{|( \lyaV'' (x))(z,z)|}{(\lyaV(x))^{(p-2)/p}\|z\|^2},
\tfrac{c\|x\|+\|\mu(0)\|}{(\lyaV(x))^{1/p}}, 
\tfrac{c\|x\|+[ \sum_{ i=1 }^m \| \sigma_i(0) \|^2 ]^{ 1/2 }} {(\lyaV(x))^{1/p}} \bigr\}
\leq c,
\\
\label{v04b_x}
\max\bigl\{|Tf(t,x,0)|,|g(x)|\bigr\}\leq b(\lyaV(x))^{\beta/p},
\\
\label{v07_x}
\max\bigl\{|g(x)-g(y)|,T|f(t,x,v)-f(t,y,w)|\bigr\}\leq cT|v-w|+ \tfrac{(\lyaV(x)+\lyaV(y))^{\beta/p}
\|x-y\|}{T^{1/2}b^{-1}}
,
\\
\max\bigl\{\|\mu(x)-\mu(y)\|^2,
 \smallsum_{ i=1 }^m \| \sigma_i(x)-\sigma_i(y) \|^2 
\bigr\}
\leq c^2\|x-y\|^2,\label{v03-2_x}
\end{gather}
let $(\Omega,\mathcal{F},\P, (\F_t)_{t\in[0,T]})$ be a filtered probability space which satisfies the usual conditions,
let 
$  \Theta = \bigcup_{ n \in \N }\! \Z^n$,
let $\unif^\theta\colon \Omega\to[0,1]$, $\theta\in \Theta$, be i.i.d.\ random variables,
assume for all $t\in(0,1)$ that $\P(\unif^0\le t)=t$, 
let $\uniform^\theta\colon [0,T]\times \Omega\to [0, T]$, $\theta\in\Theta$, satisfy 
for all $\theta\in \Theta$, $t\in [0,T]$ that 
$\uniform^\theta _t = t+ (T-t)\unif^\theta$,
 let $W^\theta\colon [0,T]\times\Omega \to \R^{m}$, $\theta\in\Theta$, be i.i.d.\ standard $(\F_{t})_{t\in[0,T]}$-Brownian motions, 
assume that
$(\unif^\theta)_{\theta\in\Theta}$ and
$(W^\theta)_{\theta\in\Theta}$ are independent,
for every $N\in\N$,
$\theta\in\Theta$, $x\in\R^d$,
$t\in[0,T]$
let 
$
Y^{N,\theta,x}_{t}=
(Y^{N,\theta,x}_{t,s})_{s\in[t,T]}\colon[t,T]\times\Omega\to\R^d$ 
 satisfy for all $n\in\{0,1,\ldots,N\}$,
 $s\in[\frac{nT}{N},\frac{(n+1)T}{N}]\cap[t,T]$ that $Y_{t,t}^{N,\theta,x}=x$ and
\begin{equation}\label{v01_x}
\begin{split}
&Y_{t,s}^{N,\theta,x} -
Y_{t,\max\left\{t,nT/N\right\}}^{N,\theta,x}\\
&=
\mu\big(Y_{t,\max\left\{t,nT/N\right\}}^{N,\theta,x}\big)\big(s-\max\!\big\{t,\tfrac{nT}{N}\big\}\big)+
\sigma\big(Y_{t,\max\left\{t,nT/N\right\}}^{N,\theta,x}\big)\big(W^{\theta}_{s} -W^{\theta}_{\max\{t,nT/N\}} \big),
\end{split}
\end{equation}
let
$ 
  {\bigV}_{ n,M}^{\theta} \colon [0, T] \times \R^d \times \Omega \to \R
$, $n,M\in\Z$, $\theta\in\Theta$,
satisfy 
for all 
$\theta\in\Theta $,
$n\in \N_0$,
$M\in \N$,
$ t \in [0,T]$, $x\in\R^d $
that 
\begin{equation}\label{v02_x}
\begin{split}
  &{\bigV}_{n,M}^{\theta}(t,x)
=
  \frac{\1_{\N}(n)}{M^n}
 \sum_{i=1}^{M^n} 
      \funcG \big(Y^{M^M,(\theta,0,-i),x}_{t,T}\big)
 \\
&  +
  \sum_{\ell=0}^{n-1} \frac{(T-t)}{M^{n-\ell}}
    \left[\sum_{i=1}^{M^{n-\ell}}
      \bigl(\funcF\bigl({\bigV}_{\ell,M}^{(\theta,\ell,i)}\bigr)-\1_{\N}(\ell)\funcF\bigl( {\bigV}_{\ell-1,M}^{(\theta,-\ell,i)}\bigr)\bigr)
      \bigl(\uniform_t^{(\theta,\ell,i)},Y_{t,\uniform_t^{(\theta,\ell,i)}}^{M^M,(\theta,\ell,i),x}\bigr)
    \right],
\end{split}
\end{equation}
and
let 
$\FE_{n,M} \in \R$, $n,M\in \Z$, satisfy for all $n\in \Z$, $M\in \N$ that
\begin{align}
\FE_{n,M}\le M^n(M^M\cone+\ctwo) \1_{ \N }( n ) +\sum_{\ell=0}^{n-1}\left[M^{n-\ell}
(M^M\cone +\cthree+\FE_{\ell,M}+\FE_{\ell-1,M}
)
\right].\label{c01}
\end{align} 
Then
\begin{enumerate}[(i)]
\item \label{k20_x}
for every $t\in[0,T]$,  $x\in\R^d$, $\theta\in\Theta$
there exists a unique $(\F_{s})_{s\in[t,T]}$-adapted stochastic process
$
X^{ \theta, x }_{ t } = ( X^{ \theta, x }_{ t, s } )_{ s\in[t,T] } \colon
[t,T]\times\Omega\to\R^d$ 
with continuous sample paths
which satisfies that for all $s\in[t,T]$ it holds  $\P$-a.s. that
\begin{align}
X_{t,s}^{\theta,x}= x+\int_{t}^s \mu(X_{t,r}^{\theta,x})\,dr +\int_t^s \sigma(X_{t,r}^{\theta,x})\,dW_r^\theta,
\end{align}
\item \label{k21_x}there exists a unique
measurable
$u\colon [0,T]\times\R^d\to\R$ which satisfies
for all $t\in[0,T]$, $x\in \R^d$ that
$
\bigl(\sup_{s\in[0,T],y\in\R^d} [{|\smallU(s,y)|}{(\lyaV(y))^{-\beta /p}}]\bigr)+\int_{t}^{T}\E\bigl[|f(s,X_{t,s}^{0,x},u(s,X_{t,s}^{0,x}))| \bigr]\,ds
+
\E\bigl[|g(X^{0,x}_{t,T})|\bigr] \allowbreak < \infty$
and
\begin{align}
u(t,x)=\E\!\left[g(X^{0,x}_{t,T})\right]+\int_{t}^{T}\E\!\left[f(s,X_{t,s}^{0,x},u(s,X_{t,s}^{0,x}))\right]ds ,
\end{align}%
\item \label{k22_x}it holds for all $t\in[0,T]$, $x\in\R^d$, 
$n\in\N_0$, $M\in \N$, $\theta \in \Theta$
 that ${\bigV}_{n,M}^{\theta}(t,x)$ is 
 measurable,
\item \label{k22b_x} 
 it holds for all $t\in[0,T]$, $x\in\R^d$, 
$n\in\N_0$, $M\in \N$ that
\begin{align}\small\begin{split}
&\big(\E\big[|{\bigV}_{n,M}^{0}(t,x)-\smallU(t,x)|^2\big]\big)^{\!\nicefrac{1}{2}}
\leq 
\left[
\frac{ \exp( 2 n c T + \frac{ M }{ 2 } ) }{ M^{ \nicefrac{n}{2} } }
+\frac
{1}{M^{M/2}}\right]
12bc^2
|\lyaV(x)|^{\frac{\beta+1}{p}}
\exp(9c^3T)
,\end{split}
\end{align}
\item\label{k23_y} it holds for all $n\in \N$ that
$
\sum_{k=1}^{n+1}\FE_{k,k}
\le 
12(3 \cone +\ctwo+2\cthree) 36^n n^{2n}
$,
and
\item\label{k23_x}there exist $\mathsf{n}
\colon  (0,1]\times \R^d \to \N$ such that
 for all 
 $\epsilon,\gamma\in (0,1]$,
$t\in[0,T]$, $x\in\R^d$
it holds that $
\sup_{n\in[\mathsf{n}(\epsilon,x),\infty)\cap\N}
(\E[|{\bigV}_{n,n}^{0}(t,x)-\smallU(t,x)|^2])^{1/2}<\epsilon
$ and
\begin{align}\begin{split}
\left[\sum_{n=1}^{\mathsf{n}(\epsilon,x)}
\FE_{n,n}\right]\epsilon^{\gamma+4}
&\leq (3 \cone +\ctwo+2\cthree)
\big[\sup\nolimits_{n\in\N}\left[n^{-\gamma n/2}\left(5^n\exp(2ncT)\right)^{\gamma+4}\right]
\big]
\\
&\quad \cdot
\bigl[45bc^2\! \exp(9c^3T)
(\lyaV(x))^{(\beta+1)/p}\bigr]^{\gamma+4}<\infty.
\end{split}\end{align}
\end{enumerate}
\end{theorem}

\begin{proof}[Proof of \cref{m01_x}]\sloppy
\newcommand{\bfdelta}{\boldsymbol{\delta}}
Throughout this proof 
let $\mathsf{n}\colon (0,1]\times \R^d\to [1,\infty]$
satisfy for all $\epsilon\in (0,1]$, $x\in\R^d$
that  
\begin{align}\label{r33}\small\begin{split}
\mathsf{n}(\epsilon,x)= \inf\left(\left\{n\in\N\colon 
\sup_{k\in[n,\infty)\cap\N}\sup_{t\in[0,T]}
\E\Bigl[\bigl|{\bigV}_{k,k}^{0}(t,x)-\smallU(t,x)\bigr|^2\Bigr]<\epsilon^2
\right\}\cup\{\infty\}\right).
\end{split}\end{align}
Observe that \cref{m01} (applied for every $M\in \N$ with $K\defeq M^M$, $(\tau_k)_{k\in \{0,1,\ldots,K\}}\defeq  (\frac{kT}{M^M})_{k\in \{0,1,\ldots,M^M\}}$ in the notation of \cref{m01}) establishes \cref{k20_x,k21_x,k22_x,k22b_x}.
Next note that the fact that 
$\lim_{n\to\infty}(e^{n/2}e^{2ncT}n^{-n/2})=0$ and
 \cref{k22b_x} show that for all $x\in\R^d$, $\epsilon\in(0,1]$ it holds that \begin{align}
\mathsf{n}(\epsilon,x)\in\N.\label{r01}
\end{align} 
Moreover, observe that
\eqref{c01} and, e.g.,
Beck et al.~\cite[Lemma~3.14]{beck2020overcomingelliptic} (applied for every $M\in \N$ with 
$\alpha \defeq (2M^M \cone +\ctwo+\cthree)$, $\beta \defeq (M^M\cone + \cthree)$, $(C_n)_{n\in \N_0}\defeq (\FE_{n,M})_{n\in \N_0}$ in the notation of Beck et al.~\cite[Lemma~3.14]{beck2020overcomingelliptic}) demonstrate that for all $n,M\in \N$ it holds that
\begin{equation}
\FE_{n,M}\le \left[\frac{3M^M \cone +\ctwo+2\cthree}{2} \right](3M)^n.
\end{equation}
This implies that for all $n\in \N$, $k\in \{1,2,\ldots,n\}$ it holds that
\begin{equation}\label{r01b}
\FE_{k,k}\le \tfrac{(3 \cone +\ctwo+2\cthree)(3k^2)^{k}}{2}\le  
\tfrac{(3 \cone +\ctwo+2\cthree)(3(n+1)^2)^{n+1}}{2}
\le  
\tfrac{(3 \cone +\ctwo+2\cthree)(3(2n)^2)^{n+1}}{2}
=
\tfrac{(3 \cone +\ctwo+2\cthree)(12n^2)^{n+1}}{2}.
\end{equation}
The fact that for all $n\in \N$ it holds that $n^3 \le 3^n$ hence ensures that for all $n\in \N$ it holds that
\begin{equation}
\begin{split}
\sum_{k=1}^{n+1}\FE_{k,k}&\le \tfrac{(3 \cone +\ctwo+2\cthree)(n+1)(12n^2)^{n+1}}{2}\le 
(3 \cone +\ctwo+2\cthree)n(12n^2)^{n+1}
\le 
12(3 \cone +\ctwo+2\cthree) 36^n n^{2n}.
\end{split}
\end{equation}
This establishes \cref{k23_y}. 
Next observe that \cref{k22b_x} and \cref{k23_y} prove that for all 
 $\gamma\in (0,1]$,
 $t\in[0,T]$, $x\in\R^d$,
$n\in\N$ it holds that
\begin{align}\begin{split}
&\left[\sum_{k=1}^{n+1}
\FE_{k,k}\right]
\left(\E\Bigl[\bigl|{\bigV}_{n,n}^{0}(t,x)-\smallU(t,x)\bigr|^2\Bigr]\right)^{\frac{4+\gamma}{2}}\\
&\leq 
12(3 \cone +\ctwo+2\cthree) 36^n n^{2n}\left(e^{n/2}e^{2ncT}n^{-n/2}+n^{-n/2}\right)^{\gamma+4}\bigl[12bc^2 e^{ {9}c^3T}
(\lyaV(x))^{(\beta+1)/p}\bigr]^{\gamma+4}\\
&= 
12(3 \cone +\ctwo+2\cthree) 36^nn^{-\gamma n/2} \left(e^{n/2}e^{2ncT}+1\right)^{\gamma+4}\bigl[12bc^2 e^{ {9}c^3T}
(\lyaV(x))^{(\beta+1)/p}\bigr]^{\gamma+4}
\\
&\le 
12(3 \cone +\ctwo+2\cthree) n^{-\gamma n/2} \left(36^{n/4}e^{n/2}e^{2ncT}\right)^{\gamma+4}\bigl[24bc^2 e^{ {9}c^3T}
(\lyaV(x))^{(\beta+1)/p}\bigr]^{\gamma+4}
\\
&\le 
(3 \cone +\ctwo+2\cthree) n^{-\gamma n/2} \left(5^ne^{2ncT}\right)^{\gamma+4}\bigl[45bc^2 e^{ {9}c^3T}
(\lyaV(x))^{(\beta+1)/p}\bigr]^{\gamma+4}.
%
%
%
\end{split}\end{align}
This, \eqref{r33}, and \eqref{r01}
show that for all 
 $\epsilon,\gamma\in (0,1]$,
 $t\in[0,T]$, $x\in\R^d$ 
 with $\mathsf{n}(\epsilon,x)\geq 2$ 
 it holds that
\begin{align}\begin{split}\label{eq:aux1}
&
\left[\sum_{k=1}^{\mathsf{n}(\epsilon,x)}
\FE_{k,k}\right]\epsilon^{4+\gamma}\leq 
\left[\sum_{k=1}^{\mathsf{n}(\epsilon,x)}
\FE_{k,k}\right]
\Bigl(\E\Bigl[\bigl|{\bigV}_{\mathsf{n}(\epsilon,x)-1,\mathsf{n}(\epsilon,x)-1}^{0}(t,x)-\smallU(t,x)\bigr|^2\Bigr]\Bigr)^{\frac{4+\gamma}{2}}\\
&
\leq (3 \cone +\ctwo+2\cthree)\left[\sup_{n\in\N}\left[n^{-\gamma n/2}\left(5^ne^{2ncT}\right)^{\gamma+4}\right]
\right]
\bigl[45bc^2 e^{9c^3T}
(\lyaV(x))^{(\beta+1)/p}\bigr]^{\gamma+4}.
\end{split}\end{align}
Moreover, observe that \eqref{r01b}
demonstrates that for all $\epsilon,\gamma\in (0,1]$,
 $t\in[0,T]$, $x\in\R^d$ with $\mathsf{n}(\epsilon,x)=1$
 it holds
 that
 $
(\sum_{k=1}^{\mathsf{n}(\epsilon,x)}
\FE_{k,k})\epsilon^{4+\gamma}\leq 
\FE_{1,1}\leq 72(3 \cone +\ctwo+2\cthree)
$.
The fact that $b\geq 1$, the fact that $c\geq 1$, the fact that $\lyaV\geq 1$, the fact that $72\leq 45^4$, and \eqref{eq:aux1} therefore prove that 
\begin{align}
\left[\sum_{k=1}^{\mathsf{n}(\epsilon,x)}
\FE_{k,k}\right]\epsilon^{4+\gamma}\leq
(3 \cone +\ctwo+2\cthree)
\left[\sup_{n\in\N}\left[n^{-\gamma n/2}\left(5^ne^{2ncT}\right)^{\gamma+4}\right]
\right]
\bigl[45bc^2 e^{9c^3T}
(\lyaV(x))^{(\beta+1)/p}\bigr]^{\gamma+4}.
\end{align} 
Combining this with \eqref{r33} and \eqref{r01}
establishes \cref{k23_x}. The proof of \cref{m01_x} is thus complete.
\end{proof}

\subsection{Complexity analysis for MLP approximations in variable space dimensions}
\label{sec:var_dim}

\begin{lemma}\label{a01}
Let  
$d\in\N$, 
$a\in[0,\infty)$, let $\lVert\cdot\rVert\colon\R^d\to [0,\infty)$ and
$\varphi\colon \R^d\to \R$
 satisfy for all $x=(x_1,x_2,\ldots, x_d)\in\R^d$ that
$
\varphi(x)=2a+2\|x\|^2=2a+2[\sum_{i=1}^{d}|x_i|^2]
$. Then it holds for all $x,y\in\R^d$ that
$\sqrt{a}+\|x\|\leq |\varphi(x)|^{\nicefrac{1}{2}}$,
$
|(\varphi'(x))(y)|\leq 4|\varphi(x)|^{\nicefrac{1}{2}}
\|y\|$, and $
(\varphi''(x))(y,y)= 4\|y\|^2.$
\end{lemma}

\begin{proof}
[Proof of \cref{a01}]
Observe that the fact that for all 
$v, w \in \R$ it holds that 
$2 v w \leq v^2 + w^2$ ensures that 
for all $s, t \in [0,\infty)$ it holds that
$\sqrt{s} + \sqrt{t} \leq \sqrt{ 2 s + 2 t }$. 
This and the hypothesis that for all $x \in \R^d$
it holds that $\varphi(x)=2a+2\|x\|^2
$ prove that for all 
$x \in \R^d$ it holds that
\begin{equation}\label{eq:a02}
\sqrt{a}+\|x\|\leq
( 2a+2\|x\|^2)^{\nicefrac{1}{2}}=|\varphi(x)|^{\nicefrac{1}{2}}.
\end{equation}
Next note that the hypothesis that for all $x\in\R^d$ it holds that $\varphi(x)=2a+2\|x\|^2$
shows that for all $i, j \in \{ 1, 2, \ldots, d \}$, $x=(x_1,x_2,\ldots, x_d) \in \R^d$ it holds that
$\varphi\in C^2(\R^d,\R)$, $ (\tfrac{\partial}{\partial x_i}\varphi) (x)=4x_i$, and $(\tfrac{\partial^2}{\partial x_i\partial x_j}\varphi) (x)=4\1_{\{i\}} (j)$. 
Combining
this, the Cauchy-Schwarz inequality, and \eqref{eq:a02}
demonstrates that for all $x=(x_1,x_2,\ldots, x_d)$, $y=(y_1,y_2,\ldots,y_d)\in\R^d$
it holds that $
|(\varphi'(x))(y)|=|\sum_{i=1}^{d}4x_iy_i|\leq 4\|x\|\|y\|\leq 4|\varphi(x)|^{\nicefrac{1}{2}}
\|y\|
$
and
$
(\varphi''(x))(y,y)=4\|y\|^2
$. The proof of \cref{a01} is thus complete.
\end{proof}

\newcommand{\crv}{\mathfrak{v}}
\newcommand{\cmu}{\mathfrak{m}}
\newcommand{\cg}{\mathfrak{g}}
\newcommand{\cf}{\mathfrak{f}}
\begin{corollary}
\label{m01c}
Let
$\gamma \in (0,1]$,
$T, c, \crv, \cmu, \cf,\cg \in [0,\infty)$,
$f\in C( \R,\R)$,
for every $d\in\N$ let $u_d\in C^{1,2}([0,T]\times\R^d,\R)$,
$\mu_d=
(\mu_{d,i})_{i\in\{1,2,\ldots,d\}}
\in C( \R^{d} ,\R^{d})$, $\sigma_d=(\sigma_{d,i,j})_{i,j\in \{1,2,\ldots,d\}}\in C( \R^{d},\R^{d\times d})$
satisfy for all $t\in[0,T]$,
 $x=(x_1,x_2,\ldots,x_d)$, $y=(y_1,y_2,\ldots,y_d)\in\R^d$
  that
\begin{equation}\label{eq:Lipschitz_f_cor}
|u_d(t,x)|^2+ 
\max_{ i, j \in \{ 1, 2, \ldots, d \} } ( | \mu_{ d, i }( 0 ) | + | \sigma_{ d, i, j }( 0 ) |)
\leq c \Bigl[d^c+\textstyle\sum\limits_{i=1}^{d}\displaystyle |x_i|^2\Bigr],
\end{equation}
\begin{equation}\label{eq:Lipschitz_coeffs_cor}
|u_d(T,x)-u_d(T,y)|^2+
\textstyle \sum\limits_{i=1}^{d} \displaystyle
 |\mu_{d,i}(x)-\mu_{d,i}(y)|^2+
\textstyle \sum\limits_{i,j=1}^{d}\displaystyle
|\sigma_{d,i,j}(x)-\sigma_{d,i,j}(y)|^2
\leq c^2\Big[
\textstyle \sum\limits_{i=1}^{d} \displaystyle |x_i-y_i|^2
\Big],
\end{equation}
\begin{equation}\label{eq:PDE_cor}
(\tfrac{\partial }{\partial t}u_d)(t,x)+
(\tfrac{ \partial }{ \partial x } u_d )( t, x ) \, \mu_d(x) 
+ \tfrac{1}{2}
\trace\big(\sigma_d(x)[\sigma_d(x)]^* (\Hess_xu)(t,x)\big)=
-f(u_d(t,x)),
\end{equation}
and $|f(x_1)-f(y_1)|\leq c|x_1-y_1|$,
let $(\Omega,\mathcal{F},\P)$ be a probability space,
let 
$  \Theta = \bigcup_{ n \in \N }\! \Z^n$,
let $\unif^\theta\colon \Omega\to[0,1]$, $\theta\in \Theta$, be i.i.d.\ random variables, 
 let $W^{d,\theta}\colon [0,T]\times\Omega \to \R^{d}$,
$d\in\N$, $\theta\in\Theta$, be i.i.d.\ standard Brownian motions,
assume for all $t\in (0,1)$ that $\P(\unif^{0}\le t)=t$,
assume that $( \unif^{ \theta } )_{ \theta \in\Theta}$ and $( W^{ d, \theta } )_{ (d,\theta) \in \N\times\Theta }$ are independent,
for every $d,N\in\N$,
$\theta\in\Theta$, $x\in\R^d$,
$t\in[0,T]$
let 
$
Y^{d,N,\theta,x}_{t}=
(Y^{d,N,\theta,x}_{t,s})_{s\in[t,T]}\colon[t,T]\times\Omega\to\R^d$ 
 satisfy for all $n\in\{0,1,\ldots,N\}$,
 $s\in[\frac{nT}{N},\frac{(n+1)T}{N}]\cap[t,T]$ that $Y_{t,t}^{d,N,\theta,x}=x$ and
\begin{equation}\label{eq:euler_cor}
\begin{split}
&Y_{t,s}^{d,N,\theta,x} -
Y_{t,\max\left\{t,nT/N\right\}}^{d,N,\theta,x}\\
&=
\mu_d\big(Y_{t,\max\left\{t,nT/N\right\}}^{d,N,\theta,x}\big)\big(s-\max\!\big\{t,\tfrac{nT}{N}\big\}\big)+
\sigma_d\big(Y_{t,\max\left\{t,nT/N\right\}}^{d,N,\theta,x}\big)\big(W^{d,\theta}_{s} -W^{d,\theta}_{\max\{t,nT/N\}} \big),
\end{split}
\end{equation}
let
$ 
  {\bigV}_{ n,M}^{d,\theta} \colon [0, T] \times \R^d \times \Omega \to \R
$, 
$d,n,M\in\Z$, $\theta\in\Theta$, satisfy
for all $d,M \in \N$, $n\in \N_0$, $\theta\in\Theta $,
$ t \in [0,T]$, $x\in\R^d $
that 
\begin{align}\label{eq:method_cor}
  &{\bigV}_{n,M}^{d,\theta}(t,x)
=
  \tfrac{ \1_{ \N }( n )}{M^n}
 \textstyle\sum\limits_{i=1}^{M^n} \displaystyle
      u_d\big(T,\sppr^{d,M^M,(\theta,0,-i),x}_{t,T}\big)
 \\
 \nonumber
&  +
  \textstyle\sum\limits_{\ell=0}^{n-1} \displaystyle \left[ \tfrac{(T-t)}{M^{n-\ell}}
   \textstyle\sum\limits_{i=1}^{M^{n-\ell}}\displaystyle
      \big(\smallF\circ {\bigV}_{\ell,M}^{d,(\theta,\ell,i)}-\1_{\N}(\ell)\,\smallF\circ{\bigV}_{\ell-1,M}^{d,(\theta,-\ell,i)}\big)
\big(t+(T-t)\unif^{(\theta,\ell,i)},\sppr_{t,t+(T-t)\unif^{(\theta,\ell,i)}}^{d,M^M,(\theta,\ell,i),x}\big)
    \right]\!,
\end{align}
and let $\FE_{d,n,M}\in \R$, $d,n,M\in  \Z$,
satisfy for all $n\in \Z$, $d,M\in \N$ that  
\begin{equation}\label{eq:comp_effort_cor}
\begin{split}
\FE_{d,n,M}&\le M^n(M^Md\crv+M^M\cmu+\cg) \1_{ \N }( n )\\ &+\sum_{\ell=0}^{n-1}\left[M^{n-\ell}
\Bigl((M^Md+1)\crv+M^M\cmu+2\cf+\FE_{d,\ell,M}+\FE_{d,\ell-1,M}
\Bigr)
\right].
\end{split}
\end{equation} 
Then there exist $\mathfrak{c}\in \R$ and $\mathsf{n}\colon \N\times (0,1] \to \N$ such that
for all $d\in\N$, $\epsilon \in(0,1]$
it holds that $\big(\E\bigl[|\smallU_d(0,0)-{\bigV}_{\mathsf{n}(d,\epsilon),\mathsf{n}(d,\epsilon)}^{d,0}(0,0)|^2\big]\big)^{ 1/2 }\leq \epsilon$ and $\FE_{d,\mathsf{n}(d,\epsilon),\mathsf{n}(d,\epsilon)}
\leq \mathfrak{c}(1+d\crv )
d^{(\gamma+4)(2c+2)}\epsilon^{-(\gamma+4)}$.
\end{corollary}
\begin{proof}[Proof of \cref{m01c}]
Throughout this proof assume without loss of generality that $T > 0$ and let $\mathfrak{c}\in \R$ satisfy that
\begin{equation}
\begin{split}
\mathfrak{c}&= (3 \cmu +\cg+2(\crv+2\cf)+3)\big[
\sup\nolimits_{n\in\N}\big(n^{-\gamma n/2}\left(5^ne^{2n(c+4)T}\right)^{\gamma+4}\big)\big]
\\
&\quad \cdot
\bigl[90( \sqrt{c}+c\sqrt{T}+T|f(0)|)(c+4)^2 e^{9(\sqrt{c}+c+4)^3T}
\bigr]^{\gamma+4}.
\end{split}
\end{equation}
Note that \cref{m01_x} (applied for every $d\in \N$ with $m\defeq d$, $g \defeq \smallU_d(T,\cdot)$, $\mu\defeq \mu_d$, $\sigma\defeq \sigma_d$,
 $\varphi \defeq (\R^d \ni x=(x_1,x_2,\ldots, x_d)\mapsto 2d^{2c+2}+2[\sum_{i=1}^d|x_i|^2]\in [1,\infty))$,
 $f\defeq ([0,T]\times \R^d\times \R \ni (t,x,v)\mapsto f(v)\in \R)$,
 $\beta \defeq 1$, $b\defeq (\sqrt{c}+c\sqrt{T}+T|f(0)|)$, $c\defeq (c+4)$, $p\defeq 2$,
 $\cone \defeq (\cmu+d\crv)$, $\ctwo \defeq \cg$, $\cthree \defeq (\crv+2\cf)$ in the notation of \cref{m01_x}) and \cref{a01} (applied for every $d\in \N$ with $a\defeq d^{2c+2}$ in the notation of \cref{a01}) prove that there exists 
$\mathsf{n}\colon \N\times (0,1] \to \N$ such that for all $d\in \N$, $\epsilon \in (0,1]$ it holds that 
$\big(\E\bigl[|\smallU_d(0,0)-{\bigV}_{\mathsf{n}(d,\epsilon),\mathsf{n}(d,\epsilon)}^{d,0}(0,0)|^2\big]\big)^{ 1/2 }\leq \epsilon$
and
\begin{align}\begin{split}
\FE_{d,\mathsf{n}(d,\epsilon),\mathsf{n}(d,\epsilon)}
\epsilon^{\gamma+4}
&\leq  (3 (\cmu+d\crv) +\cg+2(\crv+2\cf))
\big[\sup\nolimits_{n\in\N}\big(n^{-\gamma n/2}\left(5^ne^{2n(c+4)T}\right)^{\gamma+4}\big)\big]
\\
&\quad \cdot
\bigl[45( \sqrt{c}+c\sqrt{T}+T|f(0)|)(c+4)^2 e^{9(\sqrt{c}+c+4)^3T}
(2d^{2c+2})\bigr]^{\gamma+4}\\
&\leq  (3 \cmu +\cg+2(\crv+2\cf)+3)(1+d\crv)
\big[\sup\nolimits_{n\in\N}\big(n^{-\gamma n/2}\left(5^ne^{2n(c+4)T}\right)^{\gamma+4}\big)\big]
\\
&\quad \cdot
\bigl[90( \sqrt{c}+c\sqrt{T}+T|f(0)|)(c+4)^2 e^{9(\sqrt{c}+c+4)^3T}
\bigr]^{\gamma+4}d^{(\gamma+4)(2c+2)}\\
&= \mathfrak{c}(1+d\crv )d^{(\gamma+4)(2c+2)}.
\end{split}\end{align}
The proof of \cref{m01c} is thus complete.
\end{proof}

\subsubsection*{Acknowledgements}
This work has been funded by the Deutsche Forschungsgemeinschaft (DFG, German Research Foundation) under Germany’s Excellence Strategy EXC 2044-390685587, Mathematics M\"unster:  Dynamics-Geometry-Structure and through the research grant HU1889/6-1.

{
\small
\bibliographystyle{acm}
\bibliography{bibfile}
}

\end{document}